\documentclass[11pt,a4paper]{article}
\usepackage[utf8]{inputenc}
\usepackage[T1]{fontenc}
\usepackage{amsmath}
\usepackage{amsfonts}
\usepackage{amssymb}
\usepackage{amsthm}
\usepackage{bm}
\usepackage{enumitem}
\usepackage{upgreek}
\usepackage{mathrsfs}
\usepackage{color}

\theoremstyle{plain}
\newtheorem{theorem}{Theorem}
\newtheorem{assumption}[theorem]{Assumption}

\newtheorem{proposition}[theorem]{Proposition}

\newtheorem{lemma}[theorem]{Lemma}
\newtheorem{remark}[theorem]{Remark}
\newtheorem{corollary}[theorem]{Corollary}
\newtheorem*{remarque*}{Remark}

\newcommand{\dr}{\mathrm{d}}
\newcommand{\sca}{\mathfrak{s}}
\newcommand{\crule}[3][c]{%
    \par\noindent
    \makebox[\linewidth][#1]{\rule{#2\linewidth}{#3}}}

\usepackage[left=2.5cm,right=2.5cm,top=2cm,bottom=2cm]{geometry}
\usepackage{graphicx}
\usepackage{subcaption}
\usepackage{float}
\usepackage{verbatim}
\usepackage[colorlinks=true,
            linkcolor=black,
            urlcolor=black,
            citecolor=black]{hyperref}
\usepackage{cite}
\usepackage{dsfont}
\usepackage[multiple]{footmisc}

\title{\bf{Fractional diffusion limit for a kinetic Fokker-Planck equation with diffusive boundary conditions in the half-line}}
\author{Loïc Béthencourt\thanks{Sorbonne Université, CNRS, Laboratoire de Probabilités, Statistique et Modélisation, F-75005 Paris, France. Email: loic.bethencourt@sorbonne-universite.fr.}}
\date{}

\begin{document}

\maketitle
\begin{abstract}
We consider a particle with position $(X_t)_{t\geq0}$ living in $\mathbb{R}_+$, whose velocity $(V_t)_{t\geq0}$ is a positive recurrent diffusion with heavy-tailed invariant distribution when the particle lives in $(0,\infty)$. When it hits the boundary $x=0$, the particle restarts with a random strictly positive velocity. We show that the properly rescaled position process converges weakly to a stable process reflected on its infimum. From a P.D.E. point of view, the time-marginals of $(X_t, V_t)_{t\geq0}$ solve a kinetic Fokker-Planck equation on $(0,\infty)\times\mathbb{R}_+ \times \mathbb{R}$ with diffusive boundary conditions. Properly rescaled, the space-marginal converges to the solution of some fractional heat equation on $(0,\infty)\times\mathbb{R}_+$.
\end{abstract}

\crule{.75}{0.5pt}


\section{Introduction}

In the last two decades, many mathematical works showed how to derive \textit{anomalous diffusion limit} results, also called \textit{fractional diffusion limits}, from different kinetic equations with heavy-tailed equilibria. In short, these types of results state that the properly rescaled density of the position of a particle subject to some kinetic equation, is asymptotically non-gaussian. A case of particular interest is when  the scaling limit of the position of the particle is a stable process, of which the time-marginals satisfies the fractional heat equation.

\medskip
Mellet, Mischler and Mouhot \cite{mellet2011fractional} showed a fractional diffusion limit result for a linearized Boltzmann equation with heavy-tailed invariant distribution using Fourier transform arguments, a result which was improved by Mellet \cite{mellet2010fractional} using a moments method. The proofs rely entirely on analytic tools. A similar result was shown in Jara, Komorowski and Olla \cite{ollamilton}, although it is derived from an $\alpha$-stable central limit theorem for additive functional of Markov chains, i.e. from a probabilistic result.

\medskip
Anomalous diffusion limits also occur for kinetic equations with degenerate collision frequency, see Ben Abdallah, Mellet and Puel \cite{ben2011anomalous}, and in transport of particles in plasma, see Cesbron, Mellet and Trivisa \cite{cesbron2012anomalous}. In \cite{bouin2020quantitative}, Bouin and Mouhot propose a unified approach to derive fractional diffusion limits from several linear collisional kinetic equations.

\medskip
In \cite{lebeau2019diffusion}, Lebeau and Puel showed a fractional diffusion limit result for a one-dimensional Fokker-Planck equation with heavy-tailed equilibria. From a probabilistic point of view, the corresponding toy model is a one-dimensional particle whose velocity is subject to a restoring force $\mathrm{F}$ and random shocks modeled by a Brownian motion $(B_t)_{t\geq0}$, which leads to the following stochastic differential equation:
\begin{equation}\label{eq_non_reflected}
    V_t = v_0 + \int_0^t \mathrm{F}(V_s)\dr s + B_t, \quad X_t = x_0 + \int_0^t V_s \dr s,
\end{equation}

\noindent
where $(x_0, v_0)\in\mathbb{R}^2$, $(V_t)_{t\geq0}$ and $(X_t)_{t\geq0}$ are the velocity and position processes of the particle. The force at stake is $\mathrm{F}(v) = -\frac{\beta}{2}\frac{v}{1+v^2}$, where $\beta\in(1,5)\setminus\{2, 3, 4\}$ leading to an invariant measure which behaves as $(1+|v|)^{-\beta}$ as $|v|\to\infty$. Their result states that in this case, the properly rescaled position process resembles a stable process in large time. When $\beta > 5$, Nasreddine and Puel \cite{nasreddine2015diffusion} established that $(\epsilon^{1/2}X_{t/\epsilon})_{t\geq0}$ resembles a Brownian motion as $\epsilon\to0$, corresponding to a classical diffusion limit type theorem. Then Cattiaux, Nasreddine and Puel \cite{cattiaux2019diffusion} later showed that in the critical case $\beta = 5$, the same result holds up to a logarithmic correction term.

\medskip
This phenomenon was actually observed by physicists who discovered experimentally that atoms cooled by a laser diffuse anomalously, see for instance Castin, Dalibard and Cohen-Tannoudji \cite{castin1991limits}, Sagi, Brook, Almog and
Davidson \cite{sagi2012observation} and Marksteiner, Ellinger and Zoller \cite{marksteiner1996anomalous}.  A theoretical study (see Barkai, Aghion and Kessler \cite{barkai2014area}) modeling the motion of atoms precisely by \eqref{eq_non_reflected} proved with quite a high level of rigor the observed phenomenons.

\medskip
Then, using probabilistic techniques, Fournier and Tardif \cite{fournier2018one} treated all cases of \eqref{eq_non_reflected} (i.e. $\beta>0$) for a slightly larger class of symmetric forces. When $\beta \geq 5$, the limiting distribution is Gaussian whereas when $\beta\in(1,5)$, they show that the following convergence in finite dimensional distributions holds, for any initial condition $v_0\in\mathbb{R}$:
\begin{equation}\label{result_fournier}
    \left(\epsilon^{1/\alpha}X_{t/\epsilon}\right)_{t\geq0} \overset{f.d.}{\longrightarrow} \left(\sigma_{\alpha}Z_t^{\alpha}\right)_{t\geq0}\qquad\hbox{as $\epsilon\to0$},
\end{equation}

\noindent
where $(Z_t^{\alpha})_{t\geq0}$ is a symmetric $\alpha$-stable process with $\alpha = (\beta + 1) / 3$, and $\sigma_{\alpha}$ is some positive diffusive constant. Naturally, they recover the result of \cite{nasreddine2015diffusion, cattiaux2019diffusion, lebeau2019diffusion} and even go beyond, treating the case $\beta\in(0,1)$ which was new. In this regime, the velocity is null recurrent, and the rescaled process was shown to converge to a symmetric Bessel process of dimension $\delta\in(0,2)$. Then the position process naturally converges to an integrated symmetric Bessel process, which is no longer Markov. Their proof heavily relies on Feller's representation of diffusion processes through their scale functions and speed measures, enabling them to treat all cases at once, even the critical cases $\beta = 1, 2, 5$. This method was generalized in \cite{bethencourt2021stable} and \eqref{result_fournier} was shown to be a special case of an $\alpha$-stable central limit theorem for additive functional of one-dimensional Markov processes. In a companion paper, Fournier and Tardif \cite{fournier_dimension} also showed that these results hold in any dimension and the proofs are much more involved than in dimension 1.

\medskip
As it is very natural in kinetic theory to consider gas particles interacting with a surface in thermodynamical equilibrium, corresponding to the case of diffusive boundary conditions, we propose in this article to study a version of the process $(X_t)_{t\geq0}$ living in $\mathbb{R}_+$ and reflected diffusively through its velocity when the particles hits $0$. In other words, we consider the case of a particle governed by \eqref{eq_non_reflected}, which interacts with a wall located at $x=0$. When the particle hits the boundary, it reemerges from the wall with a random velocity distributed according to some probability measure.

\medskip
The aim of this paper is to study the scaling limit of such a particle. More precisely, we will show that in the Lévy regime ($\beta \in(1,5)$), the rescaled position process converges in law to a stable process reflected on its infimum. This result should be related to the recent articles of Cesbron, Mellet and Puel \cite{cesbron2020fractional, cesbron2021fractional} and Cesbron \cite{cs2020refelctive}, which extend the results obtained in \cite{ben2011anomalous, mellet2011fractional, mellet2010fractional}. They deal with the kinetic scattering equation describing particles living in the half-space or in bounded domains, with specular and / or diffusive boundary conditions. They obtain as a scaling limit a fractional heat equation with some boundary conditions depending on the original boundary conditions. This differs from the normal diffusive case where one would always obtain the classical heat equation with Neumann boundary conditions.  Here, the kinetic equation at stake is different, but we obtain a similar limiting equation as in \cite{cesbron2020fractional, cesbron2021fractional, cs2020refelctive} with however some different behaviour at the boundary. We refer to the the PDE section below for more details. In \cite{cesbron2020fractional, cs2020refelctive}, the limiting PDE is clearly identified, but they have a uniqueness issue due to the weakness of their solutions. This problem was solved in dimension 1 in \cite{cesbron2021fractional}. We emphasize that we have no such issue and that the limiting process is quite explicit. We also point out that their result only holds for $\alpha\in(1,2)$.

\medskip
We should also mention the works of Komorowski, Olla, Ryzhik \cite{komo_olla} and Bogdan, Komorowski, Marino \cite{bogdan2022anomalous}, which are, to our knowledge, the only probabilistic results of fractional diffusion limits with boundary interactions. They both study a scattering equation (linear Boltzmann), as in \cite{cesbron2020fractional, cesbron2021fractional, cs2020refelctive}, but with a mixed reflective / transmissive / absorbing boundary conditions. We refer to the comments section below, where these papers are further discussed.

\medskip
Let us now introduce more fomally the model studied. We will denote by $\mathbb{N} = \{1,2,3, \cdots\}$ the set of positive integers. Let $v_0 > 0$, $(B_t)_{t\geq0}$ be a Brownian motion and $(M_n)_{n\in\mathbb{N}}$ be a sequence of i.i.d. random variables whose law $\mu$ is supported in $(0,\infty)$. Everything is assumed to be independent. The object at stake in this article is the strong Markov process $(\bm{X}_t, \bm{V}_t)_{t\geq0}$ valued in $E = ((0,\infty)\times\mathbb{R}) \cup (\{0\} \times(0,\infty))$ and defined by the following stochastic differential equation
\begin{equation}\label{reflected_equation}
    \left\lbrace
        \begin{aligned}
            \bm{X}_t & = x_0 + \int_0^t \bm{V}_s \dr s, \\
            \bm{V}_t & = v_0 + \int_0^t \mathrm{F}(\bm{V}_s)\dr s + B_t + \sum_{n\in\mathbb{N}}\left(M_n - \bm{V}_{\tau_n-}\right)\bm{1}_{\{\tau_n \leq t\}}, \\
            \bm{\tau}_1 & = \inf\{t> 0, \: \bm{X}_t = 0\} \quad \text{and} \quad \bm{\tau}_{n+1} = \inf\{t> \bm{\tau}_n, \: \bm{X}_t = 0\},
        \end{aligned}
    \right.
\end{equation}

\noindent
where $\mathrm{F}$ fulfills Assumption \ref{assump_fournier} below and $(x_0, v_0)\in E$. This equation is well-posed and we refer to Subsection \ref{subsection_well_pd} for more details. It describes the motion of a particle evolving in $[0,\infty)$ and being reflected when it hits $0$. More precisely, the velocity $(\bm{V}_t)_{t\geq0}$ and position $(\bm{X}_t)_{t\geq0}$ processes are governed by \eqref{eq_non_reflected} when $\bm{X}_t > 0$, and the particle is reflected through the velocity when it hits the boundary, i.e. when $\bm{X}_t = 0$. Note that $t \mapsto \bm{V}_t$ is a.s. càdlàg and that the jumps only occur when the particle hits the boundary, i.e. when $t = \bm{\tau}_n$ for some $n\in\mathbb{N}$. In this case, the value of the velocity after the jump is
\begin{equation*}
    \bm{V}_{\bm{\tau}_n} = \bm{V}_{\bm{\tau}_n-} + \Delta \bm{V}_{\bm{\tau}_n} = M_n
\end{equation*}

\noindent
As for every $n\in\mathbb{N}$, $M_n > 0$ a.s., the particle reaches $(0,\infty)$ instantaneously after hitting the boundary and thus spend a strictly positive amount of time in $(0,\infty)$. Hence the zeros of $(\bm{X}_t)_{t\geq0}$ are countable and the successive hitting times $(\bm{\tau}_n)_{n\in\mathbb{N}}$ are well defined. Finally, we point out that since a solution $(X_t, V_t)_{t\geq0}$ of \eqref{eq_non_reflected} reaches $\{0\}\times\mathbb{R}$ with a necessarily non-positive velocity, the process $(\bm{X}_{t-}, \bm{V}_{t-})_{t\geq0}$ is valued in $((0,\infty)\times\mathbb{R}) \cup (\{0\} \times(-\infty,0])$.

\medskip
Let us mention some related works, dealing with Langevin-type models reflected in the half-space with diffusive-reflective boundary conditions. In \cite{jacob2012langevin, jacob2013langevin}, Jacob studies the classical Langevin process, namely the process $(B_t, \int_0^t B_s \dr s)_{t\geq0}$, reflected at a partially elastic boundary. In \cite{jabir2019stable}, Jabir and Profeta study a stable Langevin process with diffusive-reflective boundary conditions. Roughly, their work treats the well-posedness of the equations and wether or not, the obtained process is conservative, observing some phase transitions. The question of the existence of a scaling limit does not make sense since the processes are intrinsically self-similar (at least in the purely reflective case).

\medskip
In the rest of the article, the process $(X_t, V_t)_{t\geq0}$ will always be defined by \eqref{eq_non_reflected}, and will be refered to as the ''free process``. On the other hand, $(\bm{X}_t, \bm{V}_t)_{t\geq0}$ will always refer to the process defined by \eqref{reflected_equation} and will be called the ''reflected process`` or ''reflected process with diffusive boundary conditions``. Let us now be a little more precise on the assumptions we will suppose. Regarding the force field $\mathrm{F}$, we assume the following.

\begin{assumption}\label{assump_fournier}
The Lipschitz and bounded force $\mathrm{F}:\mathbb{R}\to\mathbb{R}$ is such that $\mathrm{F} = \frac{\beta}{2}\frac{\Theta '}{\Theta}$ where $\beta \in (1,5)$ and $\Theta:\mathbb{R}\to(0,\infty)$ is a $C^1$ even function satisfying $\lim_{v\to\pm\infty}|v|\Theta(v) = 1$.
\end{assumption}

This assumption is a little bit stronger than the one in \cite{fournier2018one}, and thus \eqref{result_fournier} holds for the free process. The typical force we have in mind is $\mathrm{F}(v) = -\frac{\beta}{2}\frac{v}{1+v^2}$ studied in \cite{lebeau2019diffusion, nasreddine2015diffusion, cattiaux2019diffusion}, corresponding to the function $\Theta(v) = (1 + v^2)^{-1/2}$. The value of the diffusive constant $\sigma_{\alpha}$ from \eqref{result_fournier} is given by
\begin{equation}\label{sigma}
 \sigma_\alpha = \frac{3^{1- 2\alpha}2^{\alpha - 1}\pi c_{\beta}}{\Gamma^2(\alpha)\sin(\pi\alpha/2)}, \quad \text{where} \quad c_{\beta} = \left(\int_{\mathbb{R}}\Theta^{\beta}(v)\dr v\right)^{-1}
\end{equation}

\noindent
Regarding the probability measure $\mu$ governing the reflection, we will have two different assumptions.

\begin{assumption}\label{assump_measure}
\noindent
\begin{enumerate}[label=(\roman*)]
 \item There exists $\eta > 0$ such that $\mu$ has a moment of order $(\beta + 1) / 2 + \eta$.
 \item There exists $\eta > 0$ such that $\mu$ has a moment of order $(\beta + 1)(\beta + 2) / 6 + \eta$.
\end{enumerate}
\end{assumption}

Since $\beta > 1$, Assumption \ref{assump_measure}-\textit{(ii)} is obviously stronger than Assumption \ref{assump_measure}-\textit{(i)}. These assumptions are satisfied by many probability measures of interest such as the Gaussian density, every sub-exponential distributions, and many heavy-tailed distributions.

\medskip
During the whole paper, Assumption \ref{assump_fournier} is always in force and Assumption \ref{assump_measure} will be mentioned when necessary.

\medskip
For a family of processes $((Y_t^{\epsilon})_{t\geq0})_{\epsilon>0}$, we say that $(Y_t^{\epsilon})_{t\geq0} \overset{f.d.}{\longrightarrow}(Y_t^{0})_{t\geq0}$ as $\epsilon\to0$ if for all $n\geq1$, for all $t_1, \dots, t_n \geq 0$, the vector $(Y_{t_i}^{\epsilon})_{1\leq i \leq n}$ converges in law to $(Y_{t_i}^{0})_{1\leq i \leq n}$ in $\mathbb{R}^n$. Most of the convergence results obtained are actually stronger than convergence in finite dimensional distributions, i.e. we obtain convergence in law of processes in the space of càdlàg functions. As the usual Skorokhod topology is not suited for convergence of continuous processes to a discontinuous process, we will instead use a weaker topology, namely the $\bm{\mathrm{M}}_1$-topology, and we refer to Section \ref{section_conv_free} for more details. The following theorem is the main result of this paper.
\begin{theorem}\label{main_theorem}
 Grant Assumptions \ref{assump_fournier} and \ref{assump_measure}-\textit{(i)}, and let $(\bm{X}_t, \bm{V}_t)_{t\geq0}$ be a solution of \eqref{reflected_equation} starting at $(0,v_0)$ with $v_0 > 0$. Let $(Z_t^{\alpha})_{t\geq0}$ be a symmetric stable process with $\alpha = (\beta + 1) / 3$ and such that $\mathbb{E}[e^{i\xi Z_t^{\alpha}}] = \exp(-t\sigma_\alpha |\xi|^{\alpha})$ where $\sigma_{\alpha}$ is defined in \eqref{sigma}. Let $(R_t^{\alpha})_{t\geq0}$ be the stable process reflected on its infimum, i.e. $R_t^{\alpha} = Z_t^{\alpha} - \inf_{s\in[0,t]}Z_s^{\alpha}$. Then we have
 \[
(\epsilon^{1/\alpha}\bm{X}_{t/\epsilon})_{t\geq0}\overset{f.d.}{\longrightarrow}(R_t^{\alpha})_{t\geq0}\qquad \text{as }\epsilon\to0.
\]

\noindent
Moreover, if we grant Assumption \ref{assump_measure}-\textit{(ii)}, this convergence in law holds in the space of càdlàg functions endowed with the $\bm{\mathrm{M}}_1$-topology.
\end{theorem}

It is likely that this theorem could be extended for any initial condition $(x_0, v_0) \in E$ with some small adjustments but for the sake of simplicity we will only consider the case $x_0 = 0$ and $v_0 > 0$.

\medskip
Let us try to explain informally why the limiting process should be the stable process reflected on its infimum. Remember that when $\bm{X}_{t} >0$, the process is governed by \eqref{eq_non_reflected}, and therefore, we should expect the limit process to behave like a stable process in $(0,\infty)$. Only the behavior at the boundary is to be identified. When $\bm{X}_{t}$ reaches the boundary with a very high speed, corresponding to a jump in the limit, it is suddenly reflected and slowed down, as the new velocity is distributed according to $\mu$; and the process somehow restarts afresh. As a consequence, we should expect the following behavior for the limiting process: when it tries to jump across the boundary, the jump is ''cut`` and the process restarts from $0$. This is exactly the behavior of $(R_t^{\alpha})_{t\geq0}$: when $R_t^{\alpha} > 0$, it behaves as $Z_t^{\alpha}$ and when $Z_t^{\alpha}$ jumps below its past infimum, the latter one is immediately ''updated``, corresponding to $R_t^{\alpha}$ trying to cross the boundary and being set to $0$.

\medskip
We emphasize that $(R_t^{\alpha})_{t\geq0}$ is a Markov process, see \cite[Chapter 6, Proposition 1]{bertoin1996levy}. Let us now point out an interesting phenomena: the limiting process really depends on the way we reflect the inital process. For instance, let us consider \eqref{reflected_equation} with a specular boundary condition, i.e. when the process hits the boundary, it is reflected with the same incoming velocity, which is flipped. Then it is easy to see that, since the force field is symmetric, $(|X_t|, \text{sgn}(X_t)V_t)_{t\geq0}$ is a solution of the corresponding reflected equation, where $(X_t, V_t)_{t\geq0}$ is a solution of \eqref{eq_non_reflected}. Then by \eqref{result_fournier}, the limiting process is $(|Z_t^{\alpha}|)_{t\geq0}$ whose behavior at the boundary is different from $(R_t^{\alpha})_{t\geq0}$: when it tries to cross the boundary, the process is moved back in $\mathbb{R}_+$ by a mirror reflection. Note that $(|Z_t^{\alpha}|)_{t\geq0}$ is a Markov process only in the symmetric case, so it is not clear what happens in the disymmetric case.

\medskip
Unlike the Brownian motion, there is no unique way to reflect a stable process. Indeed, by a famous result of Paul Lévy, it is well known that $(|B_t|)_{t\geq0}\overset{d}{=}(B_t - \inf_{s\in[0,t]}B_s)_{t\geq0}$. This is to be related with the fact that unlike the Laplacian, the fractional Laplacian is a non-local operator. We believe that, with a little bit of work, we could extend Theorem \ref{main_theorem} to the diffusive regime, i.e. when $\beta \geq 5$, and the rescaled process would converge to a reflected Brownian motion.

\medskip
On our way to establish Theorem \ref{main_theorem}, we will encouter a singular equation which describes the motion of a particle reflected at a completely inelastic boundary, which is very close to the equation studied by Bertoin in \cite{bertoin2008second}. We will study a solution of the following stochastic differential equation.
\begin{equation}\label{reflected_equation_zero}
    \left\lbrace
        \begin{aligned}
            \mathbb{X}_t & = \int_0^t \mathbb{V}_s \dr s, \\
            \mathbb{V}_t & = v_0 + \int_0^t \mathrm{F}(\mathbb{V}_s)\dr s + B_t - \sum_{0 <s\leq t}\mathbb{V}_{s-}\bm{1}_{\{\mathbb{X}_s = 0\}},
        \end{aligned}
    \right.
\end{equation}

\noindent
where $v_0 >0$. Let $a > 0$ and consider a solution $(\bm{X}_t^a, \bm{V}_t^a)_{t\geq0}$ of \eqref{reflected_equation} starting at $(0, v_0)$ for the particular choice $\mu = \delta_a$. Then, informally, $(\bm{X}_t^a, \bm{V}_t^a)_{t\geq0}$ should tend to $(\mathbb{X}_t, \mathbb{V}_t)_{t\geq0}$ as we let $a\to0$. Since for every $a>0$, the rescaled process $(\epsilon^{1/\alpha}\bm{X}_{t/\epsilon}^a)_{t\geq0}$ converges in law to $(R_t^{\alpha})_{t\geq0}$, it should be expected that $(\mathbb{X}_t)_{t\geq0}$ has the same scaling limit. We will see that it is indeed the case, see Theorem \ref{cv_free_ref} below. While it is clear that we can construct a solution to \eqref{reflected_equation}, it is non-trivial that \eqref{reflected_equation_zero} posesses a solution and let us quickly explain why.

\medskip
Consider a solution $(X_t, V_t)_{t\geq0}$ of \eqref{eq_non_reflected} starting at $(0,0)$, then one can easily see that $0$ is an accumulation point of the instants at which $X_t = 0$. Now if we consider a solution $(\mathbb{X}_t, \mathbb{V}_t)_{t\geq0}$ of \eqref{reflected_equation_zero}, its energy is fully absorbed when the particle hits the boundary i.e. when $\mathbb{X}_t = 0$, the value of the velocity is $\mathbb{V}_t = \mathbb{V}_{t-} - \mathbb{V}_{t-} = 0$. Hence it is not clear at all whether the particle will ever reemerge in $(0,\infty)$ or will remain stuck at $0$. As we will see, it turns out that \eqref{reflected_equation_zero} admits a conservative solution, which never gets stuck at $0$.

\medskip
In \cite{bertoin2007reflecting, bertoin2008second}, Bertoin studies equation \eqref{reflected_equation_zero} in the special case $\mathrm{F} =0$ and he establishes the existence and uniqueness (in some sense). We will not be interested in studying the uniqueness of \eqref{reflected_equation_zero} in our case but we will use the construction developped in  \cite{bertoin2007reflecting} and \cite{bertoin2008second}.

\subsection*{Comments and comparison with the litterature}
At this point, we should discuss the works of Komorowski, Olla, Ryzhik \cite{komo_olla} and Bogdan, Komorowski, Marino \cite{bogdan2022anomalous}. In \cite{komo_olla}, they study a scattering equation with a reflective / transmissive / absorbing boundary conditions. When the particle hits the boundary $x=0$, it is either reflected (by flipping the incoming velocity), either unchanged (transmitted), or killed. Their model is such that the rescaled particle always satisfies $\bm{X}_0^\epsilon = x$ for some $x > 0$, and they find the following limiting process. Consider a symmetric $\alpha$-stable process $(Z_t^\alpha)_{t\geq0}$ started at $x$ with $\alpha\in(1,2)$, as well as its successive crossing times $(\sigma_n)_{n\in\mathbb{N}}$ at the level $0$. Consider also i.i.d. random variables $(\xi_n)_{n\in\mathbb{N}}$ such that $\mathbb{P}(\xi_1 = 1) = p_+$, $\mathbb{P}(\xi_1 = -1) = p_-$ and $\mathbb{P}(\xi_1 = 0) = p_0$, with $p_+, p_-, p_0 > 0$ such that $p_+ + p_- + p_0 = 1$. The reflected stable process $(R_t^\alpha)_{t\geq0}$ is then defined as follows: $R_t^\alpha = Z_t^\alpha$ on $[0,\sigma_1)$, and for any $n\in\mathbb{N}$ and any $t\in[\sigma_{n}, \sigma_{n+1})$,
\[
 R_t^\alpha = \Big(\prod_{k=1}^n \xi_k\Big)Z_t^\alpha.
\]
In other words, each time $(Z_t^\alpha)_{t\geq0}$ crosses the boundary, the trajectory of $(R_t^\alpha)_{t\geq0}$ is transmitted with probability $p_+$, reflected with probability $p_-$ or absorbed with probability $p_0$. As it is well-known, a symmetric stable process with index $\alpha > 1$ eventually touches $0$, but crosses the boundary infinitely many times before doing so, see for instance \cite[Chapter VIII, Proposition 8]{bertoin1996levy}. Therefore $\sigma_\infty = \lim_{n\to\infty}\sigma_n$ is a.s. finite but since $p_0 > 0$, the reflected process is a.s. absorbed before $\sigma_\infty$ and $(R_t^\alpha)_{t\geq0}$ is naturally set to $0$ after $\sigma_\infty$.

\medskip
In \cite{bogdan2022anomalous}, they study the very same model, but the probability $p_0^\epsilon$ for the kinetic process to be absorbed when it hits the boundary vanishes as $\epsilon\to0$ and behaves as $1/|\log\epsilon|$. The limiting process obtained is the same as above with $p_0 = 0$ for $t < \sigma_\infty$, and the process is absorbed at $0$ at $t = \sigma_\infty$.

\medskip
Observe that in both cases, the limiting process is killed before (or precisely when) hitting the boundary. In the present paper, we thus have a substantial additional difficulty which is to characterize the limiting process when starting from $0$. Indeed, a symmetric stable process started from $0$ touches $0$ infinitely many times immediately after (when $\alpha > 1$), so that the behavior of the limiting process started from $0$ is not trivially defined. We cannot avoid this difficulty as the kinetic process is restarted with some small velocity.

\medskip
We emphasize that these kinds of results are very recent and were mostly treated from an analytic point of view, see \cite{cesbron2020fractional, cesbron2021fractional, cs2020refelctive}. With the papers \cite{komo_olla, bogdan2022anomalous}, our work seems to be the only probabilistic study of fractional diffusion limit with boundary conditions. Our proof borrows different tools from stochastic analysis such as excursion theory, Wiener-Hopf factorization and a bit of enlargment of filtrations.

\medskip
Regarding Assumption \ref{assump_measure}-\textit{(i)}, we believe that it is near optimality. As we will see, it appears naturally in the proofs at several places. Moreover, we believe the limiting process should differ at criticality, i.e. when $v\mapsto\mu((v,\infty))$ is regularly varying with index $-(\beta + 1) / 2$ as $v\to\infty$. Roughly, the restarting velocities are no longer negligeable and, the limiting process directly enters the domain after hitting $0$ by a jump with ''law'' $x^{-\alpha}\dr x$. Observe that this is the reason why they \cite{cesbron2020fractional, cesbron2021fractional} have to assume $\alpha > 1$. We think that, at criticality, we might obtain the same limiting distribution as in \cite{cesbron2020fractional, cesbron2021fractional}.

\medskip
Assumption \ref{assump_measure}-\textit{(ii)} is technical and we believe the convergence in the $\bm{\mathrm{M}}_1$-topology should also hold under Assumption \ref{assump_measure}-\textit{(i)}.

\medskip
The present results might be extended, through the same line of proof, to integrated powers of the velocity, i.e. to $\bm{X}_t = \int_0^t\text{sgn}(\bm{V}_s)|\bm{V}_s|^{\gamma}\dr s$ and some $\gamma >0$. At least, we already know from \cite[Theorem 5]{bethencourt2021stable} that an $\alpha$-stable central limit theorems holds for the free process.

\medskip
Finally, it would be interesting to study what happens in higher dimensions (using the results of \cite{fournier_dimension}), as well as a more complete model, with both diffusive and specular boundary conditions: when the particle hits the boundary, it is reflected diffusively with probability $p\in(0,1)$ and specularly with probability $1-p$.

\subsection*{Informal PDE description of the result}
Let us expose our result with a kinetic theory point of view, making a bridge with the P.D.E. papers \cite{nasreddine2015diffusion, cattiaux2019diffusion, lebeau2019diffusion, cesbron2020fractional, cesbron2021fractional, cs2020refelctive}. Let us denote by $f_t$ the law of the process $(\bm{X}_t, \bm{V}_t)_{t\geq0}$ solution to \eqref{reflected_equation} starting at $(0,v_0)$ with $v_0 > 0$, i.e. $f_t(\dr x, \dr v) = \mathbb{P}(\bm{X}_t \in \dr x, \bm{V}_t \in \dr v)$ which is a probability measure on $\mathbb{R}_+ \times \mathbb{R}$. Then, see Proposition \ref{prop_32} and Remark \ref{remark_33}, $(f_t)_{t\geq0}$ is a weak solution of the kinetic Fokker-Planck equation with diffusive boundary conditions
\begin{equation*}
    \left\lbrace
        \begin{aligned}
            \partial_t & f_t + v\partial_x f_t = \frac{1}{2}\partial_v^2 f_t - \partial_v[\mathrm{F}(v)f_t]  \quad &\text{for} \:\: (t,x,v)\in(0,\infty)^2 \times \mathbb{R} \\
            v &  f_t(0,v) = -\mu(v)\int_{(-\infty, 0)}w f_t(0,w) \dr w & \text{for} \:\: (t,v)\in(0,\infty)^2 \\
            \:\:f_0&  = \delta_{(0,v_0)} 
        \end{aligned}
    \right.
\end{equation*}
where we assume for simplicity that $\mu(dv) = \mu(v) \dr v$.

\smallskip
We now set $\rho_t(\dr x) = \mathbb{P}(R_t^\alpha \in \dr x)$ where $(R_t^\alpha)_{t\geq0}$ is the limiting process defined in Theorem \ref{main_theorem}. Then, see Proposition \ref{yyy} and Remark \ref{ggg}, $(\rho_t)_{t\geq0}$ is a weak solution of
\begin{equation*}
    \left\lbrace
        \begin{aligned}
            &\partial_t \rho_t(x) = \frac{\sigma_\alpha}{2}\int_{\mathbb{R}}\frac{\rho_t(x-z) \bm{1}_{\{x > z\}} - \rho_t(x) + z\partial_x\rho_t(x)\bm{1}_{\{|z|< x\}}}{|z|^{1+\alpha}}\dr z \;\; &\text{for} \:\: (t,x)\in(0,\infty)^2, \\
            &\int_{0}^\infty \rho_t(x) \dr x = 1 \;\; &\text{for} \:\: t\in(0,\infty), \\
            &\:\rho_0(\cdot)  = \delta_{0}. 
        \end{aligned}
    \right.
\end{equation*}
We believe that the above equation might be well-posed, just as for the heat equation $\partial_t \rho_t(x) = \partial_{xx}\rho_t(x)$ on $(0,\infty)^2$ where the Neumann boundary condition $\partial_x\rho_t(0) = 0$ can classicaly be replaced by the constraint $\int_{0}^\infty \rho_t(x) \dr x = 1$. The following statement immediately follows from Theorem \ref{main_theorem}
\begin{corollary}
 Grant Assumptions \ref{assump_fournier} and \ref{assump_measure}-\textit{(i)}. Let $g_t(\dr x) = \int_{v\in\mathbb{R}} f_t(\dr x, \dr v) = \mathbb{P}(\bm{X}_{t} \in \dr x)$ and set with an abuse of notation $g_t^{\epsilon}(x) = \epsilon^{-1/\alpha}g_{t/\epsilon}(\epsilon^{-1/\alpha}x)$, that is $g_t^\epsilon$ is the pushforward measure of $g_{t/\epsilon}$ by the map $x\mapsto
 \epsilon^{1/\alpha}x$. It holds that for each $t\geq0$, $g_t^{\epsilon}$ converges weakly (in the sense of measures) to $\rho_t$ as $\epsilon\to 0$.
\end{corollary}
It is classical, see e.g. Bertoin \cite[Chapter VI, Proposition 3]{bertoin1996levy}, that for any fixed $t\geq0$, $R_t^\alpha$ has the same law as $\sup_{s\in[0,t]}Z_s^\alpha$. Hence, the results of Doney-Savov \cite[Theorem 1]{doney_asymptotic} tell us that $\rho_t$ has a continuous density $\rho_t(x)$ satisfying the following asymptotics:
\[
 \rho_t(x) \sim A tx^{-1 - \alpha} \quad \text{as }x\to\infty, \quad \text{and}\quad \rho_t(x) \sim B t^{-1/2}x^{\alpha/2 - 1} \quad \text{as }x\to0.,
\]
for some constants $A,B >0$ and for all $t >0$.

\smallskip
Regarding the limiting fractional diffusion equation, we stress that it is not the same as the one from \cite{cesbron2020fractional, cesbron2021fractional, cs2020refelctive}. Both corresponding Markov processes possess the same infinitesimal generator, with however a different domain: they both behave like a stable process when strictly positive but are not reflected in the same manner when hitting $0$. We do not find the same limiting P.D.E. inside the domain: they have an additional term expressing that particles can jump from the boundary to the interior of the domain, whereas in our case the process $(R_t^\alpha)_{t\geq0}$ leaves $0$ continuously. More precisely, their limiting P.D.E. can be written as
\begin{align*}
 \partial_t \rho_t(x) & = \int_{\mathbb{R}}\frac{\rho_t((x-z)_+) - \rho_t(x) + z\partial_x\rho_t(x)\bm{1}_{\{|z|< x\}}}{|z|^{1+\alpha}}\dr z \\
  & = \int_{\mathbb{R}}\frac{\rho_t(x-z) \bm{1}_{\{x > z\}} - \rho_t(x) + z\partial_x\rho_t(x)\bm{1}_{\{|z|< x\}}}{|z|^{1+\alpha}}\dr z + \rho_t(0)\frac{x^\alpha}{\alpha},
\end{align*}
together with some boundary condition ensuring the mass conservation.

\subsection*{Plan of the paper and sketch of the proof}

Once the limiting process is identified, it is very natural to try to establish the scaling limit of $(X_t - \inf_{s\in[0,t]}X_s)_{t\geq0}$, where $(X_t)_{t\geq0}$ is defined by \eqref{eq_non_reflected}. It should be clear that we need more than the convergence in finite dimensional distributions, and we need at least the convergence of past supremum and infimum. This is why we will use the convergence in the $\bm{\mathrm{M}}_1$-topology. 

\medskip
In Section \ref{section_preli}, we first explain why \eqref{reflected_equation} is well-posed. Then we recall and define rigorously the notion of convergence in the $\bm{\mathrm{M}}_1$-topology. This section ends with Theorem \ref{conv_m2}, which states that the convergence \eqref{result_fournier} actually holds in the space of càdlàg functions endowed with the $\bm{\mathrm{M}}_1$-topology, strengthening the result of \cite{fournier2018one}. The proof is postponed to Section \ref{append_proof_m1}.

\medskip
In Section \ref{section_inelastic}, we study the particle reflected at a completely inelastic boundary, i.e. the solution of \eqref{reflected_equation_zero}. We will see that $(X_t - \inf_{s\in[0,t]}X_s)_{t\geq0}$ plays a central role in the construction of a solution to \eqref{reflected_equation_zero}. This construction is essentially the same as in \cite{bertoin2007reflecting, bertoin2008second}. Then we prove, see Theorem \ref{cv_free_ref}, that $(X_t - \inf_{s\in[0,t]}X_s)_{t\geq0}$ and $(\mathbb{X}_t)_{t\geq0}$ have the same scaling limit, which is $(R_t^\alpha)_{t\geq0}$. The proof relies on Theorem \ref{conv_m2}, the continuous mapping theorem and Skorokhod's representation theorem.

\medskip
In Section \ref{reflected_section}, we finally establish the scaling limit of $(\bm{X}_t)_{t\geq0}$. The proof consists in using the scaling limit of $(X_t - \inf_{s\in[0,t]}X_s)_{t\geq0}$ and to compare this process with $(\bm{X}_t)_{t\geq0}$. More precisely, we will show two comparison results. First we will see that $\bm{X}_t \geq X_t - \inf_{s\in[0,t]}X_s$. Then, inspired by the work of Bertoin \cite{bertoin2007reflecting, bertoin2008second} and its construction of a solution to \eqref{reflected_equation_zero}, we will show how we can construct a solution to \eqref{reflected_equation} from the free process $(X_t, V_t)_{t\geq0}$. From this construction, we will remark that, up to a time-change $(A_t')_{t\geq0}$, we have $\bm{X}_{A_t'} \leq X_t - \inf_{s\in[0,t]}X_s$. Then the proof is almost complete if we can show that $A_t' \sim t$ as $t\to\infty$ and we will see that it is indeed the case. Subsections \ref{section_time_change} and \ref{section_persistence} are dedicated to the proof of this result. We believe that these are the most technical parts of the paper. We stress that Assumption \ref{assump_measure} is only used in Subsection \ref{section_persistence}.

\subsection*{Acknowledgements}

The author would like to thank Nicolas Fournier for his support, and for the countless fruitful discussions and advice. The author is also grateful to Camille Tardif and Quentin Berger for the forthcoming joint work, which proved to be of great help to finish this paper. Finally, the author would like to thank two anonymous referees for their careful reading and for pointing out some subtleties.

\section{Preliminaries}\label{section_preli}
\subsection{Well-posedness}\label{subsection_well_pd}

In this subsection, we explain quickly why \eqref{reflected_equation} possesses a unique solution. To do so, we will regard \eqref{eq_non_reflected} as an ordinary differential equation with a continuous random source. The force $\mathrm{F}$ being Lipschitz continuous, for any $(x_0, v_0) \in \mathbb{R}^2$, for every $\omega\in\Omega$, there exists a unique global solution $(X_t(\omega), V_t(\omega))_{t\geq0}$ to
\begin{equation}\label{free_equation_omega}
 V_t(\omega) = v_0 + \int_0^t \mathrm{F}(V_s(\omega))\dr s + B_t(\omega), \quad X_t(\omega) = x_0 + \int_0^t V_s(\omega)\dr s.
\end{equation}

One can easily construct a solution to \eqref{reflected_equation} by ''gluing`` together solutions of \eqref{eq_non_reflected} until their first hitting times of $\{0\}\times\mathbb{R}$. First consider the solution $(X_t^1, V_t^1)_{t\geq0}$ of \eqref{eq_non_reflected} starting at $(x_0, v_0)\in E$, define $\bm{\tau}_1 = \inf\{t> 0, \: X_t^1 = 0\}$ and set $(\bm{X}_t, \bm{V}_t)_{t\in[0,\bm{\tau}_1)} = (X_t^1, V_t^1)_{t\in[0,\bm{\tau}_1)}$. Then consider the solution $(X_t^2, V_t^2)_{t\geq0}$ of \eqref{eq_non_reflected} starting at $(0, M_1)$ with $B_t$ replaced by $B_{t+\bm{\tau}_1} - B_{\bm{\tau}_1}$, define $\bm{\tau}_2 - \bm{\tau}_1 = \inf\{t> 0, \: X_t^2 = 0\}$ and set $(\bm{X}_t, \bm{V}_t)_{t\in[\bm{\tau}_1,\bm{\tau}_2)} = (X_t^2, V_t^2)_{t\in[0,\bm{\tau}_2 -\bm{\tau}_1)}$. Iterating this operation indefinitely, we obtain a process $(\bm{X}_t, \bm{V}_t)_{t\in[0,\bm{\tau}_{\infty})}$ defined on $[0,\bm{\tau}_{\infty})$ where $\bm{\tau}_{\infty} = \lim_{n\to\infty}\bm{\tau}_n$, which solves \eqref{reflected_equation} on $[0,\bm{\tau}_{\infty})$.

\medskip
Consider now two solutions $(\bm{X}_t^1, \bm{V}_t^1)_{t\in[0,\bm{\tau}_{\infty}^1)}$ and $(\bm{X}_t^2, \bm{V}_t^2)_{t\in[0,\bm{\tau}_{\infty}^2)}$ of \eqref{reflected_equation} with the same initial condition, together with their respective sequence of hitting times $(\bm{\tau}_n^1)_{n\in\mathbb{N}}$ and $(\bm{\tau}_n^2)_{n\in\mathbb{N}}$. For every $\omega\in\Omega$, $(\bm{X}_t^1(\omega), \bm{V}_t^1(\omega))_{t\geq0}$ and $(\bm{X}_t^2(\omega), \bm{V}_t^2(\omega))_{t\geq0}$ are two solutions of \eqref{free_equation_omega} on the time interval $[0,\bm{\tau}_1^1(\omega)\wedge\bm{\tau}_1^2(\omega))$. Hence they are equal on this interval and $\bm{\tau}_1^1(\omega) = \bm{\tau}_1^2(\omega)$, and we can extend this reasoning to deduce that $\bm{\tau}_{\infty}^1 = \bm{\tau}_{\infty}^2 = \bm{\tau}_{\infty}$ and that $(\bm{X}_t^1, \bm{V}_t^1)_{t\in[0,\bm{\tau}_{\infty})}$ and $(\bm{X}_t^2, \bm{V}_t^2)_{t\in[0,\bm{\tau}_{\infty})}$ are equal. Therefore, uniqueness holds for \eqref{reflected_equation} for each $\omega\in\Omega$. Note that so far, we did not need the use of filtrations.

\medskip
For a solution $(\bm{X}_t, \bm{V}_t)_{t\geq0}$ of \eqref{reflected_equation}, we set $(\mathcal{F}_t)_{t\geq0}$ for the usual completion of the filtration generated by the process. Then $(\bm{X}_t, \bm{V}_t)_{t\geq0}$ is a strong Markov process in the filtration $(\mathcal{F}_t)_{t\geq0}$. Since $(\bm{\tau}_n)_{n\in\mathbb{N}}$ is the sequence of successive hitting times of $(\bm{X}_t, \bm{V}_t)_{t\geq0}$ in $\{0\}\times\mathbb{R}$, it is a sequence of $(\mathcal{F}_t)_{t\geq0}$-stopping times and we deduce from the strong Markov property that the sequence $(\bm{\tau}_{n+1} - \bm{\tau}_n)_{n\in\mathbb{N}}$ is i.i.d. and therefore $\bm{\tau}_n\to\infty$ as $n\to\infty$ almost surely. In other words, any solution of \eqref{reflected_equation} is global.

\medskip
Finally, we stress that uniqueness in law holds for \eqref{reflected_equation} as pathwise uniqueness for S.D.E. implies uniqueness in law. To summarize, \eqref{reflected_equation} is well-posed, i.e. there exists a unique and global solution to \eqref{reflected_equation} for any initial condition $(x_0, v_0)\in E$.

\subsection{$\mathrm{M}_1$-topology and the scaling limit of the free process}\label{section_conv_free}

The main result of \cite{fournier2018one}, see \eqref{result_fournier}, is a convergence in the finite dimensional distributions sense and we cannot hope to obtain a convergence in law as a process in the usual Skorokhod distance, namely the $\bm{\mathrm{J}}_1$-topology. This is due to the fact that the space of continuous functions is closed in the space of càdlàg functions endowed with the $\bm{\mathrm{J}}_1$-topology. But the process may converge in a weaker topology and we will show that the process actually converges in the $\bm{\mathrm{M}}_1$-topology, first introduced in the seminal work of Skorokhod \cite{skorohod}. In this subsection, we recall the definition and a few properties of this topology. All of the results stated here may be found in Skorokhod \cite{skorohod} or in the book of Whitt \cite[Chapter 12]{book_whitt}.

\medskip
For any $T > 0$, we denote by $\mathcal{D}_T = \mathcal{D}([0,T], \mathbb{R})$ the usual sets of càdlàg functions on $[0,T]$ valued in $\mathbb{R}$. For a function $x\in\mathcal{D}_T$, we define the completed graph $\Gamma_{T,x}$ of $x$ as follows:
\[
 \Gamma_{T,x} = \left\{(t,z)\in[0,T]\times\mathbb{R}, \: z \in [x(t-), x(t)]\right\}.
\]

\noindent
The $\bm{\mathrm{M}}_1$-topology on $\mathcal{D}_T$ is metrizable through parametric representations of the complete graphs. A parametric representation of $x$ is a continuous non-decreasing function $(u, r)$ mapping $[0,1]$ onto $\Gamma_{T,x}$. Let us denote by $\Pi_{T, x}$ the set of parametric representations of $x$. Then the $\bm{\mathrm{M}}_1$-distance on $\mathcal{D}_T$ is defined for $x_1, x_2 \in \mathcal{D}_T$ as
\[
 \dr_{\bm{\mathrm{M}}_1, T}(x_1, x_2)= \inf_{(u_i, r_i)\in\Pi_{T, x_i}}\left(||u_1 - u_2|| \vee ||r_1 - r_2||\right),
\]

\noindent
where $||\cdot ||$ is the uniform distance. The metric space $(\mathcal{D}_T, \dr_{\bm{\mathrm{M}}_1, T})$ is separable and topologically complete.

\medskip
Let us now denote by $\mathcal{D} = \mathcal{D}(\mathbb{R}_+, \mathbb{R})$ the set of càdlàg functions on $\mathbb{R}_+$. We introduce for any $t \geq0$, the usual restriction map $\bm{r}_t$ from $\mathcal{D}$ to $\mathcal{D}_t$. Then the $\bm{\mathrm{M}}_1$-distance on $\mathcal{D}$ is defined for $x, y\in \mathcal{D}$ as
\[
 \dr_{\bm{\mathrm{M}}_1}(x, y) = \int_0^{\infty}e^{-t}\left(\dr_{\bm{\mathrm{M}}_1, t}(\bm{r}_t(x), \bm{r}_t(y))\wedge 1\right) \dr t.
\]

\noindent
Again, the metric space $(\mathcal{D}, \dr_{\bm{\mathrm{M}}_1})$ is separable and topologically complete. We now briefly recall some characterization of converging sequences in $(\mathcal{D}, \dr_{\bm{\mathrm{M}}_1})$. To this end, we first introduce for $x\in\mathcal{D}$ and $\delta, T >0$, the following oscillation function
\[
 w(x, T, \delta) = \sup_{t\in[0,T]}\sup_{t_{\delta-}\leq t_1 < t_2 < t_3 \leq t_{\delta+}}d(x(t_2), [x(t_1), x(t_3)]),
\]

\noindent
where $t_{\delta-} = 0 \vee (t-\delta)$, $t_{\delta+} = T \wedge (t+\delta)$ and $d(x(t_2), [x(t_1), x(t_3)])$ is the distance of $x(t_2)$ to the segment $[x(t_1), x(t_3)]$, i.e.
\begin{equation*}
 d(x(t_2), [x(t_1), x(t_3)]) = \left\lbrace
        \begin{array}{ll}
            0 &\quad\hbox{if $x(t_2) \in [x(t_1), x(t_3)]$,} \\
            |x(t_2) - x(t_1)|\wedge |x(t_2) - x(t_3)| &\quad\hbox{otherwise.}
        \end{array}
    \right.
\end{equation*}

\noindent
We finally introduce, for $x\in\mathcal{D}$, the set $\mathrm{Disc}(x)=\{t\geq0, \Delta x (t) \neq 0\}$ of the discontinuities of $x$. We have the following characterization, which can be found in Whitt \cite[Theorem 12.5.1 and Theorem 12.9.3]{book_whitt}.
\begin{theorem}
 Let $(x_n)_{n\in\mathbb{N}}$ be a sequence in $\mathcal{D}$ and let $x$ in $\mathcal{D}$. Then the following assertions are equivalent.
 \begin{enumerate}[label=(\roman*)]
  \item $\dr_{\bm{\mathrm{M}}_1}(x_n, x)\to 0$ as $n\to\infty$.
  \item $x_n(t)$ converges to $x(t)$ in a dense subset of $\mathbb{R}_+$ containing $0$, and for every $T\notin\mathrm{Disc}(x)$, 
 \end{enumerate}
 \[
 \lim_{\delta\to0}\limsup_{n\to\infty}w(x_n, T, \delta) = 0.
\]
\noindent
In any case, we will write $x_n \overset{\bm{\mathrm{M}}_1}{\longrightarrow}x$ if $x_n$ converges to $x$ in $(\mathcal{D}, \dr_{\bm{\mathrm{M}}_1})$.
\end{theorem}

\noindent
We are now ready to characterize convergence in law for sequences of random variables valued in $(\mathcal{D}, \dr_{\bm{\mathrm{M}}_1})$. We have the following result, see for instance \cite[Theorem 3.2.1 and Theorem 3.2.2]{skorohod}. 

\begin{theorem}\label{thm_conv_m1}
Let $(X_n)_{n\in\mathbb{N}}$ be a sequence of random variables valued in $\mathcal{D}$ and let $X$ be a random variable valued in $\mathcal{D}$ such that for any $t\geq0$, $\mathbb{P}(\Delta X(t) \neq 0) = 0$. Then the following assertions are equivalent.
\begin{enumerate}[label=(\roman*)]
 \item $(X_n)_{n\in\mathbb{N}}$ converges in law to $X$ as $n\to\infty$ in $\mathcal{D}$ endowed with the $\bm{\mathrm{M}}_1$-topology.
 \item The finite dimensional distributions of $X_n$ converge to those of $X$ in some dense subset of $\mathbb{R}_+$ containing $0$, and for any $T >0$ and any $\eta >0$ we have
 \[
  \lim_{\delta\to0}\limsup_{n\to\infty}\mathbb{P}(w(X_n, T, \delta) > \eta) = 0.
 \]
\end{enumerate}
\end{theorem}

The following theorem shows that the convergence proved in \cite{fournier2018one} can be enhanced, in a stronger convergence.

\begin{theorem}\label{conv_m2}
Grant Assumption \ref{assump_fournier} and let $(X_t, V_t)_{t\geq0}$ be a solution of \eqref{eq_non_reflected} starting at $(0,v_0)$ with $v_0\in\mathbb{R}$. Let $(Z_t^{\alpha})_{t\geq0}$ be a symmetric stable process with $\alpha = (\beta + 1) / 3$ and such that $\mathbb{E}[e^{i\xi Z_t^{\alpha}}] = \exp(-t\sigma_\alpha |\xi|^{\alpha})$, where $\sigma_\alpha$ is defined by \eqref{sigma}. Then we have
\[
\left(\epsilon^{1/\alpha}X_{t/\epsilon}\right)_{t\geq0} \longrightarrow \left(Z_t^\alpha\right)_{t\geq0}\qquad\hbox{as $\epsilon\to0$}
\]

\noindent
in law for the $\bm{\mathrm{M}}_1$-topology.
\end{theorem}

\noindent
The proof is postponed to Section \ref{append_proof_m1}.

\section{The particle reflected at an inelastic boundary and its scaling limit}\label{section_inelastic}

In this section, we construct a weak solution of \eqref{reflected_equation_zero} following Bertoin \cite{bertoin2007reflecting, bertoin2008second}. Although the author uses a Brownian motion instead of the process $(V_t)_{t\geq0}$, some of the trajectorial properties shown in \cite{bertoin2007reflecting, bertoin2008second} still hold in our case.

\medskip
Consider a weak solution $(X_t, V_t)_{t\geq0}$ of \eqref{eq_non_reflected} on some filtered probability space $(\Omega, \mathcal{F}, (\mathcal{F}_t)_{t\geq0}, \mathbb{P})$ supporting some Brownian motion $(B_t)_{t\geq0}$, starting at $(0, v_0)$ where $v_0 \geq 0$. We introduce
\begin{equation}\label{proc_ref_inf}
 \mathcal{X}_t = X_t - \inf_{s\in[0,t]}X_s.
\end{equation}

\noindent
Before introducing the solution of \eqref{reflected_equation_zero}, we will first study a little bit the process $(\mathcal{X}_t)_{t\geq0}$. It is a continuous process, which has the same motion as $X_t$ as long as $X_t > \inf_{s\in[0,t]}X_s$. When $X_t$ reaches its past infimum, it is necessarily with a non-positive velocity. When this happens with a strictly negative velocity, say at time $t_0$, the infimum process decreases until the velocity hits zero, i.e. on the interval $[t_0, d_{t_0}]$ where $d_{t_0} = \inf\{s \geq t_0, \: V_s = 0\}$. On this time interval, $X_t = \inf_{s\in[0,t]}X_s$ and thus $\mathcal{X}_t = 0$. It is interesting and useful to study
\[
 \mathcal{I}_{\mathcal{X}} = \left\{t\geq0, \: \mathcal{X}_t = 0\right\}.
\]

\noindent
As explained in \cite[page 2025]{bertoin2007reflecting}, the following random set appears naturally in the study of $\mathcal{I}_{\mathcal{X}}$:
\[
\mathcal{H} = \left\{t \geq 0,\: V_t < 0,\: \mathcal{X}_t = 0 \text{ and } \exists\epsilon >0, \:\forall s\in[t-\epsilon, t),\mathcal{X}_{s} > 0\right\}.
\]

\noindent
Each element of $\mathcal{H}$ is the first time at which $X_t$ reaches its infimum, after an ''excursion'' above its infimum. Each point of $\mathcal{H}$ is necessarily isolated and thus $\mathcal{H}$ is countable. We are now able to state the following lemma, which is identical to \cite[Lemma 2]{bertoin2007reflecting}, and which characterizes the set $\mathcal{I}_{\mathcal{X}}$.

\begin{lemma}\label{decomp_zero}
 The decomposition of the interior of $\mathcal{I}_{\mathcal{X}}$ as a union of disjoint intervals is given by
 \[
  \mathring{\mathcal{I}}_{\mathcal{X}} = \bigcup_{s\in\mathcal{H}}]s, d_s[,
 \]
\noindent
where $d_s = \inf\{u > s, \: V_u = 0\}$. Moreover, the boundary $\partial\mathcal{I}_{\mathcal{X}} = \mathcal{I_{\mathcal{X}}}\setminus\mathring{\mathcal{I}}_{\mathcal{X}}$ has zero Lebesgue measure.
\end{lemma}

\begin{proof}
The idea is to use the result from \cite{bertoin2007reflecting} and the Girsanov theorem. Indeed, the same result is proved in \cite{bertoin2007reflecting} when $\mathrm{F} = 0$ and $v_0 = 0$.

\medskip\noindent
\textit{Step 1:} We first show that the results holds when $v_0 = 0$. We first define the following local martingales
\[
 L_t = - \int_0^t \mathrm{F}(V_s) \dr B_s \quad \text{and} \quad \mathcal{E}(L)_t = \exp\left(L_t - \frac{1}{2}\langle L \rangle_t\right)
\]

\noindent
For any $T>0$, $\mathbb{E}[e^{\frac{1}{2}\langle L \rangle_T}] \leq e^{\frac{1}{2}T||F^2||_{\infty}} < \infty$, by Assumption \ref{assump_fournier}, and therefore, by the Novikov criterion and the Girsanov theorem, the measure $\mathbb{Q}_T = \mathcal{E}(L)_T\cdot\mathbb{P}$ is a probability measure on $(\Omega, \mathcal{F}_T)$ and the process
\[
 V_t = \int_0^t \mathrm{F}(V_s)\dr s + B_t = B_t - \langle B, L \rangle_t
\]

\noindent
is an $(\mathcal{F}_t)_{t\in[0,T]}$-Brownian motion starting at $0$ under $\mathbb{Q}_T$. We deduce from Lemma 2 in \cite{bertoin2007reflecting} that $\mathbb{Q}_T$-a.s.,
\[
 \mathring{\mathcal{I}}_{\mathcal{X}}\cap[0,T] = \left(\bigcup_{s\in\mathcal{H}}]s, d_s[\right) \cap [0, T]
\]

\noindent
and $\partial\mathcal{I}_{\mathcal{X}}\cap[0,T]$ has zero Lebesgue measure. Since this holds for every $T > 0$, and since $\mathbb{Q}_T \sim \mathbb{P}$, this establishes our result for the initial condition $(0,0)$.

\medskip\noindent
\textit{Step 2:} We show the results when $v_0 > 0$. Let us define $\tau = \inf\{t > 0, \: X_t = 0\}$. Since $v_0 > 0$, it is clear that $\tau > 0$ a.s. and that $\tau\in\mathcal{H}$ and
\[
 \mathcal{I}_{\mathcal{X}} \cap [0, d_{\tau}] = \{0\} \cup [\tau, d_{\tau}].
\]

\noindent
Now the process $(\bar{X}_t, \bar{V}_t)_{t\geq0} = (X_{t+d_\tau} - X_{d_\tau}, V_{t+d_\tau} )_{t\geq0}$ is a solution of \eqref{eq_non_reflected} starting at $(0,0)$. Moreover, since $X_{d_\tau} = \inf_{s\in[0,d_\tau]}X_s$, we clearly have on the event $\{d_\tau < t\}$,
\[
 \inf_{s\in[0,t]}X_s = X_{d_\tau} + \inf_{s\in[0,t-d_\tau]}\bar{X}_s,
\]

\noindent
which yields
\[
 \mathcal{I}_{\mathcal{X}} = \{0\} \cup [\tau, d_{\tau}] \cup \bar{\mathcal{I}}, \quad \text{where} \quad \bar{\mathcal{I}} = \Big\{t\geq d_\tau, \:\bar{X}_{t-d_\tau} = \inf_{s\in[0,t-d_\tau]}\bar{X}_s\Big\} = \mathcal{I}_{\mathcal{X}} \cap [d_\tau,\infty).
\]

\noindent
By the first step,
\[
 \bar{\mathcal{I}} = \bigcup_{s\in\mathcal{H}\cap[d_\tau,\infty)}]s, d_s[,
\]

\noindent
whence the result.
\end{proof}

We are now ready to complete the construction of the solution of \eqref{reflected_equation_zero}. We introduce the following time-change, as well as its right-continuous inverse 
\[
 A_t = \int_0^t \bm{1}_{\{\mathcal{X}_s > 0\}}\dr s \quad \text{and} \quad T_t = \inf\left\{s>0, A_s > t\right\}.
\]

\noindent
We claim that by using the same arguments as in \cite[page 2024-2025]{bertoin2007reflecting}, or by using the Girsanov theorem as in the proof of Lemma \ref{decomp_zero}, it holds that
\begin{equation}\label{pouet}
 \mathcal{X}_t = \int_0^t V_s\bm{1}_{\{\mathcal{X}_s > 0\}} \dr s = \int_0^t V_s \dr A_s.
\end{equation}

\noindent
Therefore, by the change of variables theorem for Stieltjes integrals, we get
\begin{equation}\label{sol_reflec_zero}
 \mathbb{X}_t :=  \mathcal{X}_{T_t} = \int_0^{T_t}V_s \dr A_s = \int_0^t \mathbb{V}_s \dr s, \quad \text{where} \quad \mathbb{V}_t := V_{T_t}.
\end{equation}

\noindent
We finally set $\mathcal{G}_t = \mathcal{F}_{T_t}$ and state the main theorem of this section.

\begin{theorem}\label{thm_sol_eq_zero}
There exists a $(\mathcal{G}_t)_{t\geq0}$-Brownian motion such that $(\mathbb{X}_t, \mathbb{V}_t)_{t\geq0}$ is a solution of \eqref{reflected_equation_zero} on $(\Omega, \mathcal{F}, (\mathcal{G}_t)_{t\geq0}, \mathbb{P})$.
\end{theorem}

\begin{proof}
 We have already seen that $\mathbb{X}_t = \int_0^t \mathbb{V}_s \dr s$. Concerning the velocity, we paraphrase the proof of Proposition 1 in \cite{bertoin2008second} and start by decomposing the process $(\mathbb{V}_t)_{t\geq0}$ as follows
 \[
  \mathbb{V}_t = v_0 + \int_0^{T_t}\bm{1}_{\{\mathcal{X}_s = 0\}}\dr V_s + \int_0^{T_t}\bm{1}_{\{\mathcal{X}_s > 0\}}\dr V_s =: v_0 + C_t + D_t. 
 \]
 
\noindent
Let us first deal with $D_t$. Using \eqref{eq_non_reflected}, we get
\[
 D_t = \int_0^{T_t}\bm{1}_{\{\mathcal{X}_s > 0\}}\mathrm{F}(V_s)\dr s + \int_0^{T_t}\bm{1}_{\{\mathcal{X}_s > 0\}} \dr B_s.
\]

\noindent
The last term is a $(\mathcal{G}_t)_{t\geq0}$-local martingale whose quadratic variation at time $t$ equals $A_{T_t} = t$ by the very definition of $(T_t)_{t\geq0}$ and thus, it is a Brownian motion that we will denote $(\mathbb{B}_t)_{t\geq0}$. Regarding the second term, we use again the change of variables for Stieltjes integrals, to deduce that $\int_0^{T_t}\bm{1}_{\{\mathcal{X}_s > 0\}}\mathrm{F}(V_s)\dr s = \int_0^{T_t}\mathrm{F}(V_s)\dr A_s = \int_0^{t}\mathrm{F}(\mathbb{V}_s)\dr s$. We have proved that $D_t = \int_0^{t}\mathrm{F}(\mathbb{V}_s)\dr s + \mathbb{B}_t$. We now deal with $C_t$. First we note that the semimartingale
\[
 \int_0^t \bm{1}_{\{s \in\partial\mathcal{I}_{\mathcal{X}}\}}\dr V_s = \int_0^t \bm{1}_{\{s \in\partial\mathcal{I}_{\mathcal{X}}\}}\mathrm{F}(V_s)\dr s + \int_0^t \bm{1}_{\{s \in\partial\mathcal{I}_{\mathcal{X}}\}}\dr B_s
\]

\noindent
is a.s. null. Indeed since $\partial\mathcal{I}_{\mathcal{X}}$ has zero Lebesgue measure by Lemma \ref{decomp_zero}, the first term is obviously equal to zero. By the same argument, the second term is a local martingale whose quadratic variation is equal to zero, and therefore it is null. Then we can write as in \cite{bertoin2007reflecting}
\[
 C_t = \int_0^{T_t}\bm{1}_{\{s \in\mathcal{I}_{\mathcal{X}}\}}\dr V_s = \int_0^{T_t}\bm{1}_{\{s \in\mathring{\mathcal{I}}_{\mathcal{X}}\}}\dr V_s = \sum_{u\in\mathcal{H}, u \leq T_t}(V_{d_u} - V_u) = -\sum_{u\in\mathcal{H}, u \leq T_t}V_u.
\]

\noindent
In the third equality, we used Lemma \ref{decomp_zero} and the fact that, by definition, $T_t\notin\mathring{\mathcal{I}}_{\mathcal{X}}$ and in the fourth that $V_{d_u} = 0$ for every $u\in\mathcal{H}$. To every point $u\in\mathcal{H}$ corresponds a unique jumping time $s$ of $(T_t)_{t\geq0}$ at which $\mathcal{X}_t$ hits the boundary, i.e. $u = T_{s-}$ and $\mathcal{X}_{T_{s-}} = \mathcal{X}_{T_{s}} = 0$. Indeed, the flat sections of $(A_t)_{t\geq0}$ are precisely $\mathring{\mathcal{I}}_{\mathcal{X}}$, and therefore, for every $u\in\mathcal{H}$, we have $A_u = A_{d_u}$ and thus if we set $s= A_u$, then $T_{s-} = u$ and $T_s = d_u$. Hence we have
\[
 C_t = -\sum_{0< s \leq t}V_{T_{s-}}\bm{1}_{\{\mathcal{X}_{T_{s}} = 0\}} = -\sum_{0< s \leq t}\mathbb{V}_{s-}\bm{1}_{\{\mathbb{X}_s = 0\}},
\]

\noindent
which completes the proof.
\end{proof}

We are now ready to study the scaling limits of $(\mathcal{X}_t)_{t\geq0}$ and $(\mathbb{X}_t)_{t\geq0}$.
\begin{theorem}\label{cv_free_ref}
 Let $(\mathcal{X}_t)_{t\geq0}$ and $(\mathbb{X}_t)_{t\geq0}$ be defined as in \eqref{proc_ref_inf} and \eqref{sol_reflec_zero}. Let also $(Z_t^{\alpha})_{t\geq0}$ be a symmetric stable process with $\alpha = (\beta + 1) / 3$ and such that $\mathbb{E}[e^{i\xi Z_t^{\alpha}}] = \exp(-t\sigma_\alpha |\xi|^{\alpha})$ where $\sigma_\alpha$ is defined by \eqref{sigma}. Let $R_t^\alpha = Z_t^{\alpha} - \inf_{s\in[0,t]}Z_s^{\alpha}$. Then we have
 \[
  \left(\epsilon^{1/\alpha}\mathcal{X}_{t/\epsilon}\right)_{t\geq0} \longrightarrow \left(R_t^\alpha\right)_{t\geq0} \quad \text{and} \quad \left(\epsilon^{1/\alpha}\mathbb{X}_{t/\epsilon}\right)_{t\geq0} \longrightarrow \left(R_t^\alpha\right)_{t\geq0} \quad \hbox{as $\epsilon\to0$}
 \]

\noindent
in law for the $\bm{\mathrm{M}}_1$-topology.
\end{theorem}

\begin{proof}
The convergence of $(\epsilon^{1/\alpha}\mathcal{X}_{t/\epsilon})_{t\geq0}$ is straightforward by the continuous mapping theorem and Theorem \ref{conv_m2}. Indeed the reflection map from $\mathcal{D}$ into itself, which maps a function $x$ to the function $y$, defined for every $t\geq0$ by $y(t) = x(t) - 0\wedge\inf_{s\in[0,t]}x(s)$, is continuous with respect to the $\bm{\mathrm{M}}_1$-topology, see Whitt \cite[Chapter 13, Theorem 13.5.1]{book_whitt}.

\medskip
We now study $(\mathbb{X}_t)_{t\geq0} = (\mathcal{X}_{T_t})_{t\geq0}$. By Skorokhod's representation theorem, there exist a family of processes $(\mathcal{X}_t^{\epsilon})_{t\geq0}$ indexed by $\epsilon > 0$, and a reflected symmetric stable process $(R_t^\alpha)_{t\geq0}$, both defined on the probability space $([0,1], \mathcal{B}([0,1]), \lambda)$, where $\lambda$ denotes the Lebesgue measure on $[0,1]$, such that for every $\epsilon > 0$, $(\mathcal{X}_t^{\epsilon})_{t\geq0}\overset{d}{=} (\mathcal{X}_{t/\epsilon})_{t\geq0}$ and such that,
\[
 \lambda-\text{a.s.,} \quad d_{\bm{\mathrm{M}}_1}((\epsilon^{1/\alpha}\mathcal{X}_t^{\epsilon})_{t\geq0}, (R_t^\alpha)_{t\geq0}) \underset{\epsilon\to0}{\longrightarrow}0.
\]

\noindent
Let us denote by $J$ be the set of discontinuities of $(R_t^\alpha)_{t\geq0}$. Then, by \cite[Chapter 12, Lemma 12.5.1]{book_whitt}, we get that
\[
 \lambda-\text{a.s.,} \quad \text{for every }t\notin J, \quad \epsilon^{1/\alpha}\mathcal{X}_t^{\epsilon} \underset{\epsilon\to0}{\longrightarrow}R_t^\alpha.
\]

\noindent
We introduce the time-change process $A_t^{\epsilon} = \int_0^t \bm{1}_{\{\epsilon^{1/\alpha}\mathcal{X}_s^{\epsilon} > 0\}} \dr s$ for every $\epsilon > 0$ and every $t\geq0$. Then we get by the Fatou lemma that $\lambda-$a.s., for every $t\geq0$,
\[
 \int_0^t \liminf_{\epsilon\to0}\bm{1}_{\{\epsilon^{1/\alpha}\mathcal{X}_s^{\epsilon} > 0\}} \dr s \leq \liminf_{\epsilon\to0}A_t^{\epsilon} \leq \limsup_{\epsilon\to0}A_t^{\epsilon} \leq t. 
\]

\noindent
Since $J$ is countable, we have $\lambda-$a.s., for a.e. $s\in[0,t]$, $\bm{1}_{\{R_s^\alpha > 0\}} \leq \liminf_{\epsilon\to0}\bm{1}_{\{\epsilon^{1/\alpha}\mathcal{X}_s^{\epsilon} > 0\}}$. Since the zero set of the reflected stable process is a.s. Lebesgue-null, we conclude that $\lambda-$a.s., for every $t\geq0$, $A_t^{\epsilon}\to t$ as $\epsilon\to0$.

\medskip
Let us denote by $(T_t^{\epsilon})_{t\geq0}$ the right-continuous inverse of $(A_t^{\epsilon})_{t\geq0}$. As an immediate consequence, we have $\lambda-$a.s., for every $t\geq0$, $T_t^{\epsilon}\to t$ as $\epsilon\to0$. Since the $\bm{\mathrm{M}}_1$-topology on the space of non-increasing functions reduces to pointwise convergence on a dense subset including $0$, see \cite[Corollary 12.5.1]{book_whitt}, we have $(T_t^{\epsilon})_{t\geq0}\longrightarrow(\mathrm{id}_t)_{t\geq0}$ as $\epsilon\to0$ in law in the $\bm{\mathrm{M}}_1$-topology, where $\mathrm{id}_t = t$ is the identity process.

\medskip
By a simple substitution, we see that $(\epsilon A_{t/\epsilon})_{t\geq0} \overset{d}{=} (A_{t}^{\epsilon})_{t\geq0}$, from which we deduce that $(\epsilon T_{t/\epsilon})_{t\geq0}\overset{d}{=} (T_{t}^{\epsilon})_{t\geq0}$ and that $(\epsilon T_{t/\epsilon})_{t\geq0}\longrightarrow(\mathrm{id}_t)_{t\geq0}$ as $\epsilon\to0$ in law in the $\bm{\mathrm{M}}_1$-topology. Hence, by a generalization of Slutsky theorem, see for instance \cite[Section 3, Theorem 3.9]{billingsley_conv}, we get that $(\epsilon^{1/\alpha}\mathcal{X}_{t/\epsilon}, \epsilon T_{t/\epsilon})_{t\geq0}$ converges in law to $(R_t^{\alpha}, \mathrm{id}_t)_{t\geq0}$ in $\mathcal{D}\times\mathcal{D} = \mathcal{D}(\mathbb{R}_+, \mathbb{R}^2)$ endowed with the $\bm{\mathrm{M}}_1$-topology.

\medskip
We are now able to conclude. Let us denote by $\mathcal{D}_\uparrow$ the set of càdlàg and non-decreasing functions from $\mathbb{R}_+$ to $\mathbb{R}_+$, and $\mathcal{C}_{\uparrow\uparrow}$ the set of continuous and strictly increasing functions from $\mathbb{R}_+$ to $\mathbb{R}_+$. Then the composition map, from $\mathcal{D}\times\mathcal{D}_\uparrow$ to $\mathcal{D}$, which maps $(x,y)$ to $x\circ y$, is continuous on $\mathcal{D}\times\mathcal{C}_{\uparrow\uparrow}$, see \cite[Chapter 13, Theorem 13.2.3]{book_whitt}. Hence $(\epsilon^{1/\alpha}\mathbb{X}_{t/\epsilon})_{t\geq0} = (\epsilon^{1/\alpha}\mathcal{X}_{T_{t/\epsilon}})_{t\geq0}$ converges to $(R_t^\alpha)_{t\geq0}$ by the continuous mapping theorem.
\end{proof}

\section{The reflected process with diffusive boundary condition}\label{reflected_section}

In this section, we finally study the process $(\bm{X}_t, \bm{V}_t)_{t\geq0}$, solution of \eqref{reflected_equation}, and its scaling limit. To establish our result, we will rely on the convergence of $(\mathcal{X}_t)_{t\geq0}$ from Theorem \ref{cv_free_ref}. First, we will show that $\bm{X}_t \geq \mathcal{X}_t = X_t - \inf_{s\in[0,t]}X_s$ where $(\bm{X}_t)_{t\geq0}$ and $(X_t)_{t\geq0}$ are the solutions to \eqref{reflected_equation} and \eqref{eq_non_reflected} with the same Brownian motion. Then, inspired by the work of Bertoin \cite{bertoin2007reflecting}, we will give another construction of $(\bm{X}_t, \bm{V}_t)_{t\geq0}$ and we will prove that, up to a time-change, $\bm{X}_t \leq \mathcal{X}_t$. Finally, we will conclude since the time-change at stake is asymptotically equivalent to $t$. We believe that the proof of the limit of the time-change is the most technical part of the paper and this will the subject of Subsections \ref{section_time_change} and \ref{section_persistence}.

\subsection{A comparison result}

Let $(\Omega, \mathcal{F}, \mathbb{P})$ be some probability space supporting a Brownian motion $(B_t)_{t\geq0}$ and a sequence of i.i.d $\mu$-distributed random variables $(M_n)_{n\in\mathbb{N}}$, independent of each other. Consider the solution $(\bm{X}_t, \bm{V}_t)_{t\geq0}$ of \eqref{reflected_equation}, starting at $(0, v_0)$ where $v_0 >0$, as well as its sequence of hitting times $(\bm{\tau}_n)_{n\in\mathbb{N}}$. We also consider on the same probability space the solution $(X_t, V_t)_{t\geq0}$ of \eqref{eq_non_reflected} starting at $(0, v_0)$, with the same driving Brownian motion $(B_t)_{t\geq0}$. Let also $(\mathcal{X}_t)_{t\geq0}$ be defined as in \eqref{proc_ref_inf}, i.e. $\mathcal{X}_t = X_t - \inf_{s\in[0, t]}X_s$. We have the following proposition.
\begin{proposition}\label{first_compar}
 Almost surely, for any $t\geq0$, $\bm{X}_t \geq \mathcal{X}_t$.
\end{proposition}

\begin{proof}
 \textit{Step 1:} We first prove that a.s. for any $t\geq0$, $V_t \leq \bm{V}_t$ and to do so, we use the classical comparison theorem for O.D.E's. We prove recursively that for any $n\in\mathbb{N}$, a.s. for any $t\in[0, \bm{\tau}_n)$, $V_t \leq \bm{V}_t$. This is true for $n=1$, since the processes are both solutions of the same well-posed O.D.E. on $[0,\bm{\tau}_1)$, with the same starting point. Hence they are equal on this interval. 
 
 \medskip
Now let us assume that for some $n\in\mathbb{N}$, a.s. for any $t\in[0, \bm{\tau}_n)$, $V_t \leq \bm{V}_t$. Then a.s. $V_{\bm{\tau}_n} = V_{\bm{\tau}_n -} \leq \bm{V}_{\bm{\tau}_n -} \leq 0$. We also see that $(\bm{V}_t)_{t\geq0}$ and $(V_t)_{t\geq0}$ are two solutions of the same O.D.E. on the interval $[\bm{\tau}_n, \bm{\tau}_{n+1})$. Indeed, a.s. for any $t\in[\bm{\tau}_n, \bm{\tau}_{n+1})$, we have
 \[
  \bm{V}_t = M_n + \int_{\bm{\tau}_n}^t\mathrm{F}(\bm{V}_s)\dr s + B_t - B_{\bm{\tau}_n} \quad \text{and} \quad V_t = V_{\bm{\tau}_n} + \int_{\bm{\tau}_n}^t\mathrm{F}(V_s)\dr s + B_t - B_{\bm{\tau}_n}.
 \]
\noindent
Since $M_n > 0$ with probability one, we have a.s. $M_n \geq V_{\bm{\tau}_n}$ and thus, by the comparison theorem, we deduce that a.s., for any $t\in[\bm{\tau}_n, \bm{\tau}_{n+1})$, $V_t \leq \bm{V}_t$. This achieves the first step.

\bigskip\noindent
\textit{Step 2:} We conclude. Almost surely, for any $0\leq s \leq t$, we have
\[
 X_t -X_s = \int_s^t V_s \dr s \leq \int_s^t \bm{V}_s \dr s = \bm{X_t} - \bm{X_s} \leq \bm{X_t},
\]

\noindent
the last inequality holding since $\bm{X}_s \geq 0$ a.s. This implies that a.s., $\inf_{s\in[0,t]}X_s \geq X_t - \bm{X}_t$, i.e. $\bm{X}_t \geq \mathcal{X}_t$, for any $t\geq0$.
\end{proof}

\subsection{A second construction}

In this subsection, we give another construction of $(\bm{X}_t, \bm{V}_t)_{t\geq0}$, which is inspired by the construction given by Bertoin \cite{bertoin2007reflecting} of the reflected Langevin process at an inelastic boundary.

\medskip
Let $(\Omega, \mathcal{F}, \mathbb{P})$ be a probability space supporting a Brownian motion $(B_t)_{t\geq0}$ and a sequence of i.i.d $\mu$-distributed random variables $(M_n)_{n\in\mathbb{N}}$, independent from the Brownian motion. We set $(\mathcal{F}_t)_{t\geq0}$ to be the filtration generated by $(B_t)_{t\geq0}$ after the usual completions, and we introduce the filtration $(\mathcal{G}_t)_{t\geq0}$ defined for every $t\geq0$ as 
\[
 \mathcal{G}_t = \mathcal{F}_{t} \vee \sigma\left(\left\{(M_n)_{n\in\mathbb{N}}\right\}\right).
\]
It is clear that $(B_t)_{t\geq0}$ remains a Brownian motion in the filtration $(\mathcal{G}_t)_{t\geq0}$. Next, we consider the strong solution $(X_t, V_t)_{t\geq0}$ of \eqref{eq_non_reflected} starting at $(0, v_0)$, which remains a strong Markov process in this filtration. We set $\sigma_0 = 0$, then we set $\tau_1 = \inf\{t >0, X_t = 0\}$ and we define recursively the sequence of random times
\[
 \sigma_{n} = \inf\left\{t > \tau_n,\: V_t = M_{n}\right\}\quad \text{and} \quad \tau_{n+1} = \inf\left\{t>\sigma_{n},\: X_t = X_{\sigma_{n}}\right\},
\] 
where $n\in\mathbb{N}$. We have the following lemma.
\begin{lemma}\label{lemma_random_times}
 The random times $(\sigma_n)_{n\in\mathbb{N}}$ and $(\tau_n)_{n\in\mathbb{N}}$ are $(\mathcal{G}_t)_{t\geq0}$-stopping times which are almost surely finite. Moreover, the sequence $(\tau_n - \sigma_{n-1}, \sigma_n - \sigma_{n-1})_{n\geq2}$ forms a sequence of identically distributed random variables and the subsequences $(\tau_{2n} - \sigma_{2n-1}, \sigma_{2n} - \sigma_{2n-1})_{n\geq1}$ and $(\tau_{2n + 1} - \sigma_{2n}, \sigma_{2n + 1} - \sigma_{2n})_{n\geq1}$ form sequences of i.i.d. random variables.
\end{lemma}

\begin{proof}
 The process $(V_t)_{t\geq0}$ being a recurrent diffusion, see for instance \cite{fournier2018one}, we have almost surely, $\liminf_{t\to\infty}V_t = -\infty$ and $\limsup_{t\to\infty}V_t = \infty$. Moreover, by Theorem \ref{conv_m2}, we also have $\liminf_{t\to\infty}X_t = -\infty$ and $\limsup_{t\to\infty}X_t = \infty$ a.s. Hence, the random times previously defined are almost surely finite. Moreover, it is clear by a simple induction that the random times $(\sigma_n)_{n\in\mathbb{N}}$ and $(\tau_n)_{n\in\mathbb{N}}$ are $(\mathcal{G}_t)_{t\geq0}$-stopping times. Indeed, if $\sigma_n$ is a $(\mathcal{G}_t)_{t\geq0}$-stopping time for some $n\geq0$, then $\tau_{n+1} = \inf\left\{t>\sigma_{n},\: X_t = X_{\sigma_{n}}\right\}$ is obviously a $(\mathcal{G}_t)_{t\geq0}$-stopping time and since $\sigma_{n+1}$ is the first hitting of zero after $\tau_{n+1}$ of the $(\mathcal{G}_t)_{t\geq0}$-adapted process $(V_t - M_{n+1})_{t\geq0}$, it is also a stopping time for $(\mathcal{G}_t)_{t\geq0}$. 
 
 \smallskip
 Then, for any $n \geq 2$, applying the strong Markov property of $(X_t, V_t)_{t\geq0}$ at time $\sigma_{n-1}$, we see that for any $n\geq 2$, $(\tau_n - \sigma_{n-1}, \sigma_n - \sigma_{n-1})$ has the same law as $(\tau_2 - \sigma_1, \sigma_2 - \sigma_1)$. The fact that the subsequences form sequences of i.i.d. random variables follows again from the strong Markov property and the fact that for any $n \geq 2$, $(\tau_n - \sigma_{n-1}, \sigma_n - \sigma_{n-1})$ only depends on $(B_{t} - B_{\sigma_{n-1}})_{t\in[\sigma_{n-1}, \sigma_n]}$ and $(M_{n-1}, M_n)$
\end{proof}

We are finally ready to start the construction of the reflected process. We define the $(\mathcal{G}_t)_{t\geq0}$-adapted processes
\begin{equation}\label{second_cons}
 \mathfrak{X}_t = \sum_{n\in\mathbb{N}}\left(X_t - X_{\sigma_{n-1}}\right)\bm{1}_{\{\sigma_{n-1} \leq t <\tau_n\}} \quad\text{and}\quad\mathfrak{V}_t = \bm{1}_{\{\mathfrak{X}_t > 0\}}V_t.
\end{equation}

\noindent
We refer to Figure \ref{fig:decomp} for a visual representation. For every $n\in\mathbb{N}$, the process $(\mathfrak{X}_t)_{t\geq0}$ has the same trajectory as $(X_t)_{t\geq0}$ on $[\sigma_{n-1}, \tau_n]$ shifted by $X_{\sigma_{n-1}}$ and is null on $[\tau_n, \sigma_{n}]$. Now by the very definition of $\sigma_{n}$ and $\tau_n$, $X_t$ is above $X_{\sigma_{n-1}}$ on $[\sigma_{n-1}, \tau_n]$, and it should be clear that a.s., for every $t\geq0$, $\mathfrak{X}_t \geq0$ and that
\begin{equation}\label{zero_set}
 \mathcal{I}_{\mathfrak{X}} = \left\{t\geq0,\: \mathfrak{X}_t = 0\right\} = \{0\}\cup\left(\bigcup_{n\in\mathbb{N}}[\tau_n, \sigma_{n}]\right).
\end{equation}

\begin{figure}
\begin{center}
\includegraphics[scale=1]{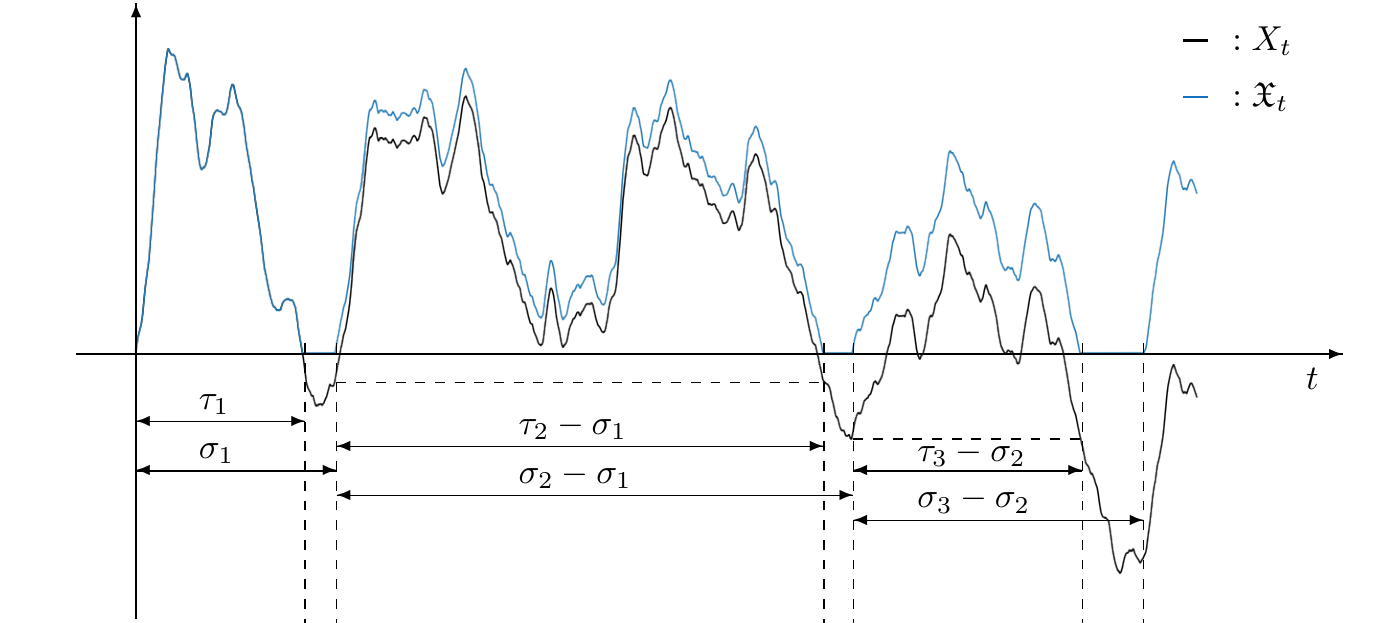}
\end{center}
\caption{\footnotesize Graphical representation of a realization of $(X_t)_{t\geq0}$ and $(\mathfrak{X}_t)_{t\geq0}$.}
\label{fig:decomp}
\end{figure}

\noindent
As in \eqref{sol_reflec_zero}, we also have
\begin{equation}\label{x_integ_v}
 \mathfrak{X}_t = \int_0^t \mathfrak{V}_s \dr s
\end{equation}

\noindent
Indeed, this can be easily checked using that for every $s\geq0$, $\bm{1}_{\{\mathfrak{X}_s > 0\}} = \sum_{n\in\mathbb{N}}\bm{1}_{\{\sigma_{n-1} < s <\tau_n\}}$ and that $X_{\tau_n} = X_{\sigma_{n-1}}$. As in the previous subsection, we can compare $\mathfrak{X}_t$ to $\mathcal{X}_t = X_t - \inf_{s\in[0,t]}X_s$.

\begin{proposition}\label{second_comparison}
Almost surely, for any $t\geq0$, $\mathfrak{X}_t \leq \mathcal{X}_t$.
\end{proposition}

\begin{proof}
The result follows almost immediately from the definition of $(\mathfrak{X}_t)_{t\geq0}$. Indeed, almost surely, for any $n\in\mathbb{N}$ and any $t\geq0$,
\[
 X_{\sigma_{n-1}}\bm{1}_{\{\sigma_{n-1} \leq t <\tau_n\}} \geq \inf_{s\in[0,t]}X_s \times \bm{1}_{\{\sigma_{n-1} \leq t <\tau_n\}},
\]

\noindent
and therefore we have almost surely, for any $t\geq0$,
\[
 \mathfrak{X}_t = \sum_{n\in\mathbb{N}}\left(X_t - X_{\sigma_{n-1}}\right)\bm{1}_{\{\sigma_{n-1} \leq t <\tau_n\}} \leq \mathcal{X}_t\times\sum_{n\in\mathbb{N}}\bm{1}_{\{\sigma_{n-1} \leq t <\tau_n\}} \leq \mathcal{X}_t,
\]

\noindent
which achieves the proof.
\end{proof}
We are now ready to complete the construction of the solution of \eqref{reflected_equation}, as we did in Section \ref{section_inelastic}. We introduce the $(\mathcal{G}_t)_{t\geq0}$-adapted time-change $A_t' = \int_0^t \bm{1}_{\{\mathfrak{X}_s > 0\}} \dr s$ as well as its right-continuous inverse $T_t'$, and we define the processes
\[
 \bm{X}_t = \mathfrak{X}_{T_t'} = \int_0^{T_t'}V_s \dr A_s' = \int_0^t \bm{V}_s \dr s \quad \text{where} \quad \bm{V}_t = V_{T_t'}.
\]

\noindent
We finally define the filtration $(\bm{\mathcal{F}}_t)_{t\geq0} = (\mathcal{G}_{T_t'})_{t\geq0}$. We have the following theorem.
\begin{theorem}\label{thm_second_cons}
 There exists an $(\bm{\mathcal{F}}_t)_{t\geq0}$-Brownian motion such that the process $(\bm{X}_t, \bm{V}_t)_{t\geq0}$ is a solution of \eqref{reflected_equation} on $(\Omega, \mathcal{F}, (\bm{\mathcal{F}}_t)_{t\geq0}, \mathbb{P})$.
\end{theorem}

\begin{proof}
 The proof is very similar to the proof of Theorem \ref{thm_sol_eq_zero}. We start by decomposing the process $(\bm{V}_t)_{t\geq0}$ as follows
\[
\bm{V}_t = v_0 + \int_0^{T_t'}\bm{1}_{\{\mathfrak{X}_s = 0\}}\dr V_s + \int_0^{T_t'}\bm{1}_{\{\mathfrak{X}_s > 0\}}\dr V_s =: v_0 + C_t + D_t.
\]

\noindent
As in the proof of Theorem \ref{thm_sol_eq_zero}, we easily see, using \eqref{eq_non_reflected}, the definition of $(T_t')_{t\geq0}$ and the change of variables theorem for Stieltjes integrals, that
\[
 D_t = \int_0^t \mathrm{F}(\bm{V_s})\dr s + \bm{B}_t,
\]
\noindent
where $(\bm{B}_t)_{t\geq0} = (\int_0^{T_t'}\bm{1}_{\{\mathfrak{X}_s > 0\}}\dr B_s)_{t\geq0}$ is an $(\bm{\mathcal{F}}_t)_{t\geq0}$-Brownian motion. We now deal with $C_t$ and we write
\[
 C_t = \int_0^{T_t'}\bm{1}_{\{s\in\mathcal{I}_{\mathfrak{X}}\}}\dr V_s = \sum_{n\in\mathbb{N}, \tau_n \leq T_t'}(V_{\sigma_n} - V_{\tau_n}) = \sum_{n\in\mathbb{N}, \tau_n \leq T_t'}(M_n - V_{\tau_n}).
\]
\noindent
In the second inequality, we used \eqref{zero_set} and that, by definition, $T_t'\notin\mathring{\mathcal{I}}_{\mathfrak{X}}$. Let us now define for all $n\in\mathbb{N}$, $\bm{\tau}_n = \sum_{k=1}^n(\tau_k - \sigma_{k-1})$. Then by definition of $A_t'$, we have for all $n\in\mathbb{N}$, $A_{\tau_n}' = \bm{\tau}_n$, which leads to $\tau_n = T_{\bm{\tau}_n -}'$. Indeed the flat section of $(A_t')_{t\geq0}$ consists in $\cup_{n\in\mathbb{N}}[\tau_n, \sigma_n]$ and thus the jumping times of $(T_t')_{t\geq0}$ are precisely the times $\bm{\tau}_n$. Then we get
\[
C_t = \sum_{n\in\mathbb{N}}(M_n - \bm{V}_{\bm{\tau}_n-})\bm{1}_{\{\bm{\tau}_n \leq t\}}.
\]

\noindent
Finally it is clear that $\bm{\tau}_1 = \inf\{t> 0, \: \bm{X}_t = 0\}$ and that $\bm{\tau}_{n+1} = \inf\{t> \bm{\tau}_n, \: \bm{X}_t = 0\}$.
\end{proof}

\subsection{Convergence of the time-change}\label{section_time_change}

The goal of this subsection is to see that the time change $(A_t')_{t\geq0}$ is asymptotically equivalent to $t$, i.e. the size of $\sigma_n - \tau_n$ is small compared to the size of $\tau_n -\sigma_{n-1}$. Recall we are concerned with process $(\mathfrak{X}_t)_{t\geq0}$ defined in \eqref{second_cons} and $A_t' = \int_0^t\bm{1}_{\{\mathfrak{X}_s > 0\}} \dr s$. The main result of this subsection is the following result.

\begin{proposition}\label{conv_time_change}
 Under Assumption \ref{assump_measure}-\textit{(i)}, we have $t^{-1}A_t' \overset{\mathbb{P}}{\longrightarrow}1$ as $t\to\infty$.
\end{proposition}

We recall from from Lemma \ref{lemma_random_times} that $(\tau_{n} - \sigma_{n-1}, \sigma_{n} - \sigma_{n-1})_{n\geq 2}$ is a sequence of identically distributed random variables and each element is equal in law to the random variable $(\tau, \sigma)$ defined as follows: let $(\Omega, \mathcal{F}, (\mathcal{F}_t)_{t\geq0}, \mathbb{P})$ be a filtered probability space supporting an $(\mathcal{F}_t)_{t\geq0}$-Brownian motion $(B_t)_{t\geq0}$ and two independent $\mu$-distributed random variables $V_0$ and $M$, also independent of the Brownian motion. Consider the process $(X_t, V_t)_{t\geq0}$ solution of \eqref{eq_non_reflected} starting at $(0, V_0)$. Then $\tau$ and $\sigma$ are defined as
\[
 \tau = \inf\{t>0, \: X_t = 0\} \quad \text{and} \quad\sigma = \inf\{t\geq\tau, \: V_t = M\}.
\]

\noindent
For $t\geq0$, we define $N_t = \sup\{n\geq0, \:\sigma_n \leq t\}$, for which $\sigma_{N_t} \leq t < \sigma_{N_t + 1}$. Then the time-change $(A_t')_{t\geq0}$ satisfies, see \eqref{second_cons},
\begin{equation}\label{eq_time_change}
 A_t' = \sum_{k=1}^{N_t}(\tau_{k} - \sigma_{k-1}) + \tau_{N_t + 1} \wedge t - \sigma_{N_t}.
\end{equation}

\noindent
Roughly, the reason why Proposition \ref{conv_time_change} is true is that $\sigma - \tau$ is actually small compared to $\tau$. More precisely, we will show that $\tau$ and $\sigma$ have exactly the same probability tail.

\medskip
First, since $(X_t)_{t\geq0}$ resembles a symmetric stable process as $t$ is large, we should expect $\mathbb{P}(\tau > t)$ to behave like the probability for a symmetric stable process started at $\eta > 0$, to stay positive up to time $t$, which is well-known to behave like $t^{-1/2}$ as $t\to\infty$.

\medskip
Then, since $(V_t)_{t\geq0}$ is positive recurrent, we should expect $\mathbb{P}(\sigma - \tau > t)$ to have a lighter tail than $\tau$. Indeed, $X_t$ reaches $0$ at time $\tau$ with some random negative velocity $V_{\tau}$, and $\sigma - \tau$ is the amount of time it takes for $(V_t)_{t\geq0}$ to reach $M$, which should not be too big, thanks to the positive recurrence of the velocity and the assumption on $\mu$.

\medskip
However, we did not manage to employ this strategy, as the law of $V_{\tau}$ is unknown and it is not clear at all how to get an exact asymptotic of $\mathbb{P}(\tau > t)$ by approaching $(X_t)_{t\geq0}$ by its scaling limit. We will rather use tools introduced in Berger, Béthencourt and Tardif \cite{persistence_bbt}.

\medskip
We state the following crucial lemma which describes the tails of $\tau$ and $\sigma$. Its proof is rather technical and a bit independent of the rest, so it is postponed to Subsection \ref{section_persistence}.
\begin{lemma}\label{persistence_result}
Grant Assumptions \ref{assump_measure}-\textit{(i)}. Then there exists a constant $C >0$ such that
\[
 \mathbb{P}(\tau > t) \sim \mathbb{P}(\sigma > t) \sim Ct^{-1/2}\quad \hbox{as $t\to\infty$}.
\]
\end{lemma}

With these results at hand, we are able to prove Proposition \ref{conv_time_change}.

\begin{proof}[Proof of Proposition \ref{conv_time_change}]
We seek to show that $(1-t^{-1}A_t')$ converges to $0$ in probability as $t\to\infty$. From \eqref{eq_time_change}, we have
\begin{align*}
 t- A_t' = & \sum_{k=1}^{N_t}(\sigma_{k-1} - \tau_k) + (t - \tau_{N_t + 1})\bm{1}_{\{\tau_{N_t + 1} \leq t\}} + \sigma_{N_t} \\
  = & \sum_{k=1}^{N_t}(\sigma_{k-1} - \tau_k) + (t - \tau_{N_t + 1})\bm{1}_{\{\tau_{N_t + 1} \leq t\}} + \sum_{k=1}^{N_t}(\sigma_k - \sigma_{k-1}) \\
  = & \sum_{k=1}^{N_t}(\sigma_k - \tau_k) + (t - \tau_{N_t + 1})\bm{1}_{\{\tau_{N_t + 1} \leq t\}}.
\end{align*}

\noindent
Since by definition, $\sigma_{N_t + 1} \geq t$, we have the following bound:
\[
 0 \leq 1 - \frac{A_t'}{t} \leq \frac{1}{t}\sum_{k=1}^{N_t + 1}(\sigma_k - \tau_k) = \frac{\sigma_1-\tau_1}{t} + \frac{1}{t}\sum_{k=2}^{N_t + 1}(\sigma_k - \tau_k).
\]

\noindent
Obviously, the first term on the right-hand side almost surely vanishes as $t\to\infty$. We now use Lemma \ref{persistence_result} to study the asymptotic behavior of $\sum_{k=2}^{N_t + 1}(\sigma_k - \tau_k)$ as $t\to\infty$. We divide the rest of the proof in two steps.

\smallskip\noindent
\textit{Step 1:} We first show that
\begin{equation}\label{N_t_limsup}
 \lim_{A\to\infty}\limsup_{t\to\infty} \mathbb{P}(N_t \geq At^{1/2}) = 0.
\end{equation}
To show this, we introduce $\tilde{N}_t = \min\{n \geq 1, \: \sum_{k=1}^n(\sigma_{2k} - \sigma_{2k - 1}) \leq t\}$. Since for any $n\geq 1$, we have $\sigma_{2n} \geq \sum_{k=1}^n(\sigma_{2k} - \sigma_{2k - 1})$, it should be clear that for any $t > 0$,  $N_{t} \leq \tilde{N}_t / 2$. Next, by Lemma \ref{lemma_random_times}, $(\sigma_{2n} - \sigma_{2n - 1})_{n\in\mathbb{N}}$ is a sequence of i.i.d. random variables whose common law is that of $\sigma$. Thanks to Lemma \ref{persistence_result}, we can apply the classical $\alpha$-stable central limit theorem, and it is clear that
\[
 \frac{1}{n^2}\sum_{k=1}^n(\sigma_{2k} - \sigma_{2k-1}) \overset{\mathcal{L}}{\longrightarrow}\mathcal{S}^{1/2}\qquad \hbox{as $n\to\infty$},
\]
where $\mathcal{S}^{1/2}$ is a positive $1/2$-stable random variable, see \cite[Chapter XII.6 Theorem 2]{feller1967introduction}. Then it is immediate that $\tilde{N}_t / t^{1/2}$ converges in law to some random variable $N_\infty$. Finally, we get that
\[
 \limsup_{t\to\infty} \mathbb{P}(N_{t} \geq At^{1/2}) \leq \limsup_{t\to\infty} \mathbb{P}(\tilde{N}_{t} \geq 2At^{1/2}) = \mathbb{P}(N_\infty \geq 2A).
\]
Letting $A\to\infty$ shows that \eqref{N_t_limsup} holds.

\smallskip\noindent
\textit{Step 2:} We show that $t^{-1}\sum_{k=1}^{N_t + 1}(\sigma_k - \tau_k)$ converges to $0$ as $t\to\infty$ in probability. First, for any $n\geq 1$, we can write
\begin{equation}\label{two_sums}
 \sum_{k=1}^n (\sigma_k - \tau_k) = \sum_{k \leq n,\: k\text{ even }} (\sigma_k - \tau_k) + \sum_{k \leq n,\: k\text{ odd }} (\sigma_k - \tau_k),
\end{equation}
which is then by Lemma \ref{lemma_random_times} the sum of two sums of i.i.d. random variables distributed as $\sigma - \tau$. Moreover, by Lemma \ref{persistence_result} and Lemma \ref{equiv_tail} in Appendix \ref{section_invert_time} with $X=\tau$ and $Y = \sigma - \tau$, the tail of $\sigma - \tau$ is lighter than the tail of $\sigma$:
\[
 \lim_{t\to\infty}t^{1/2}\mathbb{P}(\sigma - \tau > t) = 0.
\]
This entails, see Proposition \ref{conv_prob_0_neg} with $\alpha = 1/2$, that the two terms in the right-hand-side of \eqref{two_sums} divided by $n^2$, converges to $0$ as $n\to\infty$ in probability. At this point, we conclude that
\begin{equation}\label{pouet_2}
 \frac{1}{n^2}\sum_{k=2}^{n}(\sigma_k - \tau_k) \overset{\mathbb{P}}{\longrightarrow} 0\qquad \hbox{as $n\to\infty$}.
\end{equation}
Finally, for any $\eta > 0$ and any $A>0$, we have
\[
 \mathbb{P}\bigg(t^{-1}\sum_{k=1}^{N_t + 1}(\sigma_k - \tau_k)> \eta\bigg) \leq \mathbb{P}\bigg(t^{-1}\sum_{k=1}^{\lfloor At^{1/2}\rfloor}(\sigma_k - \tau_k)> \eta\bigg) + \mathbb{P}(N_t + 1 \geq At^{1/2}).
\]
Making $t\to\infty$ using \eqref{pouet_2}, then $A\to \infty$ using \eqref{N_t_limsup}, we complete the step.
\end{proof}

\subsection{Scaling limit of $(\mathfrak{X}_t)_{t\geq0}$}

In this subsection, we show that under Assumption \ref{assump_measure}-\textit{(ii)}, the scaling limit of $(\mathfrak{X}_t)_{t\geq0}$ is the stable process reflected on its infimum. The following proposition will help us showing the second part of Theorem \ref{main_theorem}.

\begin{proposition}\label{conv_strong_assump}
Grant Assumptions \ref{assump_fournier} and \ref{assump_measure}-\textit{(ii)}, and let $(\mathfrak{X}_t, \mathfrak{V}_t)_{t\geq0}$ be defined by \eqref{second_cons} with $v_0 > 0$. Then we have, in law for the $\bm{\mathrm{M}}_1$-topology,
\[
 (\epsilon^{1/\alpha}\mathfrak{X}_{t/\epsilon})_{t\geq0} \longrightarrow \left(R_t^\alpha\right)_{t\geq0} \quad \hbox{as $\epsilon\to0$},
\]
\noindent
where $R_t^{\alpha} = Z_t^{\alpha} - \inf_{s\in[0,t]}Z_s^{\alpha}$ and $(Z_t^{\alpha})_{t\geq0}$ is the stable process from Theorem \ref{main_theorem}.
\end{proposition}

\begin{proof}
We consider, for $v_0 > 0$, the processes $(\mathcal{X}_t, \mathcal{V}_t)_{t\geq0}$ and $(\mathfrak{X}_t, \mathfrak{V}_t)_{t\geq0}$ defined in Sections \ref{section_inelastic} and \ref{reflected_section}, starting at $(0, v_0)$ and both constructed from the process $(X_t, V_t)_{t\geq0}$ solution of \eqref{eq_non_reflected} also starting at $(0, v_0)$. We recall that almost surely, for all $t\geq0$,
\[
 \mathcal{X}_t = \int_0^t V_s\bm{1}_{\{\mathcal{X}_s > 0\}} \dr s \quad \text{and} \quad \mathfrak{X}_t = \int_0^t V_s\bm{1}_{\{\mathfrak{X}_s > 0\}} \dr s,
\]

\noindent
see \eqref{pouet}, \eqref{second_cons} and \eqref{x_integ_v}. We show that, under Assumption \ref{assump_measure}-\textit{(ii)}, we have for any $T >0$,
\[
  \Delta_{T,\epsilon} := \sup_{t\in[0,T]}\epsilon^{1/\alpha}(\mathcal{X}_{t/\epsilon} - \mathfrak{X}_{t/\epsilon}) \overset{\mathbb{P}}{\longrightarrow}0\quad \hbox{as $\epsilon\to0$},
\]

\noindent
which will prove the result thanks to Theorem \ref{cv_free_ref} and \cite[Section 3, Theorem 3.1]{billingsley_conv}. Since a.s. for any $t\geq0$, we have $\mathcal{X}_t \geq \mathfrak{X}_t$ by Proposition \ref{second_comparison}, we can write
\[
 0 \leq \mathcal{X}_t - \mathfrak{X}_t = \int_0^t V_s\bm{1}_{\{\mathcal{X}_s > 0\}\cap\{\mathfrak{X}_s = 0\}} \dr s \leq \int_0^t (0\vee V_s)\bm{1}_{\{\mathfrak{X}_s = 0\}}\dr s \leq \sum_{k=1}^{N_t +1}M_k(\sigma_k - \rho_k),
\]

\noindent
where for any $k\geq1$, $\rho_k = \inf\{t\geq \tau_k, \: V_t = 0\} \leq \sigma_k$ and $N_t = \sup\{n\geq0, \:\sigma_n \leq t\}$, as in the previous subsection. Indeed, for any $t\geq0$, we have 
\[
 \mathcal{I}_{\mathfrak{X}} \cap [0, t] = \bigcup_{k=1}^{N_t + 1}[\tau_k, \sigma_k],
\]

\noindent
and for any $k\geq1$, the velocity $V_s$ is non-positive on $[\tau_k, \rho_k]$ and is smaller than $M_k$ on $[\rho_k, \sigma_k]$. The sequence $(M_n(\sigma_n - \rho_n))_{n\in\mathbb{N}}$ is a sequence of i.i.d random variables and we claim that there exists $\delta' \in (0, 1 - \alpha / 2)$ such that  $\mathbb{E}[(M_1(\sigma_1 - \rho_1))^{\alpha / 2 + \delta'}] < \infty$, see Lemma \ref{tail_t0}-\textit{(ii)} below. This implies by the Markov inequality that $t^{\alpha / 2}\mathbb{P}(M_1(\sigma_1 - \rho_1) > t) \to 0$ as $t\to\infty$. Therefore, by Proposition \ref{conv_prob_0_neg}, we have
\begin{equation}\label{prob_0_mk}
 \frac{1}{n^{2/\alpha}}\sum_{k=1}^n M_k(\sigma_k - \rho_k)\overset{\mathbb{P}}{\longrightarrow} 0\qquad \hbox{as $n\to\infty$}.
\end{equation}
But since we have
\[
 \Delta_{T,\epsilon} \leq \epsilon^{1/\alpha}\sum_{k=1}^{N_{T/\epsilon} + 1}M_k(\sigma_k - \rho_k),
\]
it comes that for any $\eta > 0$ and any $A >0$,
\[
 \mathbb{P}(\Delta_{T,\epsilon} > \eta) \leq \mathbb{P}\bigg(\epsilon^{1/\alpha}\sum_{k=1}^{\lfloor A\epsilon^{-1/2}\rfloor} M_k(\sigma_k - \rho_k)> \eta\bigg) + \mathbb{P}(N_{T/\epsilon} + 1 \geq A\epsilon^{-1/2}).
\]
Letting $\epsilon\to 0$ using \eqref{prob_0_mk}, and letting $A\to\infty$ using \eqref{N_t_limsup}, we conclude that $\mathbb{P}(\Delta_{T,\epsilon} > \eta) \to 0$ as $\epsilon\to0$.
\end{proof}

\subsection{Proof of the main result}\label{section_proof_main}

\begin{proof}[Proof of Theorem \ref{main_theorem}]
\textit{Step 1:} We start by showing the convergence in the finite dimensional sense. Let $(\bm{X}_t, \bm{V}_t)_{t\geq0}$ and $(X_t, V_t)_{t\geq0}$ be solutions of \eqref{reflected_equation} and \eqref{eq_non_reflected} starting at $(0, v_0)$. Let also $(\mathcal{X}_t)_{t\geq0}$ be defined as in \eqref{proc_ref_inf}. Let $n\geq1$, $t_1, \dots, t_n >0$ and $x_1, \dots, x_n \geq0$. By Proposition \ref{first_compar}, we have
\[
 \mathbb{P}\left(\epsilon^{1/\alpha}\bm{X}_{t_1 / \epsilon} \geq x_1, \dots, \epsilon^{1/\alpha}\bm{X}_{t_n / \epsilon} \geq x_n\right) \geq \mathbb{P}\left(\epsilon^{1/\alpha}\mathcal{X}_{t_1 / \epsilon} \geq x_1, \dots, \epsilon^{1/\alpha}\mathcal{X}_{t_n / \epsilon} \geq x_n\right),
\]

\noindent
from which we deduce, by Theorem \ref{cv_free_ref}, that
\[
 \liminf_{\epsilon\to0}\mathbb{P}\left(\epsilon^{1/\alpha}\bm{X}_{t_1 / \epsilon} \geq x_1, \dots, \epsilon^{1/\alpha}\bm{X}_{t_n / \epsilon} \geq x_n\right) \geq \mathbb{P}\left(R_{t_1}^{\alpha} \geq x_1, \dots, R_{t_n}^{\alpha} \geq x_n\right).
\]

Let us now consider the process $(\mathfrak{X}_t, \mathfrak{V}_t)_{t\geq0}$ starting at $(0, v_0)$, recall \eqref{second_cons}, built from the process $(X_t, V_t)_{t\geq0}$, as well as the time change $(A_t')_{t\geq0} = (\int_0^t\bm{1}_{\{\mathfrak{X}_s > 0\}} \dr s)_{t\geq0}$  and its right-continuous inverse $(T_t')_{t\geq0}$. Then by Theorem \ref{thm_second_cons}, the process $(\bm{X}_t, \bm{V}_t)_{t\geq0} = (\mathfrak{X}_{T_t'}, V_{T_t'})_{t\geq0}$ is a solution of \eqref{reflected_equation}. By Proposition \ref{conv_time_change}, we have 
\begin{equation}\label{eq_conv_prob_T}
 (\epsilon T_{t_1 / \epsilon}', \dots, \epsilon T_{t_n / \epsilon}') \overset{\mathbb{P}}{\longrightarrow}(t_1, \dots, t_n) \quad \hbox{as $\epsilon\to0$}.
\end{equation}

\noindent
Let $\delta > 0$ such that for every $k\in\{1,\cdots,n\}$, $\delta < t_k$. We introduce the events
\[
 A_{\epsilon} = \bigcap_{k=1}^n\left\{\epsilon^{1/\alpha}\bm{X}_{t_k / \epsilon} \geq x_k\right\}, \qquad B_{\epsilon, \delta} = \bigcap_{k=1}^n\left\{|\epsilon T_{t_k / \epsilon}' - t_k| \leq \delta\right\},
\]

\noindent
and
\[
 C_{\epsilon, \delta} = \bigcap_{k=1}^n\Big\{\sup_{s\in[t_k - \delta, t_k + \delta]}\epsilon^{1/\alpha}\mathfrak{X}_{s / \epsilon} \geq x_k\Big\} \quad \text{and} \quad D_{\epsilon, \delta} = \bigcap_{k=1}^n\Big\{\sup_{s\in[t_k - \delta, t_k + \delta]}\epsilon^{1/\alpha}\mathcal{X}_{s / \epsilon} \geq x_k\Big\}.
\]

\noindent
By \eqref{eq_conv_prob_T}, it is clear that $\mathbb{P}(B_{\epsilon, \delta}^c) \to 0$ as $\epsilon\to 0$. We easily see that $A_\epsilon \cap B_{\epsilon, \delta} \subset C_{\epsilon, \delta}$, and $C_{\epsilon, \delta} \subset D_{\epsilon, \delta}$ by Proposition \ref{second_comparison}. Therefore $\mathbb{P}(A_\epsilon \cap B_{\epsilon, \delta}) \leq \mathbb{P}(C_{\epsilon, \delta}) \leq \mathbb{P}(D_{\epsilon, \delta})$. As a consequence, we have $\mathbb{P}(A_\epsilon) \leq \mathbb{P}(D_{\epsilon, \delta}) + \mathbb{P}(B_{\epsilon, \delta}^c)$, whence
\[
 \limsup_{\epsilon\to0}\mathbb{P}(A_\epsilon) \leq \limsup_{\epsilon\to0}\mathbb{P}(D_{\epsilon, \delta}).
\]

\noindent
We claim that the following convergence holds
\begin{equation}\label{conv_set_d}
 \lim_{\epsilon\to0}\mathbb{P}(D_{\epsilon, \delta}) = \mathbb{P}(D_{\delta})\quad \text{where}\quad D_{\delta} = \bigcap_{k=1}^n\Big\{\sup_{s\in[t_k - \delta, t_k + \delta]}R_{s}^{\alpha} \geq x_k\Big\}.
\end{equation}

\noindent
Since a.s., $t_1, \dots, t_n$ are not jumping times of $(R_t^{\alpha})_{t\geq0}$, $\mathbb{P}(D_{\delta}) \to\mathbb{P}(R_{t_1}^{\alpha} \geq x_1, \dots, R_{t_n}^{\alpha} \geq x_n)$ as $\delta\to0$. Hence \eqref{conv_set_d} would imply
\[
 \limsup_{\epsilon\to0}\mathbb{P}\left(\epsilon^{1/\alpha}\bm{X}_{t_1 / \epsilon} \geq x_1, \dots, \epsilon^{1/\alpha}\bm{X}_{t_n / \epsilon} \geq x_n\right) \leq \mathbb{P}\left(R_{t_1}^{\alpha} \geq x_1, \dots, R_{t_n}^{\alpha} \geq x_n\right),
\]

\noindent
which would achieve the first step. We now show that \eqref{conv_set_d} holds.

\medskip
By Theorem \ref{cv_free_ref} and Skorokhod's representation theorem, there exist a family of processes $(\mathcal{X}_t^{\epsilon})_{t\geq0}$ indexed by $\epsilon > 0$, and a reflected symmetric stable process $(R_t^\alpha)_{t\geq0}$, both defined on the probability space $([0,1], \mathcal{B}([0,1]), \lambda)$, where $\lambda$ denotes the Lebesgue measure on $[0,1]$, such that for every $\epsilon > 0$, $(\mathcal{X}_t^{\epsilon})_{t\geq0}\overset{d}{=} (\mathcal{X}_{t/\epsilon})_{t\geq0}$ and such that,
\[
 \lambda-\text{a.s.,} \quad d_{\bm{\mathrm{M}}_1}((\epsilon^{1/\alpha}\mathcal{X}_t^{\epsilon})_{t\geq0}, (R_t^\alpha)_{t\geq0}) \underset{\epsilon\to0}{\longrightarrow}0.
\]

\noindent
We now use the fact the $\bm{\mathrm{M}}_2$-topology, originally introduced by Skorokhod in his seminal paper \cite{skorohod}, is weaker than the $\bm{\mathrm{M}}_1$-topology. The convergence in $\mathcal{D}$ endowed with $\bm{\mathrm{M}}_2$ can be characterized, see \cite[page 267]{skorohod}, as follows: a sequence $(x_n)_{n\in\mathbb{N}}$ converges to $x$ in $\mathcal{D}$ endowed with the $\bm{\mathrm{M}}_2$ if and only if
\[
  \inf_{u\in[s,t]}x_n(u) \underset{n\to\infty}{\longrightarrow} \inf_{u\in[s,t]}x(u)\quad \text{and}\quad \sup_{u\in[s,t]}x_n(u) \underset{n\to\infty}{\longrightarrow} \sup_{u\in[s,t]}x(u)
 \]
\noindent
for any $0 \leq s<t$ points of continuity of $x$. Therefore, since $\lambda-$a.s., for any $k\in\{1,\dots,n\}$, $t_k - \delta$ and $t_k + \delta$ are points of continuity of $(R_t^\alpha)_{t\geq0}$, we deduce that $\lambda-$a.s., for any $k\in\{1,\dots,n\}$, the following convergence holds
\[
 \sup_{s\in[t_k - \delta, t_k + \delta]}\epsilon^{1/\alpha}\mathcal{X}_s^{\epsilon} \underset{\epsilon\to0}{\longrightarrow}\sup_{s\in[t_k - \delta, t_k + \delta]}R_{s}^{\alpha}, \quad \hbox{as $\epsilon\to0$},
\]

\noindent
which implies \eqref{conv_set_d}.

\bigskip\noindent
\textit{Step 2:} We now grant Assumption \ref{assump_measure}-\textit{(ii)} and we seek to show that, under this assumption, $(\epsilon^{1/\alpha}\bm{X}_{t/ \epsilon})_{t\geq0}$ is tight for the $\bm{\mathrm{M}}_1$-topology. Here we use the representation $(\bm{X}_t, \bm{V}_t)_{t\geq0} = (\mathfrak{X}_{T_t'}, \mathfrak{V}_{T_t'})_{t\geq0}$, see Step 1. Our goal is to show that the conditions of Theorem \ref{thm_conv_m1}-\textit{(ii)} are satisfied. By the first step, it suffices to show that for any $\eta > 0$, for any $T > 0$,
\begin{equation}\label{tightness_second}
 \lim_{\delta\to0}\limsup_{\epsilon\to0}\mathbb{P}\left(w(\epsilon^{1/\alpha}\bm{X}_{t/\epsilon}, T, \delta) > \eta\right) = 0,
\end{equation}

\noindent
By Proposition \ref{conv_strong_assump}, the process $(\epsilon^{1/\alpha}\mathfrak{X}_{t/ \epsilon})_{t\geq0}$ is tight and by Proposition \ref{conv_time_change} and Dini's theorem, we have the following convergence in probability:
\[
 \sup_{t\in[0,T]}\left|\epsilon T_{t/\epsilon}' - t\right| \overset{\mathbb{P}}{\longrightarrow} 0, \quad \hbox{as $\epsilon\to0$}.
\]

\noindent
Let $T>0$ and $\delta \in (0,1)$, we will first place ourselves on the event $A_{T,\delta}^{\epsilon} = \{\sup_{t\in[0,T]}|\epsilon T_{t/\epsilon}' - t|< \delta\}$. On this event, $T_{t/\epsilon}' \in ((t-\delta) / \epsilon,(t + \delta) / \epsilon)$ for any $t\in[0,T]$. Let $t\in[0,T]$ and recall that $t_{\delta-} = 0 \vee (t-\delta)$ and $t_{\delta+} = T \wedge (t+\delta)$. We set $t_{2\delta-} = 0 \vee (t-2\delta)$ and (abusively) $t_{2\delta+} = (T +1)\wedge (t+2\delta)$. Then on the event $A_{T,\delta}^{\epsilon}$, we have
\[
 \sup_{t_{\delta-}\leq t_1 < t_2 < t_3 \leq t_{\delta+}}d(\mathfrak{X}_{T_{t_2 / \epsilon}'}, [\mathfrak{X}_{T_{t_1 / \epsilon}'}, \mathfrak{X}_{T_{t_3 / \epsilon}'}]) \leq \sup_{t_{2\delta-}\leq t_1 < t_2 < t_3 \leq t_{2\delta+}}d(\mathfrak{X}_{t_2 / \epsilon}, [\mathfrak{X}_{t_1 / \epsilon}, \mathfrak{X}_{t_3 / \epsilon}]),
\]

\noindent
from which we deduce that
\[
 w(\epsilon^{1/\alpha}\mathfrak{X}_{T_{t/\epsilon}'}, T, \delta) \leq w(\epsilon^{1/\alpha}\mathfrak{X}_{t/\epsilon}, T+1, \delta)
\]

\noindent
on the event $A_{T,\delta}^{\epsilon}$. As a consequence, we have for any $\eta, \delta >0$ and for any $T > 0$
\[
 \mathbb{P}\left(w(\epsilon^{1/\alpha}\bm{X}_{t/\epsilon}, T, \delta) > \eta\right) \leq \mathbb{P}\left(w(\epsilon^{1/\alpha}\mathfrak{X}_{t/\epsilon}, T+1, \delta) > \eta\right) + \mathbb{P}\left(\sup_{t\in[0,T]}\left|\epsilon T_{t/\epsilon}' - t\right| \geq \delta\right).
\]

\noindent
Therefore, since $(\epsilon^{1/\alpha}\mathfrak{X}_{t/ \epsilon})_{t\geq0}$ is tight, we have for any $\eta > 0$, for any $T > 0$
\[
 \lim_{\delta\to0}\limsup_{\epsilon\to0}\mathbb{P}\left(w(\epsilon^{1/\alpha}\bm{X}_{t/\epsilon}, T, \delta) > \eta\right) = 0,
\]

\noindent
which completes the proof.
\end{proof}

\subsection{Some persistence problems}\label{section_persistence}

The aim of this subsection is to prove Lemma \ref{persistence_result}. We recall that the random variable $(\tau, \sigma)$ are defined as follows: let $(\Omega, \mathcal{F}, (\mathcal{F}_t)_{t\geq0}, \mathbb{P})$ be a filtered probability space supporting an $(\mathcal{F}_t)_{t\geq0}$-Brownian motion $(B_t)_{t\geq0}$ and two independent $\mu$-distributed random variables $V_0$ and $M$, also independent of the Brownian motion. Consider the process $(X_t, V_t)_{t\geq0}$ solution of \eqref{eq_non_reflected} starting at $(0, V_0)$. During the whole subsection, $\alpha = (\beta +1) / 3$. Then $\tau$ and $\sigma$ are defined as
\[
 \tau = \inf\{t>0, \: X_t = 0\} \quad \text{and} \quad\sigma = \inf\{t\geq\tau, \: V_t = M\}.
\] We first introduce the random times
\[
 T_0 = \inf\{t \geq0,\: V_t =0\} \quad \text{and} \quad \rho = \inf\{t \geq\tau,\: V_t =0\}.
\]

\noindent
It should be clear that a.s., $T_0 \leq \tau \leq \rho \leq \sigma$, see Figure \ref{fig:decomp_random_times}. Indeed since $V_t$ is strictly positive on $(0, T_0)$, so is $X_t$ and $T_0 \leq \tau$. Moreover, since $V_{\tau}$ is non-positive, $M > 0$ and $(V_t)_{t\geq0}$ is continuous, $\rho \leq \sigma$. Lemma \ref{persistence_result} heavily relies on the two following lemmas.

\begin{figure}
\begin{center}
\includegraphics[scale=0.9]{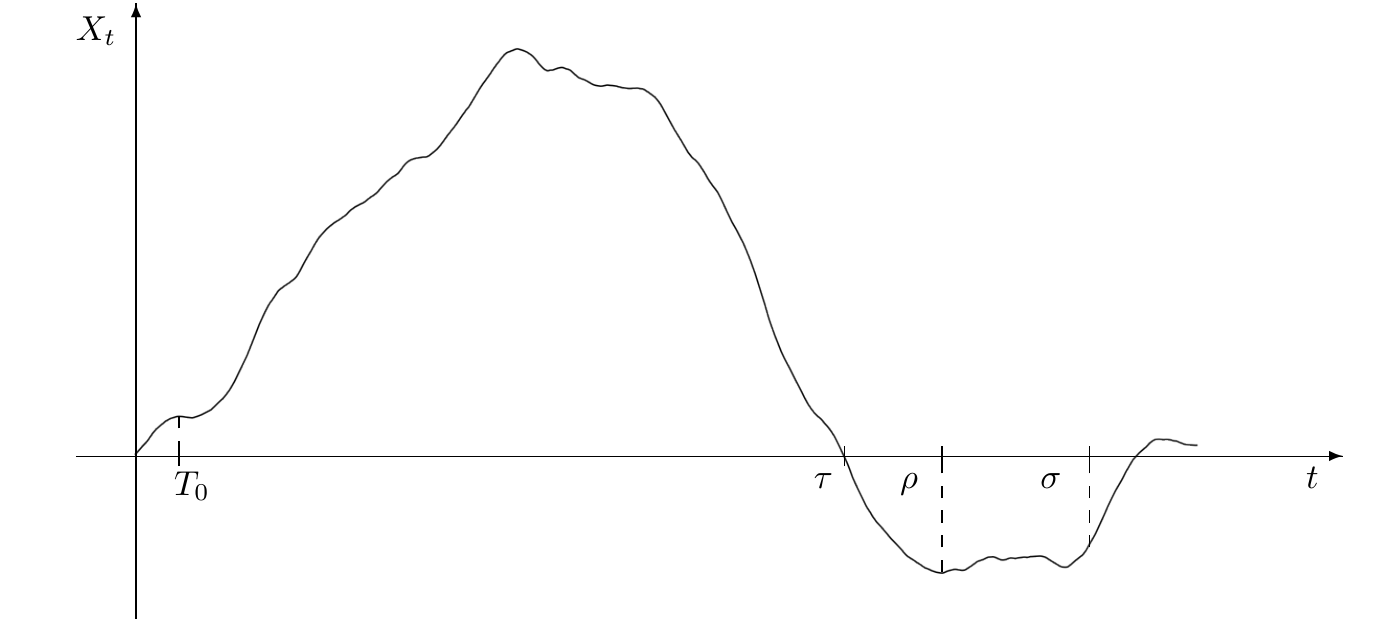}
\end{center}
\caption{\footnotesize Graphical representation of the random times $T_0$, $\tau$, $\rho$ and $\sigma$.}
\label{fig:decomp_random_times}
\end{figure}

\begin{lemma}\label{tail_t0}
\noindent
\begin{enumerate}[label=(\roman*)]
 \item Under Assumption \ref{assump_measure}-\textit{(i)}, we can find some $\delta \in (0, 1/2)$ such that $\mathbb{E}[T_0^{1/2 + \delta}] < \infty$ and $\mathbb{E}[(\sigma - \rho)^{1/2 + \delta}] < \infty$. As a consequence, we have
 \[
 \lim_{t\to\infty}t^{1/2}\mathbb{P}(T_0 > t) = \lim_{t\to\infty}t^{1/2}\mathbb{P}(\sigma - \rho > t) = 0.
\]
\item Under Assumption \ref{assump_measure}-\textit{(ii)}, there exists $\delta'\in(0, 1/\alpha - 2)$ such that $\mathbb{E}[(M(\sigma - \rho))^{\alpha / 2 + \delta'}] < \infty$.
\end{enumerate}
\end{lemma}

\begin{lemma}\label{tail_square_root}
 Grant Assumptions \ref{assump_measure}-\textit{(i)}. Then there exists a constant $C >0$ such that
\[
 \mathbb{P}(\tau - T_0 > t) \sim \mathbb{P}(\rho - T_0 > t) \sim Ct^{-1/2}\quad \hbox{as $t\to\infty$}.
\]
\end{lemma}

\begin{proof}[Proof of Lemma \ref{persistence_result}]
 Using Lemma \ref{equiv_tail} in Appendix \ref{section_invert_time} with $X=\tau-T_0$ and $Y = T_0$ and Lemmas \ref{tail_t0}-\textit{(i)} and \ref{tail_square_root}, it comes
\[
 \mathbb{P}(\tau > t) \sim Ct^{-1/2}\quad \hbox{as $t\to\infty$}.
\]

\noindent
By Lemma \ref{equiv_tail} again with $X = \rho - T_0$ and $Y=\sigma - \rho + T_0$, we get
\[
 \mathbb{P}(\sigma > t) \sim Ct^{-1/2}\quad \hbox{as $t\to\infty$},
\]

\noindent
which completes the proof.
\end{proof}

We now seek to show Lemma \ref{tail_t0}. To do so, we will use Feller's representation of regular diffusions i.e. we represent the velocity process through its scale function $\sca$ and its speed measure $m$:
$$\sca(v) = \int_0^v \Theta^{-\beta}(u)\dr u \quad \text{and}\quad m(v) = \Theta^{\beta}(v).$$

\noindent
Remember from Assumption \ref{assump_fournier} that $\Theta:\mathbb{R}\to(0,\infty)$ is a $C^1$ even function such that $\mathrm{F} = \frac{\beta}{2}\frac{\Theta '}{\Theta}$ and satisfying $\lim_{v\to\pm\infty}|v|\Theta(v) = 1$. The function $\sca$ is an increasing bijection from $\mathbb{R}$ to $\mathbb{R}$ and we denote by $\sca^{-1}$ its inverse. We also define the function $\psi = \sca' \circ \sca^{-1}$. Now consider another Brownian motion $(W_t)_{t\geq0}$ on $(\Omega, \mathcal{F}, (\mathcal{F}_t)_{t\geq0}, \mathbb{P})$ (or an enlargment of the space) and for $v\in\mathbb{R}$, we set
$$W_t^v = \sca(v) + W_t, \quad A_t^v = \int_0^t [\psi(W_s^{v})]^{-2}\dr s \quad \text{as well as} \quad \rho_t^v = \inf\{s>0, \: A_s^v > t\}.$$

\noindent
Then the process defined by $V_t^v = \sca^{-1}(W_{\rho_t^v}^{v})$ and $X_t^v = \int_0^t V_s^v \dr s$ is a solution of \eqref{eq_non_reflected} starting at $(v,0)$. This result is standard and we refer to Kallenberg \cite[Chapter 23, Theorem 23.1 and its proof]{kallenberg2006foundations} for more details. As a consequence, using the substitution $u = \rho_s^v$, we can write that almost surely, for any $t\geq0$,
\begin{equation}\label{feller_rep_x}
    X_t^v = \int_0^t \sca^{-1}(W_{\rho_s^v}^{v}) \dr s = \int_0^{\rho_t^v}\phi\left(W_{s}^{v}\right) \dr s,
\end{equation}

\noindent
where for any $v\in\mathbb{R}$, $\phi(v) = \sca^{-1}(v) / \psi^2(v)$. We set, for any $v\in\mathbb{R}$, $T_0^{v} = \inf\{t\geq0, \:V_t^{v} = 0\}$, the first hitting time of $(V_t^v)_{t\geq0}$ at the level $0$. We also note that since $\sca(v)\to\pm\infty$ as $v\to\pm\infty$ and $\int_{\mathbb{R}}m(v)\dr v < \infty$ because $\beta > 1$, the process $(V_t^{v})_{t\geq0}$ is a positive recurrent diffusion. Note that Assumption \ref{assump_fournier} yields the following asymptotics:
\begin{equation}\label{nnn}
 m(v) \sim v^{-\beta}, \quad \sca(v) \sim (\beta + 1)^{-1}v^{\beta + 1}, \quad \phi(v) \sim (\beta + 1)^{\frac{1-2\beta}{\beta +1}}v^{\frac{1-2\beta}{\beta +1}} \quad \text{as} \:\:v\to\infty.
\end{equation}

\medskip
Finally, we stress that by the strong Markov property, since $V_0$ is $\mu$-distributed, for any non negative functional $G$ of continuous functions, we have
\begin{equation}\label{markov_prop}
    \mathbb{E}\left[G((V_t)_{t\geq0})\right] = \int_0^{\infty}\mathbb{E}\left[G((V_t^v)_{t\geq0})\right]\mu(\dr v).
\end{equation}

\begin{proof}[Proof of Lemma \ref{tail_t0}]
\textit{Step 1: }We first deal with $T_0$. Let $\delta \in (0, 1/2)$, then by \eqref{markov_prop} and the Hölder inequality, we have
$$\mathbb{E}\left[T_0^{1/2 + \delta}\right] = \int_0^{\infty}\mathbb{E}\left[\left(T_0^v\right)^{1/2 + \delta}\right]\mu(\dr v) \leq \int_0^{\infty}\mathbb{E}\left[T_0^v\right]^{1/2 + \delta}\mu(\dr v).$$

We will show that the quantity on the right-hand-side is finite if $\delta$ is small enough and to do so, we need to understand the behavior of $\mathbb{E}[T_0^v]$ as $v$ tends to infinity. We use Kac's moment formula (see for instance Löcherbach \cite[Corollary 3.5]{locherbach2013ergodicity}) which, applied to our case, tells us that
\begin{equation}\label{kacs_moment_formula}
    \mathbb{E}\left[T_0^v\right] = \sca(v)\int_{v}^{\infty}m(u)\dr u + \int_0^v \sca(u)m(u)\dr u \leq \sca(v)\int_{\mathbb{R}}m(u)\dr u + \int_0^v \sca(u)m(u)\dr u.
\end{equation}

\noindent
By \eqref{nnn}, we deduce that $\int_0^v \sca(u)m(u)\dr u \sim(2(\beta + 1))^{-1}v^2$ as $v\to\infty$. Since $\sca(v) \sim (\beta + 1)^{-1}v^{\beta + 1}$ and $\beta + 1 > 2$, the dominant term on the right-hand-side in \eqref{kacs_moment_formula} is the first one and we deduce that there exists some positive constant $K$ such that for any $v\geq0$, $\mathbb{E}\left[T_0^v\right] \leq K(1 + v^{\beta + 1})$. Hence we have
\[
 \mathbb{E}\left[T_0^{1/2 + \delta}\right] \leq \int_0^{\infty}\mathbb{E}\left[T_0^v\right]^{1/2 + \delta}\mu(\dr v) \leq K^{1/2 + \delta}\int_0^{\infty}(1 + v^{\beta + 1})^{1/2 + \delta}\mu(\dr v),
\]

\noindent
which is finite by Assumption \ref{assump_measure}-\textit{(i)} if $\delta = \eta / (\beta + 1)$.

\bigskip\noindent
\textit{Step 2: }Note that the random time $\sigma - \rho$ only depends on the speed process. More precisely, since $\rho$ is a stopping time and $V_{\rho} = 0$, the process $(V_{t + \rho})_{t\geq0}$ is a solution of $\dr Y_t = F(Y_t)\dr t + \dr B_t$ starting at $0$, and thus, $\sigma - \rho$ is equal in law to $T_M^0 = \inf\{t\geq0, \: V_t^0 = M\}$. Since $M$ is independent of $(V_t)_{t\geq0}$, we can write, as in the first step, that
\[
 \mathbb{E}\left[(\sigma - \rho)^{1/2 + \delta}\right] = \int_0^{\infty}\mathbb{E}\left[\left(T_v^0\right)^{1/2 + \delta}\right]\mu(\dr v) \leq \int_0^{\infty}\mathbb{E}\left[T_v^0\right]^{1/2 + \delta}\mu(\dr v),
\]

\noindent
where $T_v^0 = \inf\{t\geq0, \: V_t^0 = v\}$. We use again Kac's moment formula, see \cite[Corollary 3.5]{locherbach2013ergodicity}:
\begin{equation}\label{kacs_moment_formula_2}
 \mathbb{E}\left[T_v^0\right] = \sca(v)\int_{\mathbb{R}}m(u)\dr u - \sca(v)\int_v^{\infty}m(u)\dr u - \int_0^v \sca(u)m(u)\dr u \leq \sca(v)\int_{\mathbb{R}}m(u)\dr u,
\end{equation}

\noindent
and we can conclude as in the previous step that $\mathbb{E}\left[T_v^0\right] \leq K(1 + v^{\beta + 1})$, whence $\mathbb{E}[(\sigma - \rho)^{1/2 + \delta}]$ is finite with the same value of $\delta$.

\bigskip\noindent
\textit{Step 3: } We deal with the second item and we grant Assumption \ref{assump_measure}-\textit{(ii)}. Let $\delta'\in(0, 1 - \alpha /2)$. Since $M$ is independent of $(V_t)_{t\geq0}$ and $\mu$ is the law of $M$, we get
\[
 \mathbb{E}\left[(M(\sigma - \rho))^{\alpha / 2 + \delta'}\right] = \int_0^\infty v^{\alpha / 2 + \delta'}\mathbb{E}\left[\left(T_v^0\right)^{\alpha / 2+ \delta'}\right]\mu(\dr v) \leq \int_0^\infty v^{\alpha / 2 + \delta'}\mathbb{E}\left[T_v^0\right]^{\alpha / 2+ \delta'}\mu(\dr v).
\]

\noindent
Then, since $\mathbb{E}\left[T_v^0\right] \leq K(1 + v^{\beta + 1})$ as in Step 2,
\[
 \mathbb{E}\left[(M(\sigma - \rho))^{\alpha / 2 + \delta'}\right] \leq K^{\alpha/2 + \delta'}\int_0^\infty v^{\alpha / 2 + \delta'}(1+v^{\beta + 1})^{\alpha / 2 + \delta'}\mu(\dr v)
\]

\noindent
which is finite by Assumption \ref{assump_measure}-\textit{(ii)} if $\delta' = \eta / (\beta +2)$, recall that $\alpha = (\beta + 1) /3$.
\end{proof}

\noindent
We now seek to show Lemma \ref{tail_square_root}, and we will need to estimate the moments of $X_{T_0}$.

\begin{lemma}\label{lemma_power_x}
Grant Assumption \ref{assump_measure}-\textit{(i)}, then we have $\mathbb{E}\left[X_{T_0}^{\alpha /2}\right] < \infty$.
\end{lemma}

\begin{proof}
 For any $v > 0$, we set $\theta^v_0 = \inf\{t\geq0, \:W_t^{v} = 0\}$. Since $\sca(v) \geq 0$ if and only if $v\geq0$, and $V_t^v = \sca^{-1}(W_{\rho_t^v}^v)$ with $(\rho_t^v)_{t\geq0}$ the inverse of $(A_t^v)_{t\geq0}$, it should be clear that $\mathbb{P}(T_0^{v} = A_{\theta_0^v}^v) = 1$. Hence, following \eqref{feller_rep_x}, we have for any $v >0$, almost surely
$$X_{T_0^v}^v = \int_0^{\theta_0^v}\phi\left(W_{s}^{v}\right) \dr s.$$

\noindent
\textit{Step 1: }We first introduce the non-negative random variable
$$Z_{v,\alpha} = \int_0^{\theta_0^v}(W_{s}^{v})^{1/\alpha - 2} \dr s.$$

\medskip
By the scaling of the Brownian motion, the law of $(W_t^v, \: 0\leq t \leq \theta_0^v)$ is the same as the law of $(\sca(v) W_{t/[\sca(v)]^2}^{\sca^{-1}(1)}, \: 0\leq t \leq [\sca(v)]^2 \theta_0^{\sca^{-1}(1)})$, and as a consequence,
$$\text{the law of } Z_{v,\alpha} \text{ is the same as the law of }[\sca(v)]^{1/\alpha} Z_{\sca^{-1}(1),\alpha}.$$

\noindent
In fact, the law of $Z_{\sca^{-1}(1),\alpha}$ is explicit and it holds that $Z_{\sca^{-1}(1),\alpha}$ has the same law as $\alpha ^2 / \bm{\Gamma}_\alpha$ where $\bm{\Gamma}_\alpha$ is a random variable whose law is the Gamma distribution of parameter $(\alpha, 1)$, see for instance Letemplier-Simon \cite[page 93]{simon_letemplier}. In particular, we have $\mathbb{P}(Z_{\sca^{-1}(1),\alpha} \geq x) \sim cx^{-\alpha}$ as $x \to \infty$ for some $c >0$. Therefore, it holds that $\mathbb{E}[Z_{\sca^{-1}(1),\alpha}^{\alpha / 2}] < \infty$.

\medskip\noindent
\textit{Step 2: }We conclude. Since $\alpha = (\beta+1)/3$, we have $1/\alpha - 2 = (1-2\beta) / (\beta + 1)$ and by \eqref{nnn}, there exists a constant $c >0$ such that $\phi(v)\sim c v^{1/\alpha - 2}$ as $v\to\infty$. Remark that $1/\alpha - 2 < 0$ as $\beta > 1$, and therefore there exists a constant $C >0$ such that for any $v \geq 0$, $\phi(v) \leq C v^{1/\alpha - 2}$. As a consequence, we have for any $v > 0$, almost surely $0 < X_{T_0^v}^v < CZ_{v,\alpha}$. Putting the pieces together,  by \eqref{markov_prop} and the previous step, we can write
\[
 \mathbb{E}\left[X_{T_0}^{\alpha /2}\right] = \int_0^{\infty}\mathbb{E}\Big[(X_{T_0^v}^v)^{\alpha /2}\Big]\mu(\dr v) \leq C^{\alpha / 2}\int_0^{\infty}\mathbb{E}\Big[(Z_{v,\alpha})^{\alpha /2}\Big]\mu(\dr v)
\]

\noindent
whence
\[
 \mathbb{E}\left[X_{T_0}^{\alpha /2}\right] \leq C^{\alpha / 2}\mathbb{E}[Z_{\sca^{-1}(1),\alpha}^{\alpha /2}]\int_0^{\infty}[\sca(v)]^{1/2}\mu(\dr v).
\]

\noindent
We can conclude since there exists a constant $K > 0$ such that for any $v > 0$, $\sca(v) \leq K(1+v)^{\beta+1}$, see the proof of Lemma \ref{tail_t0}, and $\int_0^{\infty}(1+v)^{(\beta+1) / 2}\mu(\dr v) < \infty$ by Assumption \ref{assump_measure}-\textit{(i)}.
\end{proof}

We now introduce the process $(\Bar{V}_t, \Bar{X}_t)_{t\geq0} = (V_{t+T_0}, X_{t+T_0} - X_{T_0})_{t\geq0}$ which is a solution of \eqref{eq_non_reflected} starting at $(0,0)$ and is independent of $X_{T_0}$. We emphasize that, since the restoring force $\mathrm{F}$ is odd by assumption, the processes $(\Bar{V}_t)_{t\geq0}$ and $(\Bar{X}_t)_{t\geq0}$ are symmetric. We also stress that the stopping times $\tau - T_0 = \inf\{t>0, \:\Bar{X}_t = - X_{T_0}\}$ and $\rho - T_0 = \inf\{t\geq\tau-T_0, \:\Bar{V}_t = 0\}$ only depend on $(\Bar{V}_t)_{t\geq0}$, $(\Bar{X}_t)_{t\geq0}$ and $X_{T_0}$. Let us introduce the supremum and infimum of $(\Bar{X}_t)_{t\geq0}$: $\xi_t = \sup_{s\in[0,t]}\Bar{X}_s$ and $\Lambda_t = \inf_{s\in[0,t]}\Bar{X}_s$. Let us also define for $t\geq0$
\[
 g_t = \sup\{s\leq t, \Bar{V}_s = 0\} \qquad \text{and}\qquad d_t = \inf\{s\geq t, \Bar{V}_s = 0\}.
\]

\noindent
Then we have the following inclusions of events:
\[
 \left\{\Lambda_{d_t} > -X_{T_0}\right\} \subset \left\{\tau - T_0 > t\right\} \subset \left\{\rho - T_0> t\right\} \subset \left\{\Lambda_{g_t} > -X_{T_0}\right\}.
\]

\noindent
The first two inclusions are straightforward since $\{\tau - T_0 > t\} = \{\Lambda_t > - X_{T_0}\}$, $d_t \geq t$ and $\tau \leq \rho$. Remember that $\rho-T_0$ is the first zero of $\Bar{V}_t$ after $\tau - T_0$ and thus it should be clear that $\{\Lambda_{g_t} \leq -X_{T_0}\} = \{\tau - T_0 \leq g_t\} \subset \{\rho - T_0 \leq g_t\}$, which establishes the third inclusion since $g_t \leq t$. Since $(\Bar{X}_t)_{t\geq0}$ is symmetric, we then have
\begin{equation}\label{ineq_prob}
 \mathbb{P}\left(\xi_{d_t} < X_{T_0}\right) \leq \mathbb{P}\left(\tau - T_0 > t\right) \leq \mathbb{P}\left(\rho - T_0 > t\right) \leq \mathbb{P}\left(\xi_{g_t} < X_{T_0}\right).
\end{equation}

We will show the following lemma which, combined with \eqref{ineq_prob}, immediately implies Lemma \ref{tail_square_root}.

\begin{lemma}\label{equi_dt_gt}
Grant Assumptions \ref{assump_measure}-\textit{(i)}. Then there exists a constant $C > 0$ such that
\[
 \mathbb{P}\left(\xi_{g_t} < X_{T_0}\right)\underset{t\to\infty}{\sim}\mathbb{P}\left(\xi_{d_t} < X_{T_0}\right)\underset{t\to\infty}{\sim}Ct^{-1/2}.
\]
\end{lemma}

\noindent
To prove this result, we will heavily use some results developped in \cite{persistence_bbt}, which rely on works about Itô's excursion theory and the links with Lévy processes, in particular by Bertoin \cite{bertoin1996levy} and Vallois, Salminen and Yor \cite{salminen2007excursion}. The proof is rather long and we will segment it (again) into smaller pieces, see Lemmas \ref{ddd} and \ref{renewal_equiv} below.

\medskip
We use the following standard trick: let $e = e(q)$ be an exponential random variable of parameter $q >0$, independent of everything else, then we will look at the quantities $\mathbb{P}(\xi_{g_e} < X_{T_0})$ and $\mathbb{P}(\xi_{d_e} < X_{T_0})$ instead of looking at $\mathbb{P}(\xi_{g_t} < X_{T_0})$ and $\mathbb{P}(\xi_{d_t} < X_{T_0})$. We have nothing to loose doing this since a combination of the Tauberian theorem and the monotone density theorem (see Theorem \ref{tt_dmt} below) tells us that having an asymptotic of $\mathbb{P}(\xi_{g_t} < X_{T_0})$ (respectively $\mathbb{P}(\xi_{d_t} < X_{T_0})$) as $t\to\infty$ is equivalent to having an asymptotic of $\mathbb{P}(\xi_{g_e} < X_{T_0})$ (respectively $\mathbb{P}(\xi_{d_e} < X_{T_0})$) as $q\to0$, and we will first study $\mathbb{P}(\xi_{g_e} < x)$ and $\mathbb{P}(\xi_{d_e} < x)$ for a fixed $x>0$.

\medskip
The first reason we do this is that it brings independence between the quantities we are interested in. The second reason is that by doing this, we actually have explicit formulas for some quantities of interest, for instance for $\mathbb{P}(\xi_{g_e}< x)$ and, as we will see, there is a strong link with some Lévy process associated to $(\Bar{X_t}, \Bar{V}_t)_{t\geq0}$ and fluctuation's theory for Lévy processes, see for instance \cite[Chapter VI]{bertoin1996levy}. 

\medskip
The velocity process $(\Bar{V}_t)_{t\geq0}$ possesses a local time at $0$ and we will denote by $(\gamma_t)_{t\geq0}$ its right-continuous inverse. The latter is a subordinator and we will denote by $\Phi$ its Laplace exponent, i.e. $\mathbb{E}[e^{-q\gamma_t}] = \exp(-t\Phi(q))$. The process $(\Bar{V}_t)_{t\geq0}$ being positive recurrent, we have $\mathbb{E}\left[\gamma_1\right]<\infty$, see \cite[Chapter 2, page 22]{bertoin_subordinators}, and we choose to normalize the local time so that $\mathbb{E}\left[\gamma_1\right] = 1$, whence $\Phi(q) \sim q$ as $q\to0$. The strong law of large number for subordinators entails that a.s., $t^{-1}\gamma_t \to 1$ as $t\to\infty$. We will first prove the following which tells us that we only need to study $\mathbb{P}(\xi_{g_e}< x)$.
\begin{lemma}\label{lemma_ineq}
 There exists a function $f:(0,\infty) \to [0,1]$ such that $f(q) \to 1$ as $q\to0$ and such that for any $q,x > 0$,
 \begin{equation}\label{ineq_proba}
    \mathbb{P}(\xi_{g_e}< x)f(q) \leq \mathbb{P}(\xi_{d_e}< x) \leq \mathbb{P}(\xi_{g_e}< x).
\end{equation}
\end{lemma}

\begin{proof}
 Let us define for any $t > 0$ the processes
\[
 I_t = \int_{g_t}^{d_t}\Bar{V}_s \dr s = \Bar{X}_{d_t} - \Bar{X}_{g_t} \quad \text{and} \quad \Delta_t = I_t + \Bar{X}_{g_t} - \xi_{g_t} = \Bar{X}_{d_t} - \xi_{g_t}.
\]

\noindent
Then we can express $\xi_{d_t}$ in terms of $\xi_{g_t}$ and $\Delta_t$: remarking that $(\Bar{X}_t)_{t\geq}$ is monotonic on every excursion of $(\Bar{V}_t)_{t\geq0}$, we get that $\sup_{s\in[g_t, d_t]}\Bar{X}_s = \Bar{X}_{g_t}\vee \Bar{X}_{d_t}$. Therefore, if $\Delta_t \leq 0$, then $\Bar{X}_{d_t} \leq \xi_{g_t}$ and $\xi_{d_t} = \xi_{g_t}$. On the other hand, if $\Delta_t > 0$, then $\xi_{d_t} = \Bar{X}_{d_t} = \xi_{g_t} + \Delta_t$. All in all, we have
\begin{equation}\label{fff}
 \xi_{d_t} = \xi_{g_t}\bm{1}_{\{\Delta_t \leq 0\}} + (\xi_{g_t} + \Delta_t)\bm{1}_{\{\Delta_t > 0\}}.
\end{equation}

\medskip
We can fatorize functionals of the trajectories of $(\Bar{V}_t)_{t\geq0}$ before time $g_e$ and functionals of the trajectories between $g_e$ and $d_e$. More precisely, it is shown in \cite[Theorem 9]{salminen2007excursion} that the processes $(\Bar{V}_u)_{0\leq u\leq g_e}$ and $(\Bar{V}_{u+g_e})_{0\leq u \leq d_e - g_e}$ are independent. Therefore, $I_e$ is independent of $(\xi_{g_e}, \Bar{X}_{g_e} - \xi_{g_e})$. Moreover, it is shown in \cite[Corollary 4.6]{persistence_bbt} that $\xi_{g_e}$ and $\Bar{X}_{g_e} - \xi_{g_e}$ are i.i.d., see also Lemma \ref{coro_pers} below. Hence the random variables $I_e$, $\xi_{g_e}$ and $\Bar{X}_{g_e} - \xi_{g_e}$ are mutually independent and thus $\xi_{g_e}$ is independent of $\Delta_e = I_e + \Bar{X}_{g_e} - \xi_{g_e}$. We set $f(q) = \mathbb{P}(\Delta_{e}\leq 0)$ and we deduce from \eqref{fff} that \eqref{ineq_proba} holds for any $x>0$.

\medskip
Let us now show that $f(q) \to 1$ as $q\to0$ and let us denote by $n$ the excursion measure of $(\Bar{V}_t)_{t\geq0}$ away from zero. Let $\mathcal{E}$ the set of excursions, i.e. the set of continuous functions $\varepsilon = (\varepsilon_t)_{t\geq0}$ such that $\varepsilon_0 = 0$ and such that there exists $\ell(\varepsilon) > 0$ for which $\varepsilon_s \neq 0$ for every $s \in(0,\ell(\varepsilon))$ and $\varepsilon_s = 0$ for every $s \geq \ell(\varepsilon)$. Then by Theorem 9 in \cite{salminen2007excursion}, we have for any measurable bounded function $G:\mathbb{R}\to\mathbb{R}$,
$$\mathbb{E}\left[G(I_e)\right] = \frac{1}{\Phi(q)}\int_{\mathcal{E}}G\Big(\int_0^{\ell(\varepsilon)}\varepsilon_s \dr s\Big)(1 - e^{-q\ell(\varepsilon)})n(\dr \varepsilon).$$

\noindent
But $\Phi(q) \sim q$ as $q\to 0$ and $\int_{\mathcal{E}}\ell(\varepsilon)n(\dr \varepsilon) = 1 < \infty$ since $\mathbb{E}[\gamma_1] = 1$ (by a direct application of the Master formula in the context of excursion theory), we get by dominated convergence that
$$\mathbb{E}\left[G(I_e)\right] \underset{q\to0}{\longrightarrow}\int_{\mathcal{E}}G\Big(\int_0^{\ell(\varepsilon)}\varepsilon_s \dr s\Big)\ell(\varepsilon)n(\dr \varepsilon),$$

\noindent
and thus $I_e$ converges in law as $q\to0$.

\medskip
Remember that $\xi_{g_e} - \Bar{X}_{g_e}$ is equal in law to $\xi_{g_e}$. Moreover $\xi_{g_e}\to\infty$ in probability as $q\to0$ because $g_e$ tends in probability to infinity as $q\to0$ and because $\xi_t = \sup_{s\in[0,t]}\Bar{X}_s$ tends to infinity in probability (Theorem \ref{conv_m2} clearly implies that $t^{-1/\alpha}\xi_t$ converges in law as $t\to\infty$). Since $I_e$ converges in law and $\xi_{g_e} - \Bar{X}_{g_e}$ converges in probability to $\infty$ as $q\to0$, we have that $\Delta_e = I_e - (\xi_{g_e} - \Bar{X}_{g_e})$ converges to $-\infty$ in probability, so that $\mathbb{P}(\Delta_e \leq 0) = f(q)$ converges to $1$ as $q\to0$.
\end{proof}

To handle the quantity $\mathbb{P}(\xi_{g_e}< x)$, we will rely on a Wiener-Hopf factorization of the bivariate Lévy process $(\gamma_t, Z_t)_{t\geq0} = (\gamma_t, \Bar{X}_{\gamma_t})_{t\geq0}$, which is developped in \cite[Appendix A]{persistence_bbt}. Let us denote by $(S_t)_{t\geq0}$ the supremum process of $(Z_t)_{t\geq0}$, i.e. $S_t = \sup_{s\in[0,t]}Z_s$ and let $(R_t)_{t\geq0} = (S_t - Z_t)_{t\geq0}$ be the reflected process, which is a strong Markov process that also possesses a local time at $0$, see \cite[Chapter VI]{bertoin1996levy} and we denote by $(\sigma_t)_{t\geq0}$ its right-continuous inverse. The process $(\sigma_t, \theta_t, H_t)_{t\geq0} = (\sigma_t, \gamma_{\sigma_t}, S_{\sigma_t})_{t\geq0}$ is a trivariate subordinator, see \cite[Lemma A.1 in Appendix A]{persistence_bbt}. Its Laplace exponent is denoted by $\kappa$, i.e. for any $\alpha, \beta, \delta \geq0$,
\[
 \mathbb{E}\left[e^{-\alpha\sigma_t - \beta \theta_t - \delta H_t}\right] = \exp\left(-t\kappa(\alpha, \beta, \delta)\right).
\]

\noindent
Finally we introduce the renewal function $\mathcal{V}$ defined on $[0,\infty)$ by
$$\mathcal{V}(x) = \int_0^{\infty}\mathbb{P}(H_t \leq x) \dr t,$$

\noindent
which is non-decreasing and right-continuous. We state the following lemma, which is borrowed from \cite[Proposition 4.4 and Corollary 4.6]{persistence_bbt}, which we will use several times.
\begin{lemma}\label{coro_pers}
There exists a constant $k >0$ such that for every $\alpha, \beta, \delta \geq0$, we have the following Fristedt formula
\[
 \kappa(\alpha, \beta, \delta) = k \exp\left(\int_0^{\infty}\int_{[0,\infty)\times\mathbb{R}}\frac{e^{-t} - e^{-\alpha t- \beta r - \delta x}}{t}\bm{1}_{\{x\geq0\}}\mathbb{P}(\gamma_t\in \dr r, Z_t\in \dr x)\dr t\right).
\]

\noindent
Recall that $e = e(q)$ is an exponential random variable of parameter $q >0$ independent of everything else. Then for every $\lambda,\mu\geq0$, we have
 \[
  \mathbb{E}\left[e^{-\lambda\xi_{g_e} - \mu(\xi_{g_e} - \Bar{X}_{g_e})}\right] = \frac{\kappa(0,q,0)}{\kappa(0,q,\lambda)} \frac{\kappa(0,q,0)}{\kappa(0,q,\mu)}.
 \]
\noindent
In particular, $\xi_{g_e}$ and $\xi_{g_e} - \Bar{X}_{g_e}$ are independent and have the same law.
\end{lemma}

\noindent
Note that, when applying \cite[Corollary 4.6]{persistence_bbt}, we used the fact that $(Z_t)_{t\geq0}$ is symmetric so that if $(\hat{Z}_t)_{t\geq0} = (-Z_t)_{t\geq0}$ is the dual process, then the corresponding Laplace exponent $\hat{\kappa} = \kappa$. Let us finally introduce two last tools that are key to the proof of our result and that we will use several times. Let us first remind Frullani's identity which holds for every $b\in(0,1)$, see for instance \cite[page 73]{bertoin1996levy}:
\begin{equation}\label{frullani}
 \log b = \int_0^{\infty}\frac{e^{-x} - e^{-b x}}{x}\dr x.
\end{equation}

Let us also remind the following classical theorem, which can be obtained by a careful application of the classical Karamata's Tauberian theorem and the monotone density theorem, see for instance \cite[Theorem 1.7.1 page 37 and Theorem 1.7.2 page 39]{bgt89}.

\begin{theorem}\label{tt_dmt}
 Let $u:\mathbb{R}_+ \to \mathbb{R}_+$ be a monotone function. Let us denote by $\mathcal{L}u$ its Laplace transform i.e. for any $\lambda > 0$, $\mathcal{L}u(\lambda) = \int_0^{\infty}e^{-\lambda x}u(x) \dr x$. Let $\rho > -1$ and $\Gamma$ be the usual Gamma function, then the two following assertions are equivalent.
 \begin{enumerate}[label=(\roman*)]
  \item $u(x) \sim (\rho + 1)x^{\rho}$ as $x\to\infty$.
  \item $\lambda\mathcal{L}u(\lambda) \sim \Gamma(\rho +2)\lambda^{-\rho}$ as $\lambda \to0$.
 \end{enumerate}
\end{theorem} 

\noindent
We have the following lemma.

\begin{lemma}\label{ddd}
Recall that $e = e(q)$ is an exponential random variable of parameter $q >0$ independent of everything else. There exists $k >0$ such that for any $x>0$, we have
\begin{equation}\label{equiv_proba}
    \mathbb{P}(\xi_{g_e}< x) \underset{q\to0}{\sim}\mathbb{P}(\xi_{d_e}< x)\underset{q\to0}{\sim} kq^{1/2}\mathcal{V}(x).
\end{equation}

\noindent
Moreover, there exists a constant $M>0$ such that for any $x>0$, for any $q\in(0,1)$,
\begin{equation}\label{ineq_proba_2}
    \mathbb{P}(\xi_{d_e}< x) \leq \mathbb{P}(\xi_{g_e}< x) \leq Mq^{1/2}\mathcal{V}(x).
\end{equation}
\end{lemma}

\begin{proof}
By Lemma \ref{coro_pers}, for any $q>0$, we have
\[
\lambda\int_0^{\infty}e^{-\lambda x}\mathbb{P}(\xi_{g_e} < x) \dr x = \mathbb{E}\left[e^{-\lambda \xi_{g_e}}\right] = \frac{\kappa(0,q,0)}{\kappa(0,q,\lambda)}.
\]

\noindent
We introduce for any $q>0$ the function $\mathcal{V}_q$ defined on $[0,\infty)$ by
\[
\mathcal{V}_q(x) = \mathbb{E}\left[\int_0^\infty e^{-q\theta_t}\bm{1}_{\{H_t \leq x\}}\dr t\right],
\]

\noindent
which is such that
\[
\lambda\int_0^{\infty}e^{-\lambda x}\mathcal{V}_q(x) \dr x = \mathbb{E}\left[\int_0^{\infty}e^{-q\theta_t - \lambda H_t}\dr t\right] = \int_0^{\infty}e^{-t\kappa(0,q,\lambda)}\dr t = \frac{1}{\kappa(0,q,\lambda)}.
\]

\noindent
Hence, by injectivity of the Laplace transform, we get $\mathbb{P}(\xi_{g_e} < x) = \kappa(0,q,0)\mathcal{V}_q(x)$ for any $x,q>0$. For any $x>0$, $\mathcal{V}_q(x)$ increases to $\mathcal{V}(x)$ as $q\to0$. Hence, by Lemma \ref{lemma_ineq}, to show \eqref{equiv_proba} and \eqref{ineq_proba_2}, it is enough to show that $\kappa(0,q,0)\sim kq^{1/2}$ as $q\to 0$ for some $k>0$.

\medskip
By the Fristedt formula from Lemma \ref{coro_pers} and by symmetry of $(\Bar{V}_t)_{t\geq0}$, and thus of $(Z_t)_{t\geq0}$ (observe $(Z_t)_{t\geq0}$ and $(-Z_t)_{t\geq0}$ share the same $(\gamma_t)_{t\geq0}$), we also have
\[
 \kappa(0,q,0) = k\exp\left(\int_0^{\infty}\int_{[0,\infty)\times\mathbb{R}}\frac{e^{-t} - e^{-q r}}{t}\bm{1}_{\{x\leq0\}}\mathbb{P}(\gamma_t\in \dr r, Z_t\in \dr x)\dr t\right).
\]

\noindent
Then we can write
\begin{align*}
 \log[\kappa(0,q,0)]^2 = & 2\log k + \int_0^{\infty}\int_{[0,\infty)\times\mathbb{R}}\frac{e^{-t} - e^{- qr}}{t}\mathbb{P}(\gamma_t\in \dr r, Z_t\in \dr x)\dr t \\
 = & 2\log k + \int_0^{\infty}\int_{0}^{\infty}\frac{e^{-t} - e^{- qr}}{t}\mathbb{P}(\gamma_t\in \dr r)\dr t \\
 = & 2\log k + \int_0^{\infty}\frac{e^{-t} - e^{- \Phi(q)t}}{t} \dr t \\
 = & 2\log k + \log\Phi(q)
\end{align*}

\noindent
by Frullani's formula \eqref{frullani}. Therefore $\kappa(0,q,0) = k (\Phi(q))^{1/2} \sim k q^{1/2}$ as $q\to0$.
\end{proof}

We need to show one last result, after which we will able to prove Lemma \ref{equi_dt_gt}, which will close this subsection.

\begin{lemma}\label{renewal_equiv}
There exists a constant $v_{\alpha} > 0$ such that $\mathcal{V}(x) \sim v_{\alpha}x^{\alpha / 2}$ as $x\to\infty$.
\end{lemma}

\begin{proof}
Let us first remark that we have for any $\lambda >0$, 
\[
\lambda\int_0^{\infty}e^{-\lambda x}\mathcal{V}(x) \dr y = \mathbb{E}\left[\int_0^{\infty}e^{- \lambda H_t}\dr t\right] = \int_0^{\infty}e^{-t\kappa(0,0,\lambda)}\dr t = \frac{1}{\kappa(0,0,\lambda)}.
\]

\noindent 
Then, by Theorem \ref{tt_dmt}, since $\mathcal{V}$ is non-decreasing and $\alpha / 2 > -1$, it is enough to show that there exists a constant $c_\alpha > 0$ such that $\kappa(0,0,\lambda)  \sim c_{\alpha}\lambda^{\alpha / 2}$ as $\lambda\to0$.

\medskip
To do so, we use the convergence in law of the rescaled Lévy process $(\epsilon^{1/\alpha}Z_{t/\epsilon})_{t\geq0}$ to the symmetric stable process $(Z_t^\alpha)_{t\geq0}$, see Proposition \ref{conv_levy_process} below. Moreover, since $t^{-1}\gamma_t \to 1$ as $t\to\infty$, we get
\begin{equation}\label{zzz}
 (\epsilon \gamma_{t/\epsilon}, \epsilon^{1/\alpha}Z_{t/\epsilon})_{t\geq0} \longrightarrow (t, Z_t^\alpha)_{t\geq0} \qquad \hbox{as $\epsilon\to0$},
\end{equation}

\noindent
in law for the usual $\bm{\mathrm{J}}_1$ Skorokhod topology. Indeed, since $(\gamma_t, Z_t)_{t\geq0}$ is Lévy, only the convergence in law of $(t^{-1}\gamma_t, t^{-1/\alpha}Z_t)$ to $(1, Z_1^\alpha)$ as $t\to\infty$ is required, see Jacod-Shiryaev \cite[Chapter VII, Corollary 3.6]{jacod2013limit}, and the convergence follows from Slutsky's lemma.

\medskip
Let us set, for every $\alpha,\beta,\lambda \geq0$,
\[
 \kappa^{\epsilon}(\alpha, \beta, \lambda) = k \exp\left(\int_0^{\infty}\int_{[0,\infty)\times\mathbb{R}}\frac{e^{-t} - e^{-\alpha t- \beta r - \lambda x}}{t}\bm{1}_{\{x\geq0\}}\mathbb{P}(\epsilon\gamma_{t/\epsilon}\in \dr r, \epsilon^{1/\alpha}Z_{t/\epsilon}\in \dr x)\dr t\right),
\]

\noindent
where $k>0$ is the constant from Lemma \ref{coro_pers}. We have
\begin{align*}
 \log\kappa^\epsilon(0, 0, \lambda) = & \log k +  \int_0^{\infty}t^{-1}\mathbb{E}\left[(e^{-t} - e^{-\lambda\epsilon^{1/\alpha} Z_{t/\epsilon}})\bm{1}_{\{Z_{t/\epsilon} \geq0\}}\right] \dr t\\
 = & \log k + \int_0^{\infty}t^{-1}\mathbb{E}\left[(e^{-t} - e^{-\lambda\epsilon^{1/\alpha} Z_{t}})\bm{1}_{\{Z_{t} \geq0\}}\right] \dr t \\
  & + \int_0^{\infty}t^{-1}(e^{-\epsilon t} - e^{-t})\mathbb{P}(Z_t \geq 0) \dr t \\
 = & \log \kappa(0, 0, \lambda \epsilon^{1/\alpha}) - \frac{1}{2}\log\epsilon.
\end{align*}

\noindent
In the third equality, we used Frullani's identity and the fact that $\mathbb{P}(Z_t \geq 0) = 1/2$. Therefore, we have $\kappa^\epsilon(0, 0, \lambda) = \epsilon^{-1/2}\kappa(0, 0, \lambda \epsilon^{1/\alpha})$. Then it is shown in \cite[Proposition B.2 in Appendix B]{persistence_bbt} that the convergence \eqref{zzz} entails that for any $\alpha,\beta,\lambda\geq0$,
\[
 \kappa^\epsilon(\alpha, \beta, \lambda)\underset{\epsilon\to0}{\longrightarrow}k \exp\left(\int_0^{\infty}\int_0^{\infty}\frac{e^{-t} - e^{-(\alpha +\beta)t - \lambda x}}{t}\mathbb{P}(Z_t^{\alpha}\in \dr x)\dr t\right)=:\Bar{\kappa}(\alpha, \beta, \lambda).
\]

\noindent
Now for $\delta > 0$, choosing $\lambda = 1$ and $\epsilon = \delta^{\alpha}$, we see that $\kappa(0,0,\delta) = \delta^{\alpha /2 }\kappa^{\delta^{\alpha}}(0,0,1)$, which is equivalent to $\delta^{\alpha / 2}\Bar{\kappa}(0,0,1)$, as desired.
\end{proof}

\begin{proof}[Proof of Lemma \ref{equi_dt_gt}]
Let $e = e(q)$ be an independent exponential random variable of parameter $q>0$. Since $t\mapsto\mathbb{P}(\xi_{g_t} < X_{T_0})$ and $t\mapsto\mathbb{P}(\xi_{d_t} < X_{T_0})$ are non-decreasing, since $-1/2 > -1$, and since
\[
 \mathbb{P}(\xi_{g_e} < X_{T_0}) = q\int_0^{\infty}e^{-qt}\mathbb{P}(\xi_{g_t} < X_{T_0})\dr t \quad \text{and}\quad\mathbb{P}(\xi_{d_e} < X_{T_0}) = q\int_0^{\infty}e^{-qt}\mathbb{P}(\xi_{d_t} < X_{T_0})\dr t,
\]

\noindent
by Theorem \ref{tt_dmt}, the result is equivalent to $\mathbb{P}(\xi_{g_e} < X_{T_0}) \sim \mathbb{P}(\xi_{d_e} < X_{T_0}) \sim \Bar{C}q^{1/2}$ as $q\to0$ for some $\Bar{C}>0$. Since $(\Bar{X}_t)_{t\geq0}$ is independent of $X_{T_0}$, we have
$$\mathbb{P}(\xi_{d_e} < X_{T_0}) = \int_0^{\infty}\mathbb{P}(\xi_{d_e} < x)\mathbb{P}(X_{T_0}\in \dr x).$$

\noindent
By \eqref{equiv_proba}, we know that for any $x>0$, $q^{-1/2}\mathbb{P}(\xi_{d_e} < x) \to k\mathcal{V}(x)$ as $q\to 0$, and by \eqref{ineq_proba_2}, we have $q^{-1/2}\mathbb{P}(\xi_{d_e} < x) \leq M\mathcal{V}(x)$. Finally, by Lemmas \ref{lemma_power_x} and \ref{renewal_equiv}, $\mathbb{E}[\mathcal{V}(X_{T_0})] < \infty$ and thus we can apply the dominated convergence theorem, which tells us that
$$\lim_{q\to0}q^{-1/2}\mathbb{P}(\xi_{d_e} < X_{T_0}) = k\mathbb{E}[\mathcal{V}(X_{T_0})].$$
\noindent
The same proof holds for $\mathbb{P}(\xi_{g_e} < X_{T_0})$.
\end{proof}

\section{$\mathrm{M}_1$-convergence of the free process}\label{append_proof_m1}

In this section, we give the proof of Theorem \ref{conv_m2}. Let $(X_t, V_t)_{t\geq0}$ be the solution to \eqref{eq_non_reflected} starting at $(0, v_0)$ where $v_0\in\mathbb{R}$. Let $(Z_t^\alpha)_{t\geq0}$ be the symmetric stable process with $\alpha = (\beta + 1) / 3$ as in the statement. As the convergence in the finite dimensional distribution sense was already proved in \cite{fournier2018one}, we only need to show the tightness for the $\bm{\mathrm{M}}_1$-topology. By Theorem \ref{thm_conv_m1}, we need that for any $T>0$, for any $\eta>0$,
\begin{equation}\label{tightness_criteria}
 \lim_{\delta\to0}\limsup_{\epsilon\to0}\mathbb{P}(w(\epsilon^{1/\alpha}X_{t/\epsilon}, T, \delta) > \eta) = 0,
\end{equation}

The idea of the proof is as follows: we first show that the convergence holds for $(X_t, V_t)_{t\geq0}$ starting at $(0, 0)$. To do so, we will show that some L\'evy process $(Z_t)_{t\geq0}$ associated to $(X_t)_{t\geq0}$ converges in the $\bm{\mathrm{J}}_1$-topology to a symmetric stable process (which is immediately implied by the finite dimensional distribution convergence). Therefore, $(Z_t)_{t\geq0}$ converges also in the $\bm{\mathrm{M}}_1$-topology and \eqref{tightness_criteria} is then satisfied by $(Z_t)_{t\geq0}$, which will imply that \eqref{tightness_criteria} is also satisfied by $(X_t)_{t\geq0}$. We will then extend the convergence to processes starting at $(0, v_0)$.

\subsection{Some preliminary results}

Before jumping in to the proof, we we recall some results that were used / proved in \cite{fournier2018one}. The first result can be found in Jeulin-Yor \cite{jeulin1981distributions} and Biane-Yor \cite{biane1987valeurs}, and represents symmetric stable processes using one Brownian motion.

\begin{theorem}[Biane-Yor]\label{thm_biane_yor}
Let $(W_t)_{t\geq0}$ be a Brownian motion and $(\tau_t)_{t\geq0}$ the inverse of its local time at $0$. Let $\alpha\in(0,2)$ and consider, for $\eta > 0$, the process
$$K_t^{\eta} = \int_0^t\mathrm{sgn}(W_s)|W_s|^{1/\alpha - 2}\bm{1}_{\{|W_s|>\eta\}}\dr s.$$

\noindent
Then the process $(K_t^{\eta})_{t\geq0}$ converges a.s. toward a process $(K_t)_{t\geq0}$ as $\eta \to 0$, uniformly on compact time intervals. Moreover the process $(S_t^\alpha)_{t\geq0} = (K_{\tau_t})_{t\geq0}$ is a symmetric stable process and for all $t\geq0$ and all $\xi\in\mathbb{R}$, $\mathbb{E}[\exp(i\xi S^{\alpha}_t)] = \exp(-\kappa_{\alpha}t|\xi|^{\alpha})$, where $\kappa_{\alpha} = \frac{2^{\alpha}\pi\alpha^{2\alpha}}{2\alpha\Gamma^2(\alpha)\sin(\pi\alpha/2)}$.
\end{theorem}

We now summarize the intermediate results that can be found in \cite[Lemmas 6 and 9]{fournier2018one}, enabling the authors to prove their main result, which is stated in \eqref{result_fournier}. We mention that the main tool used in the proofs is the theory of scale function and speed measure.

\begin{theorem}[Fournier-Tardif]\label{thm_fournier}
Let $(X_t, V_t)_{t\geq0}$ be a solution of \eqref{eq_non_reflected}, with $\beta \in(1,5)$ and starting at $(0,0)$. There exists a Brownian motion $(W_t)_{t\geq0}$ such that for any $\epsilon >0$, there exist a continuous process $(H_t^{\epsilon})_{t\geq0}$ and a continuous, increasing and bijective time-change $(A_t^{\epsilon})_{t\geq0}$ with inverse $(\rho_t^{\epsilon})_{t\geq0}$, adapted to the filtration generated by $(W_t)_{t\geq0}$, and having the following properties:
\begin{enumerate}[label=(\roman*)]
    \item For any $\epsilon > 0$, $(X_{t/\epsilon})_{t\geq0} \overset{d}{=} (H^{\epsilon}_{\rho_t^{\epsilon}})_{t\geq0}$.
    \item For every $t\geq0$, a.s., $\rho_t^{\epsilon}\underset{\epsilon\to0}{\longrightarrow}\tau_t$, where $(\tau_t)_{t\geq0}$ is the inverse of the local time at $0$ of $(W_t)_{t\geq0}$.
    \item Almost surely, for every $t\geq0$, $\sup_{[0,T]}\left|\epsilon^{1/\alpha}H_t^{\epsilon} - \theta K_t\right| \underset{\epsilon\to0}{\longrightarrow}0$, where $\alpha = (\beta +1)  /3$, where $\theta = (\beta  +1)^{1/\alpha - 2}c_\beta^{1/\alpha}$ and where $(K_t)_{t\geq0}$ is the process from Theorem \ref{thm_biane_yor}. 
\end{enumerate}
\end{theorem}

\noindent
Recall that $c_\beta$ is given in \eqref{sigma}. We now slightly improve their result, showing the convergence of past infimum and supremum.

\begin{proposition}\label{conv_inf_sup}
Grant Assumption \ref{assump_fournier} and let $(X_t, V_t)_{t\geq0}$ be a solution of \eqref{eq_non_reflected} with $\beta\in(1,5)$ and starting at $(0,0)$. Let $(Z_t^{\alpha})_{t\geq0}$ be a symmetric stable process with $\alpha = (\beta + 1) / 3$ and such that $\mathbb{E}[e^{i\xi Z_t^{\alpha}}] = \exp(-t\sigma_\alpha |\xi|^{\alpha})$. Then we have for every $0 \leq s < t$,
\[
\inf_{u\in[s,t]}\epsilon^{1/\alpha}X_{u/\epsilon} \overset{d}{\longrightarrow} \inf_{u\in[s,t]}Z_u^\alpha\quad \text{and} \quad \sup_{u\in[s,t]}\epsilon^{1/\alpha}X_{u/\epsilon} \overset{d}{\longrightarrow} \sup_{u\in[s,t]}Z_u^\alpha\quad\hbox{as $\epsilon\to0$}.
\]
\end{proposition}

\begin{proof}
By item \textit{(i)} of Theorem \ref{thm_fournier}, it is enough to show that for any $0 \leq s < t$, a.s., we have
\[
  \inf_{u\in[s,t]}\epsilon^{1/\alpha}H_{\rho_u^{\epsilon}}^{\epsilon}\underset{\epsilon\to0}{\longrightarrow}\inf_{u\in[s,t]}\theta K_{\tau_u} \quad \text{and} \quad \sup_{u\in[s,t]}\epsilon^{1/\alpha}H_{\rho_u^{\epsilon}}^{\epsilon}\underset{\epsilon\to0}{\longrightarrow}\sup_{u\in[s,t]}\theta K_{\tau_u}.
\]
\noindent
We will only show the result for the infimum as the proof for the supremum is identical.

\bigskip\noindent
\textit{Step 1 }: We first show that a.s., for any $s<t$, $\inf_{u\in[s,t]} K_{\tau_u} = \inf_{u\in[\tau_s,\tau_t]} K_u$, which is not straightforward since $t\mapsto\tau_t$ is discontinuous. We will first treat the case $\alpha\in(0,1)$. Observe that in this case, $\int_0^t|W_s|^{1/\alpha - 2}\dr s < \infty$ since $1/\alpha - 2 > -1$ and thus we have
$$K_{\tau_t} = \int_0^{\tau_t}\text{sgn}(W_s)|W_s|^{1/\alpha - 2}\dr s = \sum_{r\leq t} \int_{\tau_{r-}}^{\tau_r}\text{sgn}(W_s)|W_s|^{1/\alpha - 2}\dr s,$$

\noindent
which has finite variations and no drift part. For every $u\geq0$, $(W_t)_{t\geq0}$ is of constant sign on the time-interval $[\tau_{u-}, \tau_u]$ and consequently $t\mapsto K_t$ is monotone on every such interval. Hence, the infimum is necessarily reached at the extremities i.e. $\inf_{r\in[\tau_{u-}, \tau_u]}K_r = \min\{K_{\tau_{u-}}, K_{\tau_u}\}$.

\medskip
If $\alpha\in[1,2)$, we approximate $(K_t)_{t\geq0}$ by the processes $(K_t^{\eta})_{t\geq0}$, from Theorem \ref{thm_biane_yor}. Similarly, we have for every $t\geq0$,  $K_{\tau_t}^{\eta} = \sum_{s\leq t} \int_{\tau_{s-}}^{\tau_s}\text{sgn}(W_u)|W_u|^{1/\alpha - 2}\bm{1}_{\{|W_u|>\eta\}}\dr u$, and thus, by the previous reasoning, we have $\inf_{u\in[s,t]} K_{\tau_u}^{\eta} = \inf_{u\in[\tau_s,\tau_t]} K_u^{\eta}$. We can write
\begin{align*}
 \left|\inf_{u\in[s,t]} K_{\tau_u} - \inf_{u\in[\tau_s,\tau_t]} K_u\right| & \leq \left|\inf_{u\in[s,t]} K_{\tau_u} - \inf_{u\in[s,t]} K_{\tau_u}^{\eta}\right| + \left|\inf_{u\in[\tau_s,\tau_t]} K_u^{\eta} - \inf_{u\in[\tau_s,\tau_t]} K_u\right| \\
  & \leq \sup_{u\in[s,t]}\left|K_{\tau_u}^{\eta} - K_{\tau_u}\right| + \sup_{u\in[\tau_s,\tau_t]}\left|K_{u}^{\eta} - K_{u}\right| \\
  & \leq 2 \sup_{u\in[0,\tau_t]}\left|K_{u}^{\eta} - K_{u}\right|.
\end{align*}

\noindent
By Theorem \ref{thm_biane_yor}, $(K_t^{\eta})_{t\geq0}$ converges a.s. to $(K_t)_{t\geq0}$ uniformly on compact time intervals, and thus the last term vanishes as $\eta \to 0$.

\bigskip\noindent
\textit{Step 2 }: For every $\epsilon > 0$, $t\mapsto\rho_t^{\epsilon}$ and $t\mapsto H_t^{\epsilon}$ are almost surely continuous. Therefore, a.s. for every $s < t$, $\inf_{u\in[s,t]}\epsilon^{1/\alpha}H_{\rho_u^{\epsilon}}^{\epsilon} = \inf_{u\in[\rho_s^{\epsilon},\rho_t^{\epsilon}]}\epsilon^{1/\alpha}H_{u}^{\epsilon}$. Hence we can write, by Step 1,
\begin{align*}
    \left|\inf_{u\in[s,t]}\epsilon^{1/\alpha}H_{\rho_u^{\epsilon}}^{\epsilon} - \inf_{u\in[s,t]}\theta K_{\tau_u}\right| & \leq \left|\inf_{u\in[\rho_s^{\epsilon},\rho_t^{\epsilon}]}\epsilon^{1/\alpha}H_{u}^{\epsilon} - \inf_{u\in[\rho_s^{\epsilon},\rho_t^{\epsilon}]}\theta K_u\right| + \left|\inf_{u\in[\rho_s^{\epsilon},\rho_t^{\epsilon}]}\theta K_u - \inf_{u\in[\tau_s,\tau_t]}\theta K_u\right|\\
     & \leq \sup_{u\in[0,T]}\left|\epsilon^{1/\alpha}H_{u}^{\epsilon} - \theta K_u\right| + \left|\inf_{u\in[\rho_s^{\epsilon},\rho_t^{\epsilon}]}\theta K_u - \inf_{u\in[\tau_s,\tau_t]}\theta K_u\right|,
\end{align*}

\noindent
where $T = \sup_{\epsilon\in(0,1)}\rho_t^{\epsilon}$ is a.s. finite by item \textit{(ii)} of Theorem \ref{thm_fournier}. Almost surely, the first term on the right-hand-side goes to $0$ thanks to Theorem \ref{thm_fournier}-\textit{(iii)}. By item \textit{(iii)}, we have for any $0 \leq s < t$, almost surely, $\rho_s^{\epsilon} \to \tau_s$ and $\rho_t^{\epsilon} \to \tau_t$ as $\epsilon\to0$. Since $(K_t)_{t\geq0}$ is continuous, the second term vanishes as $\epsilon\to0$, which completes the proof.
\end{proof}

\subsection{Convergence of the associated L\'evy process}

In this subsection, we consider a solution $(X_t, V_t)_{t\geq0}$ of \eqref{eq_non_reflected} starting at $(0, 0)$. The velocity process possesses a local time at $0$ that we will denote by $(L_t)_{t\geq0}$ and we will also denote by $(\gamma_t)_{t\geq0}$ its right-continuous inverse. The latter is a subordinator. The process $(V_t)_{t\geq0}$ is positive recurrent  which implies that $\mathbb{E}\left[\gamma_1\right]<\infty$ and we choose to normalize the local time so that $\mathbb{E}\left[\gamma_1\right] = 1$. The strong law of large number for subordinators entails that a.s. $t^{-1}\gamma_t \to 1$ as $t\to\infty$. This also implies the same result for $(L_t)_{t\geq0}$, and by Dini theorem, we get that a.s., for any $t\geq0$,
\begin{equation}\label{dini_local_time}
    \sup_{s\in[0, t]}\left|\epsilon L_{s/\epsilon} - s\right| \underset{\epsilon\to0}{\longrightarrow}0
\end{equation}

\noindent
Next we define $(Z_t)_{t\geq0} = (X_{\gamma_t})_{t\geq0}$ which is a pure jump Levy process with finite variations and should be seen the following way:
$$Z_t = \int_0^{\gamma_t}V_s \dr s = \sum_{s\leq t}\int_{\gamma_s -}^{\gamma_s}V_u \dr u.$$

\noindent
We establish an $\alpha$-stable central limit theorem for the Levy process $(Z_t)_{t\geq0}$, which seems more or less clear in the light of \eqref{result_fournier} and the strong law of large number for $(\gamma_t)_{t\geq0}$.

\begin{proposition}\label{conv_levy_process}
Let $(Z_t^{\alpha})_{t\geq0}$ be the stable process of Proposition \ref{conv_inf_sup}. Then we have
\[
 (\epsilon^{1/\alpha}Z_{t/\epsilon})_{t\geq0} \longrightarrow(Z_t^{\alpha})_{t\geq0} \qquad \hbox{as $\epsilon\to0$}
\]
\noindent
in law in the $\bm{\mathrm{J}}_1$-topology
\end{proposition}

\begin{proof}
 As $(Z_t)_{t\geq0}$ is a L\'evy process, it is enough to show that $t^{-1/\alpha}Z_t$ converges in law to $Z_1^{\alpha}$, see for instance Jacod-Shiryaev \cite[Chapter VII, Corollary 3.6]{jacod2013limit}. Let $z\in\mathbb{R}$ and $\delta > 0$. On the one hand, we have
$$\mathbb{P}\left(t^{-1/\alpha}Z_t \geq z\right) \leq \mathbb{P}\left(t^{-1/\alpha}Z_t \geq z, \:|\gamma_t - t| \leq \delta t\right) + \mathbb{P}\left(|\gamma_t - t| > \delta t\right),$$

\noindent
and on the other hand, we have
$$\mathbb{P}\left(t^{-1/\alpha}Z_t \geq z\right) \geq \mathbb{P}\left(|\gamma_t - t| \leq \delta t\right) - \mathbb{P}\left(t^{-1/\alpha}Z_t < z, \:|\gamma_t - t| \leq \delta t\right).$$

\noindent
It follows from the strong law of large number that $\mathbb{P}(|\gamma_t - t| > \delta t)$ converges to $0$ as $t\to\infty$. Now on the event $\{|\gamma_t - t| \leq \delta t\}$, we have $\gamma_t \in [(1-\delta)t, (1+\delta)t]$, and thus, reminding that $Z_t = X_{\gamma_t}$, we have
$$\mathbb{P}\left(t^{-1/\alpha}Z_t \geq z, \:|\gamma_t - t| \leq \delta t\right) \leq \mathbb{P}\left(\sup_{s\in[1-\delta, 1+\delta]}t^{-1/\alpha}X_{st} \geq z\right),$$

\noindent
and
$$\mathbb{P}\left(t^{-1/\alpha}Z_t < z, \:|\gamma_t - t| \leq \delta t\right) \leq \mathbb{P}\left(\inf_{s\in[1-\delta, 1+\delta]}t^{-1/\alpha}X_{st} < z\right).$$

\noindent
The two quantities on the right-hand-side of the above equations converge by Proposition \ref{conv_inf_sup} to $\mathbb{P}(\sup_{s\in[1-\delta, 1+\delta]}Z_{s}^{\alpha} \geq z)$ and $\mathbb{P}(\inf_{s\in[1-\delta, 1+\delta]}Z_{s}^{\alpha} < z)$ as $t$ tends to infinity. Putting the pieces together, we have
$$\limsup_{t\to\infty}\mathbb{P}\left(t^{-1/\alpha}Z_t \geq z\right) \leq \mathbb{P}\left(\sup_{s\in[1-\delta, 1+\delta]}Z_{s}^{\alpha} \geq z\right).$$

\noindent
and
$$\liminf_{t\to\infty}\mathbb{P}\left(t^{-1/\alpha}Z_t \geq z\right) \geq \mathbb{P}\left(\inf_{s\in[1-\delta, 1+\delta]}Z_{s}^{\alpha} \geq z\right),$$

\noindent
Since almost surely, $1$ is not a jumping time of $(Z_t^{\alpha})_{t\geq0}$, it should be clear that by letting $\delta\to 0$, we can conclude that $\lim_{t\to\infty}\mathbb{P}(t^{-1/\alpha}Z_t \geq z) = \mathbb{P}(Z_1^{\alpha} \geq z)$.
\end{proof}

\subsection{Proof of Theorem \ref{conv_m2}}

\begin{proof}[Proof of Theorem \ref{conv_m2}]
\textit{Step 1}: As explained above, we start by showing that Theorem \ref{conv_m2} holds for a solution $(X_t, V_t)_{t\geq0}$ starting at $(0,0)$. We need to show that \eqref{tightness_criteria} holds. Let $(Z_t)_{t\geq0} = (X_{\gamma_t})_{t\geq0}$ as in the previous subsection, where $(\gamma_t)_{t\geq0}$ is the inverse of the local time $(L_t)_{t\geq0}$ at $0$ of $(V_t)_{t\geq0}$. Since the $\bm{\mathrm{M}}_1$-topology is weaker than the $\bm{\mathrm{J}}_1$-topology, by Proposition \ref{conv_levy_process} and Theorem \ref{thm_conv_m1}, we get, for any $T>0$, for any $\eta>0$,
\[
 \lim_{\delta\to0}\limsup_{\epsilon\to0}\mathbb{P}(w(\epsilon^{1/\alpha}Z_{t/\epsilon}, T, \delta) > \eta) = 0.
\]

\noindent
Now we show that for any $T>0$, for any $\eta>0$ and any $\delta\in(0,1)$,
\begin{equation}\label{bbb}
 \mathbb{P}\left(w(\epsilon^{1/\alpha}X_{t/\epsilon}, T, \delta) > \eta\right) \leq \mathbb{P}\left(w(\epsilon^{1/\alpha}Z_{t/\epsilon}, T + 1, 2\delta) > \eta\right) + \mathbb{P}\left(\sup_{t\in[0, T]}|\epsilon L_{t/\epsilon} - t| \geq \delta\right)
\end{equation}

\noindent
This will achieve the first step by \eqref{dini_local_time}. We first introduce for $t\geq0$,
\[
 g_t = \sup\{s\leq t, V_s = 0\} \qquad \text{and}\qquad d_t = \inf\{s\geq t, V_s = 0\}.
\]

\noindent
Note that $g_t$ and $d_t$ can be expressed in terms of the local time and its inverse, i.e. $g_t = \gamma_{L_t -}$ and $d_t = \gamma_{L_t}$. We also introduce the random function $\nu(t)$ such that $\nu(t) = 1$ if the excursion straddling the time $t$ is positive and $\nu(t) = -1$ if it is negative, i.e. $\nu(t) = \bm{1}_{\{V_t > 0\}} - \bm{1}_{\{V_t < 0\}}$.

\medskip
Let $T >0$ and $\delta \in (0,1)$, we first place ourselves on the event $A_{T,\delta} = \{\sup_{t\in[0, T]}|\epsilon L_{t/\epsilon} - t| < \delta\}$. Let $t\in[0, T]$ and $t_{\delta-}\leq t_1 < t_2 < t_3 \leq t_{\delta+}$, where $t_{\delta-} = 0 \vee (t-\delta)$ and $t_{\delta+} = T \wedge (t+\delta)$. We also introduce $t_{2\delta-} = 0 \vee (t-2\delta)$ and $t_{2\delta+} = (T +1)\wedge (t+2\delta)$. We emphasize that, since we are on the event $A_{T,\delta}$, $d_{t/\epsilon} = \gamma_{L_{t/\epsilon}} \leq \gamma_{(t+\delta) / \epsilon}$ for every $t\in[0,T]$. We first bound the distance $d(\epsilon^{1/ \alpha} X_{t_2 / \epsilon}, [\epsilon^{1/ \alpha} X_{t_1 / \epsilon}, \epsilon^{1/ \alpha} X_{t_3 / \epsilon}]) = \epsilon^{1/ \alpha}d(X_{t_2 / \epsilon}, [X_{t_1 / \epsilon}, X_{t_3 / \epsilon}])$. Without loss of generality, we will assume that $X_{t_1 / \epsilon} \leq X_{t_3 / \epsilon}$.
\begin{itemize}[leftmargin=*]
    \item First case: $X_{t_1 / \epsilon}\leq X_{t_2 / \epsilon} \leq X_{t_3 / \epsilon}$. Then we have
    \[
     d(X_{t_2 / \epsilon}, [X_{t_1 / \epsilon}, X_{t_3 / \epsilon}]) = 0 \leq \sup_{t_{2\delta-}\leq t_1 < t_2 < t_3 \leq t_{2\delta+}} d(Z_{t_2 / \epsilon}, [Z_{t_1 / \epsilon}, Z_{t_3 / \epsilon}]).
    \]
    
    \item Second case: $X_{t_2 / \epsilon}< X_{t_1 / \epsilon} \leq X_{t_3 / \epsilon}$. In this case, $d(X_{t_2 / \epsilon}, [X_{t_1 / \epsilon}, X_{t_3 / \epsilon}]) = X_{t_1 / \epsilon} - X_{t_2 / \epsilon}$. Let us note that, since $(X_t)_{t\geq0}$ is monotonic on every excursion of $(V_t)_{t\geq0}$, $t_1/\epsilon$ and $t_3/\epsilon$ can not belong to the same excursion, i.e. $d_{t_1/\epsilon} \leq g_{t_3/\epsilon}$. We define, for $i\in\{1, 2, 3\}$ and $\epsilon>0$, the positive real numbers $u_{i,\epsilon}$ defined as follows
    \begin{enumerate}
     \item $u_{2,\epsilon} = g_{t_2 / \epsilon}$ if $\nu(t_2 / \epsilon) = 1$, $u_{2,\epsilon} = d_{t_2 / \epsilon}$ if $\nu(t_2 / \epsilon) = -1$ and $u_{2,\epsilon} = t_2 / \epsilon$ if $\nu(t_2 / \epsilon) = 0$. Since $(X_t)_{t\geq0}$ is monotonic on every excursion of $(V_t)_{t\geq0}$, we have $X_{t_2 / \epsilon} \geq X_{u_{2,\epsilon}}$.
     
     \item For $i\in\{1,3\}$, $u_{i,\epsilon} = d_{t_i / \epsilon}$ if $\nu(t_i / \epsilon) = 1$, $u_{i,\epsilon} = g_{t_i / \epsilon}$ if $\nu(t_i / \epsilon) = -1$ and $u_{i,\epsilon} = t_i / \epsilon$ if $\nu(t_i / \epsilon) = -0$. We have $X_{t_i / \epsilon} \leq X_{u_{i,\epsilon}}$.
    \end{enumerate}
    
    \noindent
    Therefore, we have $X_{u_{2,\epsilon}} < X_{u_{1,\epsilon}}$ and $X_{u_{2,\epsilon}} < X_{u_{3,\epsilon}}$ so that necessarily, $u_{1,\epsilon} < u_{2,\epsilon} < u_{3,\epsilon}$. Now we stress that, if $r\geq0$ is such that $V_r=0$ (i.e. such that $\nu(r) = 0$), then since the zero set of $(V_t)_{t\geq0}$ has no isolated points, either $r = d_r$ or $r = g_r$. In any case, we always have for every $i\in\{1, 2, 3\}$, $u_{i,\epsilon} = \gamma_{L_{t_i / \epsilon}}$ or $u_{i,\epsilon} = \gamma_{L_{t_i / \epsilon}-}$. Let $\theta > 0$ and remember that on the event $A_{T,\delta}$, we have $L_{t_i / \epsilon} \in ((t_i -\delta) / \epsilon, (t_i +\delta) / \epsilon)$ for every $i\in\{1, 2, 3\}$. Using the fact that $(\gamma_t)_{t\geq0}$ is increasing and that $(X_t)_{t\geq0}$ is continuous, we can always find $s_{1} < s_{2} < s_{3} \in [t_{2\delta-}, t_{2\delta+}]$ such that
    \[
     X_{t_2/\epsilon} \geq X_{\gamma_{s_2/\epsilon}} - \frac{\theta}{2}, \quad X_{t_1/\epsilon} \leq X_{\gamma_{s_1/\epsilon}} + \frac{\theta}{2} \quad \text{and} \quad X_{t_3/\epsilon} \leq X_{\gamma_{s_3/\epsilon}} + \frac{\theta}{2},
    \]
    which leads to
    \[
      d(X_{t_2 / \epsilon}, [X_{t_1 / \epsilon}, X_{t_3 / \epsilon}]) \leq (Z_{s_1 / \epsilon} -Z_{s_2 / \epsilon}) \wedge (Z_{s_3 / \epsilon} - Z_{s_2 / \epsilon}) + \theta = d(Z_{s_2 / \epsilon}, [Z_{s_1 / \epsilon}, Z_{s_3 / \epsilon}]) + \theta.
     \]
     Since this holds for every $\theta > 0$, we deduce the following bound
     \[
      d(X_{t_2 / \epsilon}, [X_{t_1 / \epsilon}, X_{t_3 / \epsilon}])  \leq \sup_{t_{2\delta-}\leq t_1 < t_2 < t_3 \leq t_{2\delta+}} d(Z_{t_2 / \epsilon}, [Z_{t_1 / \epsilon}, Z_{t_3 / \epsilon}]). 
     \]
    
    \item Third case: $X_{t_1 / \epsilon} \leq X_{t_3 / \epsilon} < X_{t_2 / \epsilon}$. We can adapt the previous case to deduce that
    \[
      d(X_{t_2 / \epsilon}, [X_{t_1 / \epsilon}, X_{t_3 / \epsilon}])  \leq \sup_{t_{2\delta-}\leq t_1 < t_2 < t_3 \leq t_{2\delta+}} d(Z_{t_2 / \epsilon}, [Z_{t_1 / \epsilon}, Z_{t_3 / \epsilon}]). 
     \]
\end{itemize}

\noindent
To summarize, we proved that, on the event $A_{T,\delta}$, the following bound holds
\[
 w(\epsilon^{1/\alpha}X_{t/\epsilon}, T, \delta) \leq w(\epsilon^{1/\alpha}Z_{t/\epsilon}, T + 1, 2\delta).
\]

\noindent
This implies \eqref{bbb}. We proved Theorem \ref{conv_m2} in the case $v_0 = 0$.

\bigskip\noindent
\textit{Step 2:} We consider the solution $(X_t, V_t)_{t\geq0}$ of \eqref{eq_non_reflected} starting at $(0, v_0)$ where $v_0\in\mathbb{R}$. We show that there exists a constant $C > 0$ such that for any $T >0$,
\begin{equation}\label{bound_square_v}
 \mathbb{E}\Big[\sup_{t\in[0,T]}|V_t|\Big] \leq C(1+ T^{1/2}).
\end{equation}

\noindent
To this aim, we start by studying $\sup_{t\in[0,T]}|V_t|^{\beta + 1}$ and we introduce, recall Assumption \ref{assump_fournier}, the even function $\ell$ defined as
\[
 \ell(v) = 2 \int_0^{v}\Theta^{-\beta}(u)\int_0^u\Theta^{\beta}(w)\dr w \dr u,
\]

\noindent
which solves the Poisson equation $\ell'\mathrm{F} + \frac{1}{2}\ell'' = 1$. Then by the Itô formula, we have
\[
 \ell(V_t) = \ell(v_0) + t + \int_0^t\ell'(V_s)\dr B_s.
\]

\noindent
Moreover, remember that $|v|\Theta(v) \to 1$ as $v\to\pm\infty$ and that $\beta > 1$. As a consequence, there exist positive constants $c, c'>0$ such that
\[
 \ell(v)\sim c|v|^{\beta + 1} \quad \text{and} \quad \ell'(v)\sim c'\text{sgn}(v)|v|^{\beta} \quad \text{as }v\to\pm\infty.
\]

\noindent
Therefore, there exist positive constants $M, M' > 0$ such that for all $v\in\mathbb{R}$,
\begin{equation}\label{bound_ell}
 |v|^{\beta + 1} \leq M(1 + \ell(v)) \quad \text{and} \quad [\ell'(v)]^{2} \leq M'(1 + v^{2\beta}).
\end{equation}

\noindent
Using \eqref{bound_ell} and Doob's inequality, we get
\begin{align}
 \notag \mathbb{E}\Big[\sup_{t\in[0,T]}|V_t|^{\beta + 1}\Big] & \leq M\left[1 + \ell(v_0) + T + 4\mathbb{E}\Big[\Big(\int_0^T [\ell'(V_s)]^2 \dr s\Big)^{1/2}\Big]\right] \\
  & \leq M\left[1 + \ell(v_0) + T + 4(M'T)^{1/2}\right] + 4M(M'T)^{1/2}\mathbb{E}\Big[\sup_{t\in[0,T]}|V_t|^{\beta}\Big]. \label{aaa}
\end{align}

\noindent
Finally, using Young's inequality with $p=\beta + 1$ and $q=(\beta + 1) / \beta$, we find some $C > 0$ such that
\begin{align*}
 4M(M'T)^{1/2}\mathbb{E}\Big[\sup_{t\in[0,T]}|V_t|^{\beta}\Big] \leq & C T^{(\beta + 1) / 2} + \frac{1}{2}\mathbb{E}\Big[\sup_{t\in[0,T]}|V_t|^{\beta}\Big]^{(\beta + 1) / \beta} \\
 \leq & C T^{(\beta + 1) / 2} + \frac{1}{2}\mathbb{E}\Big[\sup_{t\in[0,T]}|V_t|^{\beta + 1}\Big],
\end{align*}

\noindent
which, inserted in \eqref{aaa}, implies that there exists a constant $K > 0$ such that
\[
 \mathbb{E}\Big[\sup_{t\in[0,T]}|V_t|^{\beta + 1}\Big] \leq K(1+T^{(\beta + 1)  /2}).
\]

\noindent
We deduce \eqref{bound_square_v} by Hölder's inequality again.

\bigskip\noindent
\textit{Step 3:} We finally show the result for any solution $(X_t, V_t)_{t\geq0}$ starting at $(0, v_0)$, where $v_0\in\mathbb{R}$, and we extend the technique used in \cite[page 21]{fournier2018one}. If we set $T_0 = \inf\{t\geq0, V_t = 0\}$, then the process
\[
(\bar{X}_t, \bar{V}_t)_{t\geq0} = (X_{t+T_0} - X_{T_0}, V_{t+T_0})_{t\geq0}
\]

\noindent
is a solution of \eqref{eq_non_reflected} starting at $(0,0)$. Therefore, by the first step, the process $(\epsilon^{1/\alpha}\bar{X}_{t/\epsilon})_{t\geq0}$ converges to $(Z_t^{\alpha})_{t\geq0}$ in the space $\mathcal{D}$ endowed with the $\bm{\mathrm{M}}_1$-topology. Then, by a version of the Slutsky lemma, see for instance \cite[Section 3, Theorem 3.1]{billingsley_conv}, it is enough to show that for any $T>0$,
\begin{equation}\label{conv_proba_sup}
 \sup_{t\in[0,T]}\epsilon^{1/\alpha}\left|X_{t/\epsilon} - \bar{X}_{t/\epsilon}\right| \overset{\mathbb{P}}{\longrightarrow} 0 \quad \hbox{as $\epsilon\to0$}.
\end{equation}

\noindent
Indeed, the $\bm{\mathrm{M}}_1$-topology is weaker than the topology induced by the uniform convergence on compact time-intervals. We distinguish two cases. First,
\[
 \bm{1}_{\{T_0 \geq t/\epsilon\}}\left|X_{t/\epsilon} - \bar{X}_{t/\epsilon}\right| \leq \bm{1}_{\{T_0 \geq t/\epsilon\}}\left|X_{t/\epsilon}\right| + \bm{1}_{\{T_0 \geq t/\epsilon\}}\left|X_{t/\epsilon + T_0} - X_{T_0}\right| \leq \int_0^{2T_0}|V_s|\dr s = D_1.
\]

\noindent
Second,
\[
 \bm{1}_{\{T_0 < t/\epsilon\}}\left|X_{t/\epsilon} - \bar{X}_{t/\epsilon}\right| \leq \left|X_{T_0}\right| + \bm{1}_{\{T_0 < t/\epsilon\}}\left|X_{t/\epsilon + T_0} - X_{t/\epsilon}\right| \leq D_1 + \bm{1}_{\{T_0 < t/\epsilon\}}\int_{t/\epsilon}^{t/\epsilon + T_0}|V_s|\dr s.
\]

\noindent
Hence, if we set $D_{t,\epsilon}^2 = \bm{1}_{\{T_0 < t/\epsilon\}}\int_{t/\epsilon}^{t/\epsilon + T_0}|V_s|\dr s$, we get
\[
 \sup_{t\in[0,T]}\epsilon^{1/\alpha}\left|X_{t/\epsilon} - \bar{X}_{t/\epsilon}\right| \leq \epsilon^{1/\alpha}D_1 + \sup_{t\in[0,T]}\epsilon^{1/\alpha}D_{t,\epsilon}^2.
\]

\noindent
The first term converges almost surely to $0$ as $\epsilon\to0$. Regarding the second one, we have
\begin{align*}
 \mathbb{E}\Big[\sup_{t\in[0,T]}\epsilon^{1/\alpha}D_{t,\epsilon}^2 \Big| \mathcal{F}_{T_0}\Big] & \leq \epsilon^{1/\alpha}T_0\mathbb{E}\Big[\sup_{t\in[0,T]}\bm{1}_{\{T_0 < t/\epsilon\}}\sup_{s\in[t/\epsilon, t/\epsilon + T_0]}|V_s|\Big| \mathcal{F}_{T_0}\Big] \\
  & \leq \epsilon^{1/\alpha}T_0\bm{1}_{\{T_0 < T/\epsilon\}}\mathbb{E}\Big[\sup_{t\in[0,2T/\epsilon]}|V_s|\Big] \\
  & \leq \epsilon^{1/\alpha}T_0C(1 + (2T)^{1/2}\epsilon^{-1/2})
\end{align*}

\noindent
by Step 2. This last quantity almost surely goes to $0$ as $\alpha\in(0,2)$. This implies the convergence in probability of $\sup_{t\in[0,T]}\epsilon^{1/\alpha}D_{t,\epsilon}^2$  to $0$, so that \eqref{conv_proba_sup} holds.
\end{proof}

\section*{Appendix}

\appendix

\section{Some useful results}\label{section_invert_time}

\begin{lemma}\label{equiv_tail}
Let $X$ and $Y$ be two positive random variables and let $C>0$. Assume that we have $\mathbb{P}(X > t) \sim Ct^{-1/2}$ as $t\to\infty$. Then the following assertions are equivalent.
\begin{enumerate}[label=(\roman*)]
 \item $\lim_{t\to\infty}t^{1/2}\mathbb{P}(Y > t) = 0$.
 
 \item $\mathbb{P}(X + Y > t) \sim Ct^{-1/2}$ as $t\to\infty$.
\end{enumerate}
\end{lemma}

\begin{proof}
Let us first show that \textit{(i)} implies \textit{(ii)}. Let $\delta\in(0,1)$ and write
\begin{align*}
    \mathbb{P}(X + Y > t) = &\: \mathbb{P}(X + Y > t, \:X >\delta t) + \mathbb{P}(X + Y > t,\: X \leq \delta t) \\
     \leq & \: \mathbb{P}(X >\delta t) + \mathbb{P}(Y >(1-\delta) t).
\end{align*}
\noindent
We deduce that
$$C = \liminf_{t\to\infty}t^{1/2}\mathbb{P}(X > t) \leq \liminf_{t\to\infty}t^{1/2}\mathbb{P}(X + Y > t) \leq \limsup_{t\to\infty}t^{1/2}\mathbb{P}(X + Y > t) \leq C\delta^{-1/2}.$$
\noindent
Letting $\delta\to 1$ completes the first step. 

\medskip
We now show that \textit{(ii)} implies \textit{(i)}. We first remark that
$$\{Y > 3t\} \subset \bigcup_{n\in\mathbb{N}}\{X+Y > (n+1)t\}\cap \{X \leq nt\}.$$

\noindent
Indeed, if $Y > 3t$, we set $n+2 =\lfloor\frac{X+Y}{t}\rfloor \geq3$, and get $\frac{X}{t} \leq \frac{X+Y}{t} - 3 \leq n$ and $\frac{X + Y}{t} > n+1$. Now let $N\in\mathbb{N}$, then we have $(\cup_{n>N}\{X+Y > (n+1)t\}\cap \{X \leq nt\}) \subset \{X+Y > (N+2)t\}$ and therefore
$$\mathbb{P}(Y > 3t) \leq \mathbb{P}(X+Y > (N+2)t) + \sum_{n=1}^N\mathbb{P}(X+Y > (n+1)t, \:X \leq nt).$$

\noindent
For any $n \geq 1$, we have
\begin{align*}
    \mathbb{P}(X+Y > (n+1)t, \:X \leq nt) = & \: \mathbb{P}(X+Y > (n+1)t) - \mathbb{P}(X+Y > (n+1)t, \:X > nt) \\
     & \leq \mathbb{P}(X+Y > (n+1)t) - \mathbb{P}(X > (n+1)t),
\end{align*}
\noindent
and thus $\lim_{t\to\infty}t^{1/2}\mathbb{P}(X+Y > (n+1)t, \:X \leq nt) = 0$, from which we deduce that for any $N\in\mathbb{N}$,
$$\limsup_{t\to\infty}t^{1/2}\mathbb{P}(Y > 3t) \leq C(N+2)^{-1/2}.$$

\noindent
Letting $N\to\infty$ completes the proof.
\end{proof}

\begin{proposition}\label{conv_prob_0_neg}
Let $(X_n)_{n\in\mathbb{N}}$ be i.i.d positive random variables and let $\alpha\in(0,1)$ such that $\lim_{t\to\infty}t^{\alpha}\mathbb{P}(X_1 > t) = 0$. If $S_n = \sum_{k=1}^n X_k$, then $n^{-1/\alpha}S_n$ converges to $0$ in probability as $n\to\infty$.
\end{proposition}

\begin{proof}
We will show that the Laplace transform of $n^{-1/\alpha}S_n$ converges to 1 as $n\to\infty$. We have for any $\lambda > 0$,
$$\mathbb{E}\left[e^{-\lambda n^{-1/\alpha}S_n}\right] = \mathbb{E}\left[e^{-\lambda n^{-1/\alpha}X_1}\right]^n,$$

\noindent
and thus
$$\log \mathbb{E}\left[e^{-\lambda n^{-1/\alpha}S_n}\right] \underset{n\to\infty}{\sim}n\mathbb{E}\left[e^{-\lambda n^{-1/\alpha}X_1} - 1\right] = \frac{-\lambda}{n^{1/\alpha - 1}}\int_0^{\infty}e^{-\lambda u / n^{1/\alpha}}\mathbb{P}(X_1 > u) \dr u.$$

\noindent
Then we use that for $u\geq n^{1/\alpha}$, $\mathbb{P}(X_1 > u) \leq \mathbb{P}(X_1 > n^{1/\alpha})$ and we get
$$\frac{\lambda}{n^{1/\alpha - 1}}\int_0^{\infty}e^{-\lambda u / n^{1/\alpha}}\mathbb{P}(X_1 > u) \dr u \leq \frac{\lambda}{n^{1/\alpha - 1}} \int_0^{n^{1/\alpha}}\mathbb{P}(X_1 > u) \dr u + e^{-\lambda}n\mathbb{P}(X_1 > n^{1/\alpha}).$$

\noindent
The second term on the right-hand-side converges to $0$ by assumption. Regarding the first term, since $\mathbb{P}(X_1 > u) = o(u^{-\alpha})$ as $u\to\infty$ and since $\lim_{x\to\infty}\int_1^x u^{-\alpha}\dr = \infty$, we have $\int_0^{n^{1/\alpha}}\mathbb{P}(X_1 > u) \dr u = o(\int_1^{n^{1/\alpha}}u^{-\alpha} \dr u) = o(n^{1/\alpha - 1})$ as $n\to\infty$, which completes the proof.
\end{proof}

\section{On the PDE result}\label{section_pde}

In this subsection, we formalize the P.D.E. result briefly exposed in the introduction. We first explain how the law of $(\bm{X}_t, \bm{V}_t)_{t\geq0}$ is linked with a kinetic Fokker-Planck equation with diffusive boundary conditions. For every $\varphi\in C^2_b(\mathbb{R})$, we define $\mathcal{L}\varphi = \mathrm{F}\varphi' + \frac{1}{2}\varphi''$. Then $\mathcal{L}$ is the infinitesimal generator of the (free) speed process $(V_t)_{t\geq0}$. We also denote by $\mathcal{L}^*$ its adjoint operator which is such that $\mathcal{L}^*\varphi = \frac{1}{2}\varphi'' - [F\varphi]'$.

\begin{proposition}\label{prop_32}
 Let $(\bm{X}_t, \bm{V}_t)_{t\geq0}$ be a solution to \eqref{reflected_equation} starting at $(x_0, v_0)\in((0,\infty)\times\mathbb{R}) \cup (\{0\} \times(0,\infty))$. Let us denote by $f(\dr t, \dr x, \dr v) = \mathbb{P}(\bm{X}_t \in \dr x, \bm{V}_t \in \dr v)\dr t$ which is a measure on $\mathbb{R}_+^2\times \mathbb{R}$. There exist two measures $\nu_- \in \mathcal{M}(\mathbb{R}_+ \times \mathbb{R}_-)$ and $\nu_+\in \mathcal{M}(\mathbb{R}_+^2)$ such that for every $\varphi\in C^{\infty}_c(\mathbb{R}_+^2\times \mathbb{R})$, we have
 \begin{align}\label{mmm}
  \varphi(0, x_0, v_0) + \int_{\mathbb{R}_+^2\times \mathbb{R}}&\left[\partial_t\varphi + v\partial_x\varphi + \mathcal{L}\varphi\right]f(\dr \notag s, \dr x, \dr v) \\
   & + \int_{\mathbb{R}_+^2}\varphi(s,0,v)\nu_+(\dr s, \dr v) - \int_{\mathbb{R}_+\times\mathbb{R}_-}\varphi(s,0,v)\nu_-(\dr s, \dr v) = 0.
 \end{align}

\noindent
Moreover the measures $\nu_-$ and $\nu_+$ satisfy $\nu_+(\dr t, \dr v) = \mu(\dr v)\int_{w \in \mathbb{R}_-}\nu_-(\dr t, \dr w)$.
\end{proposition}

\begin{proof}
Let $\varphi\in C^{\infty}_c(\mathbb{R}_+^2\times \mathbb{R})$, then by Itô's formula, and passing to the expectation, we get for every $T >0$ that
 \begin{align*}
  \mathbb{E}\left[\varphi(T, \bm{X}_T, \bm{V}_T)\right] = & \:\varphi(0, x_0, v_0) + \int_0^T\mathbb{E}\left[(\partial_t\varphi + v\partial_x\varphi + \mathcal{L}\varphi)(s, \bm{X}_s, \bm{V}_s)\right]\dr s \\
  & + \sum_{n\in \mathbb{N}}\mathbb{E}\left[(\varphi(\bm{\tau}_n, 0, M_n) - \varphi(\bm{\tau}_n, 0, \bm{V}_{\tau_n -}))\bm{1}_{\{\bm{\tau}_n \leq t\}}\right].
 \end{align*}
The local martingale part $\int_0^T \partial_v \varphi(s,\bm{X}_s,\bm{V}_s)\dr B_s$ is indeed a true martingale since $\partial_v\varphi$ is bounded. Let us now define the measures $\nu_- \in \mathcal{M}(\mathbb{R}_+ \times \mathbb{R}_-)$ and $\nu_+\in \mathcal{M}(\mathbb{R}_+^2)$ by
\[
 \nu_-(\dr t, \dr v) = \sum_{n\in\mathbb{N}}\mathbb{P}(\bm{\tau}_n \in \dr t, \bm{V}_{\tau_n -} \in \dr v) \quad \text{and} \quad \nu_+(\dr t, \dr v) = \sum_{n\in\mathbb{N}}\mathbb{P}(\bm{\tau}_n \in \dr t, M_n \in \dr v).
\]
Notice that for any $T > 0$, we have $\nu_-([0,T] \times \mathbb{R}_-) = \nu_+([0,T] \times \mathbb{R}_+) = \mathbb{E}[\sum_{n\in\mathbb{N}}\bm{1}_{\{\bm{\tau}_n \leq T\}}]$, which is finite since $(\bm{\tau}_{n+1} - \bm{\tau}_n)_{n\geq1}$ is an i.i.d. sequence of positive random variables. This also justifies the above exchange between $\mathbb{E}$ and $\sum$. Therefore we have for any $T > 0$
 \begin{align*}
  \mathbb{E}\left[\varphi(T, \bm{X}_T, \bm{V}_T)\right] = & \:\varphi(0, x_0, v_0) + \int_{\mathbb{R}_+^2\times \mathbb{R}}\left[\partial_t\varphi + v\partial_x\varphi + \mathcal{L}\varphi\right]\bm{1}_{\{s\leq T\}}f(\dr s, \dr x, \dr v) \\
  & + \int_{\mathbb{R}_+^2}\varphi(s,0,v)\bm{1}_{\{s\leq T\}}\nu_+(\dr s, \dr v) - \int_{\mathbb{R}_+\times\mathbb{R}_-}\varphi(s,0,v)\bm{1}_{\{s\leq T\}}\nu_-(\dr s, \dr v).
 \end{align*}
 
\noindent
Choosing $T$ large enough so that the support of $\varphi$ is included in $[0,T]\times\mathbb{R}_+\times \mathbb{R}$, we get the desired identity.

\medskip
The relation between the two measures comes from the fact that for every $n\in\mathbb{N}$, $\bm{\tau}_n$ and $M_n$ are independent and that $M_n$ is $\mu$-distributed. Indeed it is clear that, since $(M_n)_{n\in\mathbb{N}}$ is i.i.d. and also independent from the driving Brownian moton, $M_n$ is independent from $(\bm{X}_t, \bm{V}_t)_{0\leq t < \bm{\tau}_n}$ for every $n\in\mathbb{N}$. Hence we have
\[
 \nu_+(\dr t, \dr v) = \mu(\dr v)\sum_{n\in\mathbb{N}}\mathbb{P}(\bm{\tau}_n \in \dr t) = \mu(\dr v)\int_{w \in \mathbb{R}_-}\nu_-(\dr t, \dr w),
\]

\noindent
which achieves the proof.
\end{proof}

\begin{remark}\label{remark_33}
Assume for simplicity that $\mu(\dr v) = \mu(v) \dr v$. The preceding proposition shows that $f$ is a weak solution of
\begin{equation*}
    \left\lbrace
        \begin{aligned}
            \partial_t & f + v\partial_x f = \mathcal{L}^* f  \quad &\text{for} \:\: (t,x,v)\in(0,\infty)^2 \times \mathbb{R} \\
            v &  f(t,0,v) = -\mu(v)\int_{(-\infty, 0)}w f(t,0,w) \dr w & \text{for} \:\: (t,v)\in(0,\infty)^2 \\
            \:\:f&(0,\cdot,\cdot)  = \delta_{(0,v_0)} 
        \end{aligned}
    \right.
\end{equation*}
Informally, it automatically holds that $\nu_+(\dr s, \dr v) = vf(s,0,v)\bm{1}_{\{v>0\}}\dr s \dr v$ and $\nu_-(\dr s, \dr v) = -vf(s,0,v)\bm{1}_{\{v<0\}}\dr s \dr v$.
\end{remark}
For similar notions of weak solutions associated to closely related equations, we refer to Jabir-Profeta \cite[Theorem 4.2.1]{jabir2019stable} and Bernou-Fournier \cite[Definition 4]{bernou_fournier}.

\begin{proof}
We assume that $f(\dr t, \dr x, \dr v) = f(t,x,v)\dr x\dr v \dr t$ with $f$ smooth enough. Let $\varphi$ be a function belonging to $ C^{\infty}_c(\mathbb{R}_+^2\times \mathbb{R})$. We perform some integrations by parts. We have
\[
 \int_{\mathbb{R}_+^2\times \mathbb{R}}f\partial_t\varphi = -\varphi(0, x_0, v_0) - \int_{\mathbb{R}_+^2\times \mathbb{R}}\varphi \partial_tf \quad \text{and} \quad \int_{\mathbb{R}_+^2\times \mathbb{R}}f\mathcal{L}\varphi = \int_{\mathbb{R}_+^2\times \mathbb{R}}\varphi\mathcal{L}^* f.
\]

\noindent
Regarding the integration by part in $x$, we have
\[
 \int_{\mathbb{R}_+^2\times \mathbb{R}}vf\partial_x\varphi = - \int_{\mathbb{R}_+^2\times \mathbb{R}}\varphi v\partial_x f - \int_{\mathbb{R}_+}\int_{\mathbb{R}}v\varphi(s,0,v)f(s,0,v)\dr v \dr s.
\]

\noindent
Inserting the previous identities in \eqref{mmm}, it comes that
\begin{align*}
  \int_{\mathbb{R}_+^2\times \mathbb{R}}\left[\partial_tf + v\partial_xf - \mathcal{L}^*f\right]\varphi + &\int_{\mathbb{R}_+}\int_{\mathbb{R}}v\varphi(s,0,v)f(s,0,v)\dr v \dr s \\
   & = \int_{\mathbb{R}_+^2}\varphi(s,0,v)\nu_+(\dr s, \dr v) - \int_{\mathbb{R}_+\times\mathbb{R}_-}\varphi(s,0,v)\nu_-(\dr s, \dr v).
\end{align*}

\noindent
Since this holds for every $\varphi\in C^{\infty}_c(\mathbb{R}_+\times(0,\infty)\times \mathbb{R})$, we first conclude that $\partial_tf + v\partial_xf =\mathcal{L}^*f$ for $(t,x,v)\in(0,\infty)^2 \times \mathbb{R}$. Then in a second time, we see that
\[
 vf(s,0,v)\dr s \dr v = \nu_+(\dr s, \dr v) - \nu_-(\dr s, \dr v), 
\]

\noindent
i.e. $\nu_+(\dr s, \dr v) = vf(s,0,v)\bm{1}_{\{v>0\}}\dr s \dr v$ and $\nu_-(\dr s, \dr v) = -vf(s,0,v)\bm{1}_{\{v<0\}}\dr s \dr v$. Then, since $\nu_+(\dr s, \dr v) = \mu(\dr v)\int_{w \in \mathbb{R}_-}\nu_-(\dr s, \dr w)$, we conclude that $v f(t,0,v) = -\mu(v)\int_{(-\infty, 0)}w f(t,0,w) \dr w$ for $(t,v)\in(0,\infty)^2$.
\end{proof}

We finally study the limiting fractional diffusion equation. We have the following result.

\begin{proposition}\label{yyy}
Let $(R_t^\alpha)_{t\geq0} = (Z_t^\alpha - \inf_{s\in[0,t]}Z_s^\alpha)_{t\geq0}$ where $(Z_t^\alpha)_{t\geq0}$ is the stable process from Theorem \ref{main_theorem}. Let us denote by $\rho_t(\dr x) = \mathbb{P}(R_t^\alpha \in \dr x)$. The following assertions hold.
\begin{enumerate}[label=(\roman*)]
 \item If $\alpha\in(0,1)$, then for every $\varphi\in C^\infty_c(\mathbb{R}_+)$, we have
 \[
  \int_{\mathbb{R}_+}\varphi(x)\rho_t(\dr x) = \varphi(0) + \int_0^t\int_{\mathbb{R}_+}\mathcal{L}^{\alpha}\varphi(x)\rho_s(\dr x)\dr s,
 \]
 \noindent
 where
 \[
  \mathcal{L}^{\alpha}\varphi(x)=\frac{\sigma_{\alpha}}{2} \int_{\mathbb{R}}\frac{\varphi((x+z)_+) - \varphi(x)}{|z|^{1+\alpha}}\dr z = \frac{\sigma_{\alpha}}{2\alpha}\int_0^{\infty}\frac{\varphi'(y)(y-x)}{|y-x|^{\alpha + 1}}\dr y.
 \]
 
 \item If $\alpha\in[1,2)$, then for every $\varphi\in C^\infty_c(\mathbb{R}_+)$ such that $\varphi'(0) = 0$, we have
 \[
  \int_{\mathbb{R}_+}\varphi(x)\rho_t(\dr x) = \varphi(0) + \int_0^t\int_{\mathbb{R}_+}\mathcal{L}^{\alpha}\varphi(x)\rho_s(\dr x)\dr s,
 \]
 \noindent
 where
 \[
  \mathcal{L}^{\alpha}\varphi(x)=\frac{\sigma_{\alpha}}{2} \mathrm{P.V.}\int_{\mathbb{R}}\frac{\varphi((x+z)_+) - \varphi(x)}{|z|^{1+\alpha}}\dr z = \frac{\sigma_{\alpha}}{2\alpha}\mathrm{P.V.}\int_0^{\infty}\frac{\varphi'(y)(y-x)}{|y-x|^{\alpha + 1}}\dr y,
 \]
 \noindent
 where $\mathrm{P.V.}$ stands for principal values.
\end{enumerate}
\end{proposition}

\begin{proof}
We will denote by $\nu(\dr z) = \frac{\sigma_\alpha}{2}|z|^{-\alpha - 1}\dr z$ the Lévy measure of the symmetric stable process $(Z_t^{\alpha})_{t\geq0}$ and we set $I_t^\alpha = \inf_{s\in[0,t]}Z_s^\alpha$.

\medskip\noindent
\textit{Item (i):} Let us denote by $\Pi(\dr t, \dr z)$ the random Poisson measure on $\mathbb{R}_+\times \mathbb{R}$ with intensity $\dr t \otimes \nu(\dr z)$, associated with $(Z_t^\alpha)_{t\geq0}$. Since $\alpha\in(0,1)$, $(Z_t^\alpha)_{t\geq0}$ has finite variations and we have $Z_t^\alpha =  \int_0^t\int_{\mathbb{R}}z\Pi(\dr s, \dr z)$. The process $(R_t^\alpha)_{t\geq0}$ is also a pure jump Markov process and we check that $\Delta R_t^\alpha = (R_{t-}^\alpha + \Delta Z_t^\alpha)_+ - R_{t-}^\alpha$. If first $\Delta Z_t^\alpha > -R_{t-}^\alpha = -Z_{t-}^\alpha + I_{t-}^\alpha$, then $Z_{t}^\alpha > I_{t-}^\alpha$ so that $I_{t}^\alpha = I_{t-}^\alpha$ and $\Delta R_t^\alpha = \Delta Z_t^\alpha$. If next $\Delta Z_t^\alpha \leq -R_{t-}^\alpha$, then $Z_t^\alpha \leq I_{t-}^\alpha$ so that $I_{t}^\alpha = Z_{t}^\alpha$ and thus $R_{t}^\alpha = 0$ i.e. $\Delta R_{t}^\alpha = -R_{t-}^\alpha$. Therefore, we have $R_t^\alpha =  \int_0^t\int_{\mathbb{R}}[(R_{s-}^\alpha + z)_+ - R_{s-}^\alpha]\Pi(\dr s, \dr z)$. By Itô's formula, we get that for any $\varphi\in C^\infty_c(\mathbb{R}_+)$,
\[
 \mathbb{E}\left[\varphi(R_t^\alpha)\right] = \varphi(0) + \int_0^t\int_{\mathbb{R}}\mathbb{E}\left[\varphi((R_{s}^\alpha + z)_+) - \varphi(R_{s}^\alpha)\right]\nu(\dr x)\dr s,
\]

\noindent
which exactly means $\int_{\mathbb{R}_+}\varphi(x)\rho_t(\dr x) = \varphi(0) + \int_0^t\int_{\mathbb{R}_+}\mathcal{L}^{\alpha}\varphi(x)\rho_s(\dr x)\dr s$, where the operator $\mathcal{L}^\alpha$ is defined as $\mathcal{L}^{\alpha}\varphi(x) = \frac{\sigma_\alpha}{2}\int_{\mathbb{R}}[\varphi((x+z)_+) - \varphi(x)]|z|^{-1-\alpha}\dr z$. It only remains to prove the second identity for $\mathcal{L}^{\alpha}$. We assume that $x > 0$, the proof for $x=0$ being similar. We have
\[
 \frac{2}{\sigma_\alpha}\mathcal{L}^{\alpha}\varphi(x) = \frac{\varphi(0) - \varphi(x)}{\alpha}x^{-\alpha} + \int_{-x}^0\frac{\varphi(x+z) - \varphi(x)}{|z|^{1+\alpha}}\dr z + \int_0^{\infty}\frac{\varphi(x+z) - \varphi(x)}{|z|^{1+\alpha}}\dr z
\]

\noindent
Then, performing carefully two integration by parts, we see that
\[
 \int_{-x}^0\frac{\varphi(x+z) - \varphi(x)}{|z|^{1+\alpha}}\dr z = -\frac{\varphi(0) - \varphi(x)}{\alpha}x^{-\alpha} + \frac{1}{\alpha}\int_{-x}^0\frac{\varphi'(x+z)z}{|z|^{1+\alpha}}\dr z
\]

\noindent
and
\[
 \int_0^{\infty}\frac{\varphi(x+z) - \varphi(x)}{|z|^{1+\alpha}}\dr z = \frac{1}{\alpha}\int_0^{\infty}\frac{\varphi'(x+z)z}{|z|^{1+\alpha}}\dr z.
\]

\noindent
Putting the pieces together, we get
\[
 \frac{2}{\sigma_\alpha}\mathcal{L}^{\alpha}\varphi(x) = \frac{1}{\alpha}\int_{-x}^{\infty}\frac{\varphi'(x+z)z}{|z|^{1+\alpha}}\dr z = \frac{1}{\alpha}\int_{0}^{\infty}\frac{\varphi'(y)(y-x)}{|y-x|^{1+\alpha}}\dr y.
\]

\medskip\noindent
\textit{Item (ii):}  When $\alpha\in[1,2)$, the second expression of $\mathcal{L}^\alpha$ is obtained as in the previous case. We need to remove the small jumps and work with the pure jump Lévy process $(Z_t^{\alpha, \delta})_{t\geq0}$ with Lévy measure $\nu_{\delta}(\dr z) = \bm{1}_{\{|z| > \delta\}}\nu(\dr z)$. Since $\nu$ is symmetric, $Z_t^{\alpha, \delta} \to Z_t^\alpha$ in law as $\delta\to0$ for every $t\geq0$. This implies, see \cite[Chapter VII, Corollary 3.6]{jacod2013limit}, that $(Z_t^{\alpha, \delta})_{t\geq0}$ converges in law to $(Z_t^{\alpha})_{t\geq0}$ in the space of càdlàg functions endowed with the $\bm{\mathrm{J}}_1$-topology as $\delta\to0$. But since the reflection map is continuous with respect to this topology, see \cite[Chapter 13, Theorem 13.5.1]{book_whitt}, the continuous mapping theorem implies that $(R_t^{\alpha, \delta})_{t\geq0} = (Z_t^{\alpha, \delta} - \inf_{s\in[0,t]}Z_s^{\alpha, \delta})_{t\geq0}$ converges weakly to $(R_t^\alpha)_{t\geq0}$ as $\delta\to0$. In particular, if we set $\rho_t^\delta(\dr x) = \mathbb{P}(R_t^{\alpha,\delta}\in \dr x)$, then the probability measure $\rho_t^\delta$ converges weakly to $\rho_t$ as $\delta\to0$.

\medskip
But since $(R_t^{\alpha, \delta})_{t\geq0}$ has finite variations, we can use the very same argument as in the first step to see that for every function $\varphi\in C^\infty_c(\mathbb{R}_+)$, we have 
\[
  \int_{\mathbb{R}_+}\varphi(x)\rho_t^\delta(\dr x) = \varphi(0) + \int_0^t\int_{\mathbb{R}_+}\mathcal{L}^{\alpha, \delta}\varphi(x)\rho_s^\delta(\dr x)\dr s,
\]

\noindent
where $\mathcal{L}^{\alpha, \delta}$ is such that for every $x \geq0$
\[
 \mathcal{L}^{\alpha,\delta}\varphi(x)=\frac{\sigma_{\alpha}}{2} \int_{|z| > \delta}\frac{\varphi((x+z)_+) - \varphi(x)}{|z|^{1+\alpha}}\dr z.
\]

\noindent
First, it is clear that $\int_{\mathbb{R}_+}\varphi(x)\rho_t^\delta(\dr x)$ converges to $\int_{\mathbb{R}_+}\varphi(x)\rho_t(\dr x)$ as $\delta\to0$. We will now conclude by showing that, if $\varphi'(0) = 0$,
\begin{equation}\label{final}
 \int_0^t\int_{\mathbb{R}_+}\mathcal{L}^{\alpha, \delta}\varphi(x)\rho_s^\delta(\dr x)\dr s \longrightarrow \int_0^t\int_{\mathbb{R}_+}\mathcal{L}^{\alpha}\varphi(x)\rho_s
 (\dr x)\dr s \quad \text{as  }\delta\to0.
\end{equation}

\noindent
\textit{Step 1:} To do so, we show in Step 2 that there exists a positive constant $C_{\varphi}$ such that 
\begin{equation}\label{sup_norm}
 ||\mathcal{L}^{\alpha}\varphi - \mathcal{L}^{\alpha,\delta}\varphi||_{\infty} \leq C_{\varphi}\delta^{2-\alpha}.
\end{equation}

\noindent
This shows that $\mathcal{L}^{\alpha}\varphi$ is a continuous and bounded function. Moreover, since
\begin{align*}
 \Big|\int_0^t\int_{\mathbb{R}_+}\mathcal{L}^{\alpha}\varphi(x)\rho_s(\dr x)\dr s - &\int_0^t\int_{\mathbb{R}_+}\mathcal{L}^{\alpha,\delta}\varphi(x)\rho_s^\delta
 (\dr x)\dr s\Big| \\
  & \leq \int_0^t\int_{\mathbb{R}_+}|\mathcal{L}^{\alpha}\varphi(x) - \mathcal{L}^{\alpha,\delta}\varphi(x)|\rho_s^\delta(\dr x)\dr s \\
  & + \Big|\int_0^t\int_{\mathbb{R}_+}\mathcal{L}^{\alpha}\varphi(x)\rho_s^\delta(\dr x)\dr s - \int_0^t\int_{\mathbb{R}_+}\mathcal{L}^{\alpha}\varphi(x)\rho_s
 (\dr x)\dr s\Big|,
\end{align*}

\noindent
it is clear that \eqref{sup_norm} implies \eqref{final}. Indeed the first term clearly goes to $0$ as $\delta\to0$ and the second term converges since for every $s\in[0,t]$, $\int_{\mathbb{R}_+}\mathcal{L}^{\alpha}\varphi(x)\rho_s^\delta(\dr x) \to \int_{\mathbb{R}_+}\mathcal{L}^{\alpha}\varphi(x)\rho_s(\dr x)$ as $\mathcal{L}^{\alpha}\varphi$ is a continuous and bounded function and we can thus apply the dominated convergence theorem.

\medskip\noindent
\textit{Step 2:} We show that \eqref{sup_norm} holds. Let us first explicit a bit more $\mathcal{L}^{\alpha}\varphi(x)$. Let us remark that since $\varphi'(0) = 0$, there is no need for principal values for $\mathcal{L}^{\alpha}\varphi(0) = \int_0^{\infty}(\varphi(z) - \varphi(0))|z|^{-1-\alpha}\dr z$. If $x > 0$, we have for any $\epsilon < x$, $\int_{-x}^{\epsilon}z|z|^{-1-\alpha}\dr z + \int_{\epsilon}^{x}z|z|^{-1-\alpha} = 0$ \begin{equation}\label{eq_generateur}
 \mathcal{L}^{\alpha}\varphi(x) = \int_{\mathbb{R}}\frac{\varphi((x+z)_+) - \varphi(x) - z\varphi'(x)\bm{1}_{\{|z| < x\}}}{|z|^{1+\alpha}}\dr z.
\end{equation}
For the very same reason, the same identity holds for $\mathcal{L}^{\alpha, \delta}\varphi(x)$ and therefore, for any $x > 0$
\begin{align*}
 \mathcal{L}^{\alpha}\varphi(x) - \mathcal{L}^{\alpha,\delta}\varphi(x) = &  \int_{|z| < \delta}\frac{\varphi((x+z)_+) - \varphi(x) - z\varphi'(x)\bm{1}_{\{|z| < x\}}}{|z|^{1+\alpha}}\dr z \\ 
  = & \frac{\varphi(0) - \varphi(x)}{\alpha}(x^{-\alpha} - \delta^{-\alpha}) \bm{1}_{\{\delta > x\}} \\
  & + \int_{-(\delta \wedge x)}^\delta \frac{\varphi(x+z) - \varphi(x) - z\varphi'(x)\bm{1}_{\{|z| < x\}}}{|z|^{1+\alpha}}\dr z.
\end{align*}
Since $\varphi\in C^\infty_c(\mathbb{R}_+)$ and $\varphi'(0) = 0$, it is clear that if we set $D_\varphi = ||\varphi''||_{\infty}$ we have for any $x \geq 0$ and any $z\geq-x$, $|\varphi(x+z) - \varphi(x) - z\varphi'(x)| \leq D_\varphi z^2 / 2$. Since $\varphi'(0) = 0$, we also have for any $x\geq0$, $|\varphi'(x)| \leq D_\varphi |x|$, from which we get that $|\varphi(x+z) - \varphi(x)| \leq D_\varphi |x||z| + D_\varphi z^2$. In any case, we see that for any $x \geq 0$ and any $z\geq-x$
\[
 \left| \varphi(x+z) - \varphi(x) - z\varphi'(x)\bm{1}_{\{|z| < x\}} \right| \leq 2D_\varphi z^2,
\]

\noindent
It comes that 
\[
 \left|\int_{-(\delta \wedge x)}^\delta \frac{\varphi(x+z) - \varphi(x) - z\varphi'(x)\bm{1}_{\{|z| < x\}}}{|z|^{1+\alpha}}\dr z\right| \leq \frac{4D_\varphi}{2-\alpha} \delta^{2-\alpha}.
\]
We also get
\[
 \Big|\frac{\varphi(0) - \varphi(x)}{\alpha}(x^{-\alpha} - \delta^{-\alpha})\Big| \bm{1}_{\{\delta > x\}}\leq \frac{D_\varphi}{\alpha}x^{2-\alpha}\bm{1}_{\{\delta > x\}} \leq \frac{D_\varphi}{\alpha}\delta^{2-\alpha}.
\]
All in all, we showed that for any $x > 0$, $|\mathcal{L}^{\alpha}\varphi(x) - \mathcal{L}^{\alpha,\delta}\varphi(x)| \leq C_\varphi \delta^{2-\alpha}$. When $x = 0$, we have
\[
 \left|\mathcal{L}^{\alpha}\varphi(0) - \mathcal{L}^{\alpha,\delta}\varphi(0)\right| = \left|\int_0^\delta\frac{\varphi(z) - \varphi(0)}{|z|^{1+\alpha}}\dr z\right| \leq \frac{D_\varphi}{2-\alpha}\delta^{2-\alpha} \leq C_\varphi\delta^{2-\alpha}.
\]

\noindent
This shows that \eqref{sup_norm} holds.
\end{proof}

\begin{remark}\label{ggg}
 The above proposition shows that $\rho$ is a weak solution of
 \begin{equation*}
    \left\lbrace
        \begin{aligned}
            &\partial_t \rho_t(x) = \frac{\sigma_\alpha}{2}\int_{\mathbb{R}}\frac{\rho_t(x-z) \bm{1}_{\{x > z\}} - \rho_t(x) + z\partial_x\rho_t(x)\bm{1}_{\{|z|< x\}}}{|z|^{1+\alpha}}\dr z  \;\; &\text{for} \:\: (t,x)\in(0,\infty)^2, \\
            &\int_{0}^\infty \rho_t(x) \dr x = 1 \quad &\text{for} \:\: t\in(0,\infty), \\
            &\:\rho_0(\cdot)  = \delta_{0}. 
        \end{aligned}
    \right.
\end{equation*}
The term $z\partial_x\rho_t(x)\bm{1}_{\{|z|< x\}}$ is useless when $\alpha \in(0,1)$.
\end{remark}
\begin{proof}
 As usual, we do as if $\rho_t$ was sufficiently regular so that all the computations below hold true. By the way, it might actually be the case, see for instance Chaumont-Malecki \cite{chaumont2020entrance}, but this is not our purpose. Consider a function $\varphi \in C^\infty_c((0,\infty))$. Then, by Proposition \ref{yyy}, we have that
 \[
  \frac{\dr }{\dr t} \int_{\mathbb{R}_+}\varphi(x)\rho_t(x) \dr x = \int_{\mathbb{R}_+}\rho_t(x) \mathcal{L}^\alpha\varphi(x) \dr x = \frac{\sigma_\alpha}{2}[I_t(\varphi) - J_t(\varphi)],
 \]
where, recalling \eqref{eq_generateur},
\[
 I_t(\varphi) = \int_{\mathbb{R}_+}\int_{-x}^\infty\rho_t(x)\frac{\varphi(x+z) - \varphi(x) - z\varphi'(x)\bm{1}_{\{|z|< x\}}}{|z|^{1+\alpha}} \dr z \dr x
\]
and, since $\varphi(0) = 0$,
\[
 J_t(\varphi) = \int_{\mathbb{R}_+} \frac{\varphi(x)\rho_t(x)}{\alpha x^{\alpha}}.
\]
We first focus on $I_t(\varphi)$. Exchanging integrals, we get
\[
 I_t(\varphi) = \int_{\mathbb{R}}|z|^{-1-\alpha}\int_{(-z)\vee 0}^\infty\rho_t(x)[\varphi(x+z) - \varphi(x) - z\varphi'(x)\bm{1}_{\{|z|< x\}}]\dr x \dr z.
\]
We next write, for any $z \in\mathbb{R}$,
\begin{align*}
 \int_{(-z)\vee 0}^\infty&\rho_t(x)[\varphi(x+z) - \varphi(x) - z\varphi'(x)\bm{1}_{\{|z|< x\}}]\dr x \\
  & = \int_{z\vee 0}^\infty\rho_t(x-z)\varphi(x) \dr x - \int_{(-z)\vee 0}^\infty\rho_t(x)\varphi(x) \dr x - z\int_{(-z)\vee 0}^\infty\rho_t(x)\varphi'(x)\bm{1}_{\{|z|< x\}} \dr x.
\end{align*}
Regarding the third term, we have $\int_{(-z)\vee 0}\rho_t(x)\varphi'(x)\bm{1}_{\{|z|< x\}} \dr x = \int_{|z|}^\infty\rho_t(x)\varphi'(x) \dr x$, so that, performing an integration by part, we have
\begin{align*}
 z\int_{(-z)\vee 0}\rho_t(x)\varphi'(x)\bm{1}_{\{|z|< x\}} \dr x = z\rho_t(|z|)\varphi(|z|) - z\int_{(-z)\vee 0}\varphi(x)\partial_x\rho_t(x)\bm{1}_{\{|z|< x\}} \dr x.
\end{align*}
All in all, it holds that
\begin{align*}
 I_t(\varphi) = & \int_{\mathbb{R}_+}\varphi(x)\int_{\mathbb{R}}\frac{\rho_t(x-z)\bm{1}_{\{x > z\}} - \rho_t(x)\bm{1}_{\{x > -z\}} + z\partial_x\rho_t(x)\bm{1}_{\{|z|< x\}}}{|z|^{1+\alpha}} \dr z \dr x \\
 & + \int_{\mathbb{R}}\frac{z}{|z|^{1+\alpha}}\rho_t(|z|)\varphi(|z|) \dr z.
\end{align*}
The last term is equal to zero since the integrand is an odd function. Finally, we write
\[
 J_t(\varphi) = \int_{\mathbb{R}_+}\varphi(x)\int_{-\infty}^{-x}\frac{\rho_t(x)}{|z|^{1+\alpha}}\dr z.
\]
Recombinig all the terms, we get that
\[
 \int_{\mathbb{R}_+}\rho_t(x) \mathcal{L}^\alpha\varphi(x) \dr x = \int_{\mathbb{R}_+}\varphi(x) \mathcal{A}\rho_t(x) \dr x,
\]
where $\mathcal{A}\rho_t$ is defined for every $x > 0$ as
\[
 \mathcal{A}\rho_t(x) = \frac{\sigma_{\alpha}}{2}\int_{\mathbb{R}}\frac{\rho_t(x-z)\bm{1}_{\{x > z\}} - \rho_t(x) + z\partial_x\rho_t(x)\bm{1}_{\{|z|< x\}}}{|z|^{1+\alpha}} \dr z.
\]
Therefore, we conclude that for any $\varphi\in C^\infty_c((0,\infty))$, we have
\[
 \int_{\mathbb{R}_+}\varphi(x)[\partial_t \rho_t(x) - \mathcal{A}\rho_t(x)] \dr z = 0,
\]
which is enough to deduce that for any $t>0$ and any $x >0$, $\partial_t\rho_t(x) = \mathcal{A}\rho_t(x)$. 
\end{proof}

\bibliographystyle{plain}
\bibliography{refs.bib}
\end{document}